\documentclass[draft]{amsart}
\usepackage[english]{babel}
\usepackage[utf8]{inputenc}
\usepackage[colorinlistoftodos]{todonotes}
\usepackage{esint}
\usepackage{mathrsfs}
\usepackage{breqn}
\usepackage{graphics}
\usepackage{enumerate}
\usepackage{graphicx}
\usepackage{placeins}
\usepackage{amsmath,amsthm,amsfonts,amssymb,euscript,verbatim,bbm}
\usepackage{hyperref}
\usepackage{amsrefs}
\usepackage[colorinlistoftodos]{todonotes}
\usepackage{xcolor}

\usepackage{color}

\newtheorem{thm}{Theorem}

\newtheorem{prop}[thm]{Proposition}

\newtheorem{lemma}[thm]{Lemma}

\newcommand{\eqdef}{\overset{\mbox{\tiny{def}}}{=}}

\def \gd {\partial_{\gamma}}

\def \xdif {x_{k}(\gamma)\!-\!x_{k}(\gamma\!-\!\eta)}
\def \dxdif {\partial_{\gamma}x_{k}(\gamma)\!-\!\partial_{\gamma}x_{k}(\gamma\!-\!\eta)}
\def \pgx {\partial_{\gamma}x_{k}}

\def \supF {\|F(x_{k})\|_{L^{\infty}}}

\def \R {\mathbb{R}}
\def \T {\mathbb{T}}
\def \xigamma {x_{k}(\gamma)\!-\!x_{k}(\gamma\!-\!\xi)}

\def \dxigammaeta {\partial_{\gamma}x_{k}(\gamma\!-\!\eta)\!-\!\partial_{\gamma}x_{k}(\gamma\!-\!\eta\!-\!\xi)}

\def \zdif {x_{k}(\gamma)\!-\!\bar{x}_{k}(\gamma\!-\!\eta)}
\def \dzdif {\partial_{\gamma}x_{k}(\gamma)\!-\!\partial_{\gamma}\bar{x}_{k}(\gamma\!-\!\eta)}
\def \intT {\int_{\mathbb{T}}}

\def \zxigamma {x_{k}(\gamma)\!-\!\bar{x}_{k}(\gamma\!-\!\xi)}
\def \dzxigammaeta {\partial_{\gamma}x_{k}(\gamma\!-\!\eta)\!-\!\partial_{\gamma}\bar{x}_{k}(\gamma\!-\!\eta\!-\!\xi)}

\def \e {\eta}
\def \g {\gamma}
\def \al {\alpha}
\def \xd {x(\gamma)\!-\!x(\gamma\!-\!\eta)}
\def \xdx {x(\gamma)\!-\!x(\gamma\!-\!\xi)}
\def \pg {\partial_{\gamma}}
\def \xg {x(\gamma)}
\def \xe {x(\gamma\!-\!\eta)}
\def \xdp {\partial_{\gamma}x(\gamma)\!-\!\partial_{\gamma}x(\gamma\!-\!\eta)}

\def \xdpxe {\partial_{\gamma}x(\gamma\!-\!\xi)\!-\!\partial_{\gamma}x(\gamma\!-\!\eta\!-\!\xi)}
\def \xdpex {\partial_{\gamma}x(\gamma\!-\!\eta)\!-\!\partial_{\gamma}x(\gamma\!-\!\eta\!-\!\xi)}

\newcommand{\ba}{\begin{equation}}
\newcommand{\ea}{\end{equation}}

\newcommand{\bea}{\begin{eqnarray}}
\newcommand{\eea}{\end{eqnarray}}

\def\beaa{\begin{eqnarray*}}
	\def\eeaa{\end{eqnarray*}}
\begin{document}

	\title[On the local existence and blow-up for generalized SQG patches]{On the local existence and blow-up for generalized SQG patches}
	
	\author{Francisco Gancedo}
	
	\author{Neel Patel}

	\date{\today; %\Red{(DRAFT)}
	}

	\begin{abstract}
		We study patch solutions of a family of transport equations given by a parameter $\alpha$, $0< \alpha <2$, with the cases $\alpha =0$ and $\alpha =1$ corresponding to the Euler and the surface quasi-geostrophic equations respectively. In this paper, using several new cancellations, we provide the following new results. First, we prove local well-posedness for $H^{2}$ patches in the half-space setting for $0<\alpha< 1/3$, allowing self-intersection with the fixed boundary. Furthermore, we are able to extend the range of $\alpha$ for which finite time singularities have been shown in \cite{KYZ} and \cite{KRYZ}. Second, we establish that patches remain regular for $0<\alpha<2$ as long as the arc-chord condition and the regularity of order $C^{1+\delta}$ for $\delta>\al/2$ are time integrable. This finite-time singularity criterion holds for lower regularity than the regularity shown in numerical simulations in \cite{CFMR} and \cite{ScottDritschel} for surface quasi-geostrophic patches, where the curvature of the contour blows up numerically. This is the first proof of a finite-time singularity criterion lower than or equal to the regularity in the numerics. Finally, we also improve results in \cite{G} and in \cite{CCCGW}, giving local-wellposedness for patches in $H^{2}$ for $0<\alpha < 1$ and in $H^3$ for $1<\alpha<2$. 
	\end{abstract}
	
	\setcounter{tocdepth}{1}
	%\chapter is level 0
	%\section is level 1
	%\subsection is level 2
	%\subsubsection is level 3
	%\paragraph is level 4
	%\subparagraph is level 5
	
	\maketitle
	\tableofcontents

	\section{Introduction}
	
	In this paper we study a family of two dimensional transport equations 
	\begin{align}\label{te}
	\begin{split}
	\partial_t\theta(x,t)+u(x,t)\cdot\nabla\theta(x,t)=0,\quad x\in\R^2,\, t\in [0,\infty),\\
	\end{split}
	\end{align}
	where the function $\theta$ and $u$ are related by means of an $\al$-parameter equation:
	\begin{equation}\label{iv}
	u(x,t)=\nabla^{\bot}I_{2-\al}\theta(x,t),\quad 0<\al<2.
	\end{equation} 
	Above, the perpendicular gradient is given by $\nabla^{\bot}=(-\partial_{x_2},\partial_{x_1})$ and $I_{\beta}$, $0<\beta<2$, is the Riesz potential whose Fourier symbol is $|\xi|^{-\beta}=\widehat{I_{\beta}}$. We are then dealing with an active scalar $\theta(x,t)$ moving by an incompressible flow which becomes more singular as $\al$ increases. 
	
	The limiting case $\al=0$ corresponds to the 2D Euler equation, with $\theta$ representing the vorticity of an ideal fluid. For this classical PDE, there exists global existence of solutions starting from regular data \cite{bertozzi-Majda}, although its dynamics are not fully well understood. Recently, certain solutions have been shown to rotate uniformly \cite{CCG-S3}, while for others, there exists exponential growth of vorticity gradient \cite{Z} in the periodic setting and even double exponential growth of vorticity gradient in a bounded domain \cite{KS}.

	The midpoint case $\al=1$ corresponds to the surface quasi-geostrophic equation (SQG) with  $\theta$ representing the fluid temperature. This system comes from geophysical science and provides particular solutions of highly rotating 3D oceanic or atmospheric fluids \cite{P}. Specifically, SQG has been used to understand the formation of sharp fronts of temperature. From the mathematical point of view, this equation was introduced in \cite{CMT} as a 2D model of 3D Euler; they share geometrical properties and both present vortex-stretching effects. Global regularity for general smooth initial data is an open problem. Hyperbolic blow-up was ruled out in \cite{C} and squirt collision of level sets was ruled out in \cite{CF}. From the numerical point of view, some blow-up scenarios from \cite{CMT} were discarded in \cite{CCCGW}, and new ones were shown in \cite{Scott} with solutions developing formation of filaments. On the other hand, non-trivial global rotating solutions have been recently found using computer-assisted proofs \cite{CCG-S2}.

	Dealing with a more singular transport equation than 2D Euler such as (\ref{te})-(\ref{iv}), the global-in-time existence of regular solutions is an open question \cite{CCW}. Considering (\ref{te})-(\ref{iv}) for $1<\al<2$, the velocity is more singular than $\theta(x,t)$, the scalar convected. However, there still exist solutions locally in time starting from regular data \cite{CCCGW}. Finite-time blow-up is not known. Considering the whole range of $\alpha$, the system (\ref{te})-(\ref{iv}) is called the generalized SQG.
	
	In this paper we focus on patch solutions, where the active scalar $\theta(x,t)$ is given by
	\begin{equation}\label{patch}
	\theta(x,t)=\left\{\begin{array}{cl}
	\theta_0,& x\in D(t),\\
	0 ,& x\in \R^2\smallsetminus D(t),
	\end{array}\right.
	\end{equation}
	with $\theta_{0}$ a constant value different than zero and the set $D(t)$ a simply connected bounded domain with regular boundary $\partial D(t)$. The transport equation \eqref{te} preserves the structure \eqref{patch} of the scalar convected, which is understood as a weak solution of (\ref{te})-(\ref{iv}). The incompressibility keeps the area of $D(t)$ constant. Hence, the system (\ref{te})-(\ref{iv})-(\ref{patch}) becomes a contour evolution problem where the important question is to understand the dynamics of $\partial D(t)$. For 2D Euler, these solutions are vortex-patches. For SQG sharp fronts and for $0<\al<2$, we call them $\al$-patches. See \cite{IP},\cite{DM},\cite{GG-J},\cite{G2} and references therein for the study of patches evolving by other fundamental fluid mechanics PDEs such as Euler, Navier-Stokes, Boussinesq, Darcy's law, etc.
	
	This type of solution has been highly studied for the 2D Euler equation in the so-called vortex-patch problem. In this particular setting, weak solutions exist for all time and are unique \cite{Yudovich}. Although finite time blow-up was conjectured in the 1980's, persistence of $C^{1+\gamma}$ regularity, $0<\gamma<1$, of the evolving boundary patches was first proven in \cite{Ch} using striated regularity and paradifferential calculus. It was also later proven in \cite{bertozzi-Constantin} by a different geometrical approach using harmonic analysis. Non-trivial global-in-time rotating solutions exist \cite{B} and they have been proven to be $C^\infty$ \cite{HMV} and later to have analytic regularity \cite{CCG-S3}. From these approaches, different geometries have been studied such us \cite{HM} considering two patches \cite{HMV},\cite{HHMV}, fixed-boundary effects \cite{HHHM} and non-constant vorticity \cite{GHS}.
	
	The study of patch-type solutions for the SQG equation started much later. Weak solutions of the system exist for all time \cite{Resnick} and although in a more general setting they are not unique \cite{BSV}, patches certainly are \cite{CCG}. A big difference with the vortex-patch problem is that in SQG, the temperature given by \eqref{patch} provides divergent velocities at the boundary of the patches due to equation \eqref{iv}. This case provides $u\notin L^\infty$, but $u\in BMO$. Then, finding the evolution equation for the free boundary is a difficult starting point in the analysis. However, the normal direction of the velocity is well-defined and it is possible to find a contour dynamics evolution equation \cite{Resnick}. The first analytical result was shown in \cite{Rodrigo}, where local-existence for $C^\infty$ patches through a Nash-Moser implicit function theorem was proven.
	
	In \cite{CFMR}, $\al$-patches for $0<\al<1$ were introduced for the first time in the mathematics literature and showed how to extend the local-existence argument for $C^\infty$ free boundaries in \cite{Rodrigo} for system (\ref{te}-\ref{iv}-\ref{patch}). The paper also provides numerical evidence of curvature blow-up at the same instant of time as when two different particles of patches collide at the same spatial point. Using cancellation from the curve-structure of the $\al$-patches and SQG sharp front systems, local-existence in Sobolev spaces $H^k$ for $k\geq 3$ was given in \cite{G}. In particular, later results provide a justification of the numerical simulations in \cite{CFMR}. One can also see \cite{GS} where it is proven that curvature control prevents point-wise collapse. Recently, new numerical evidences of finite-time blow-up have been shown in \cite{ScottDritschel} with a self-similar cascade of filament scenario, developing also curvature blow-up. On the other hand, global-in-time nontrivial rotating solutions have been found for the $\al$-patches in \cite{HH} for $0<\al<1$ and later for SQG sharp fronts \cite{CCG-S}. See \cite{HM2} for global-in-time dynamics of rotating pairs and \cite{G-S} for existence of non-trivial stationary solutions.
	
	It is also possible to consider patch-type solutions for more singular scenario such as the generalized SQG equation (\ref{te}-\ref{iv}-\ref{patch}) with $1<\al<2$. Local-existence of $H^k$ patches for $k\geq 4$, together with global existence of general weak solutions, was given in \cite{CCCGW}. Recently, global regularity has been proven for near planar patches in the whole space \cite{CG-SI} using the dispersive properties of the contour evolution equations. See \cite{CCG-S} for the global-existence of rotating nontrivial solutions.
	
	All these patch problems have also been studied by different approaches. The links between regular solutions and sharp fronts via limit procedure was considered in \cite{FR},\cite{FR2} and references therein. A new cubic nonlinear one-dimensional approximation is shown in \cite{HS}. The equations are locally well-posed \cite{HSZ} and small initial data solutions are globally well-posed for a SQG-model case \cite{HSZ2}.
	
	In \cite{KYZ},\cite{KRYZ} a new scenario is studied for the $\al$-patch model. The papers consider (\ref{te}-\ref{iv}) in the upper-half plane with a non-slip boundary condition:
	\begin{equation}\label{fixedboundary}
	u(x,t)=(u^{(1)}(x,t),u^{(2)}(x,t)), \mbox{  such that }  u^{(2)}(x^{(1)},0,t)=0.
	\end{equation}
	The authors prove local-existence for $0<\al<1/12$ in a more singular setting allowing $H^3$ patches to touch the fixed boundary along a segment. Temperature is given by several patches in the upper half plane and is considered as follows
	\bea\label{severalpatches}
	\theta(x,t) =  \sum_{j=1}^{n} \theta_{j} \chi_{j} (x,t),\quad  \theta_j\in\R,\quad \chi_j(x,t)=\left\{\begin{array}{cl}
		1,& x\in D_{j}(t),\\
		0 ,& x\in \R^2\smallsetminus D_{j}(t),
	\end{array}\right.
	\eea
	with $D_{j}$ disjoint simply connected bounded domains with regular boundary $\partial D_{j}$ on $\R\times[0,+\infty)$. Uniqueness of this type of weak solution is also given in the frame of weak solutions for $0<\al<1/2$ and in the whole $\R^2$ for $0<\al<1$. Then, for two regular patches of opposite temperature, singularity formations are found by assuming global-in-time existence as the two free boundaries approach to each other and collide in finite time. The fact that initially the two patches are on the fixed boundary along a segment allows one to control its dynamics in this scenario.

	%such that $\theta(x,t)$ satisfies \eqref{te} and \eqref{iv} on $\mathbb{R}\times \mathbb{R}_{+}$.
	%A solution to \eqref{te} and \eqref{iv} of the form \eqref{thetasystem} and \eqref{Xij} is called a patch-like solution. 
	%Furthermore, in \cite{KYZ}, this patch-like solution is said to be in $C^{k,\lambda}$ or $H^{s}$ if there exists constant speed parametrizations $\{x_{j}\}_{j=1,\ldots, n}$ of the boundaries $\{\partial D_{j}\}_{j=1,\ldots, n}$ such that $x_{j}$ is in $C^{k,\lambda}$ or $H^{s}$ respectively. Here constant speed parametrization means that $|\gd x_{j}|$ is a constant given by $\frac{|\partial D_{j}|}{2\pi}$, where $|\partial D_{j}|$ is the length of the boundary.
	
	%They prove the following local existence and uniqueness result for patch-like solutions in the range $0< \alpha < 1/12$:
	%\begin{thm}[Kiselev-Yao-Zlatos, \cite{KYZ}]
	%	
	%	
	%Suppose $\alpha \in (0, 1/12)$ and $$\theta_{0}(x)=\sum_{j=1}^{n} \theta_{j,0}(x) \chi_{j,0} (x)$$ is a patch-like initial data in $H^{3}$. Then, for some $T>0$, there exists a unique solution of the form \eqref{thetasystem} in $H^{3}$ such that $\theta(x,0) = \theta_{0}(x)$.
	%\end{thm}

	In this paper we show the following theorem for the modified SQG system, building on the well-posedness result described in \cite{KYZ} and increasing the range of $\alpha$ for local existence and uniqueness. 
	
	%Here, we will say that the patch-like solution is in $C^{k,\lambda}$ or $H^{s}$ if there exists a parametrization $\{x_{j}\}$ of $D_{j}$ for $j= 1, \ldots, n$ such that each $x_{j}$ is in $C^{k,\lambda}$ or $H^{s}$ respectively.
	
	\begin{thm}\label{localmodifiedSQG}
		Suppose $D_j(0)\subset \R\times[0,+\infty)$, $j=1,...,n$, are bounded domains with non self-intersect $H^2$ boundaries and $D_j(0)\cap D_k(0)=\phi$ for $k\neq j$. Then there exists a time $T>0$ so that there is a unique solution of (\ref{te}-\ref{iv}-\ref{fixedboundary}-\ref{severalpatches}) for $0<\al<1/3$ with $\partial D_j(t)\in C([0,T],H^2)$ non self-intersecting, $D_j(t)\cap D_k(t)=\phi$ for $k\neq j$ and $\theta(x,0) =\sum_{j=1}^{n} \theta_{j} \chi_{j} (x,0)$.
	\end{thm}
	
	The importance of this theorem is given by the fact that it increases the range of $\al$ for which there exists finite time blow-up. This is done in the same singular scenario explained above, where the patches touch the free boundary along a segment. The lower regularity of the theorem for the patches allows us to consider higher $\al$ because the singular part of the equation due to the fixed boundary has less of a singular effect under a lower order norm. On the other hand, the non-locality of the system makes the theorem more complicated than for higher regularity as the nonlinear terms are more difficult to handle. For example, this trade-off situation is well understood in the dynamics of the very important arc-chord condition. For a domain $D(t)$ such that
	$$
	\partial D(t)=\{x(\gamma,t):\,\gamma\in[-\pi,\pi]=\T\}\subset \R\times[0,+\infty),
	$$
	we say that the arc-chord condition is satisfied if the function $F(x)$ defined below
	\begin{equation*}
	F(x)(\g,\e,t)=\frac{|\e|}{|x(\g,t)-x(\g-\e,t)|},\, \e,\g\in\T,\quad F(x)(\g,0,t)=|\pg x(\g,t)|^{-1},
	\end{equation*}
	is in $L^\infty(\T^2)$. Then, in order for the nonlocal system of contour evolution equation to make sense, $F(x)$ has to be controlled. Consider its evolution:
	$$
	\partial_t F(x)(\g,\e,t)=-\frac{|\e|(x(\g,t)-x(\g-\e,t))\cdot(\partial_t x(\g,t)- \partial_t x(\g-\e,t))}{|x(\g,t)-x(\g-\e,t)|^3}.
	$$
	Essentially, we have that
	$$
	\partial_t F(x)\leq F(x)^2\|\pg\partial_t x\|_{L^\infty},
	$$
	and consequently, the quantity $\|\pg\partial_t x\|_{L^\infty}$ needs to be estimated. For a single patch, the contour equation can be given by
	$$
	\partial_tx(\g,t)=\intT \frac{-\partial_{\gamma}x(\gamma-\eta,t)}{|x(\gamma,t)-x(\gamma-\eta,t)|^{\alpha}}d\eta+
	\intT \frac{-\partial_{\gamma}\bar{x}(\gamma-\eta,t)}{|x(\gamma,t)-\bar{x}(\gamma-\eta,t)|^{\alpha}}  d\eta
	$$
	where $\bar{x}(\gamma,t)=(x^{(1)}(\gamma,t),-x^{(2)}(\gamma,t))$, showing that $\|\pg x_t\|_{L^\infty}$ can be estimated for $0<\al<1/3$. But, on the other hand, there is lack of symmetry in the last nonlinear term, and hence, local-existence is not possible for $H^2$ curves. It is possible to symmetrize the contour equation with tangential terms, as they do not change the shape of the patch \cite{G}:         
	\begin{equation}\label{contourS}
	\partial_tx(\g,t)=\intT \frac{\partial_{\gamma}x(\gamma,t)-\partial_{\gamma}x(\gamma-\eta,t)}{|x(\gamma,t)-x(\gamma-\eta,t)|^{\alpha}}d\eta+
	\intT \frac{\partial_{\gamma}x(\gamma,t)-\partial_{\gamma}\bar{x}(\gamma-\eta,t)}{|x(\gamma,t)-\bar{x}(\gamma-\eta,t)|^{\alpha}}  d\eta,
	\end{equation} 
	Under this symmetrization, it can be seen that in above contour evolution equation, $\pg\partial_t x\sim \pg^2x$ is only in $L^2$ and not in $L^\infty$. In this paper, we show that with low regularity such as $x\in H^2$, it is possible to handle the above issue in the local-existence argument due to extra cancellation found. The uniqueness result needs a different trick, as a change of variable in the contour equation is needed in order to deal with the more singular term involving the difference among two solutions.
	
	Having establish local existence and uniqueness of solutions, we construct solutions that exhibit singularities in finite time. In \cite{KRYZ}, for the parameter range $0< \alpha < 1/12$, the authors construct initial data of two patches that are odd symmetric with respect to the vertical axis such that the patches become singular in finite time. This behavior is shown by demonstrating that a trapezoid moving towards the origin remains in the patch until the patch (and the trapezoid) reaches the origin. This holds by assuming regularity and then getting a contradiction. Following the construction of \cite{KRYZ}, it is not possible to extend the finite time singularity result to the new range $0 < \alpha < 1/3$ of our local existence result Theorem \ref{localmodifiedSQG}. In this paper, by considering odd symmetric patches that contain a trapezoid of greater slope than in \cite{KRYZ}, we can extend this finite time singularity result to the range $0<\alpha < 1/3$. Precisely, we compute the patch velocities by considering a trapezoidal subdomain of arbitrary slope and width. We can produce sharper estimates of the horizontal and vertical velocities by decomposing the domain of integration into subdomains that allow for more direct computation of the integrands. For example, in the horizontal velocity bound, we integrate over a larger subdomain of the trapezoidal region than in \cite{KRYZ}. One additional subdomain considered yields an integrand that is directly computable and another subdomain requires an approximation. Our estimates let us choose the appropriate slope that gives the correct direction of the velocity and let us choose any arbitrary width of the trapezoid to fit it within the patch. The result is as follows:
	\begin{thm}\label{FiniteTimeSingularity}
		Suppose $0 < \alpha < 1/3$. There exist non self-intersecting $H^{2}$ initial patch data for (\ref{te}-\ref{iv}-\ref{fixedboundary}-\ref{severalpatches}) that develop a singularity in finite time.
	\end{thm}
	
	Next, we give the following sequence of theorems regarding modified SQG patches in the full space $\mathbb{R}^{2}$. These statements are regarding existence and uniqueness of solutions for $0<\alpha<2$ and decrease the regularity of theorems in \cite{G} and \cite{CCCGW}. The theorems below also provide a blow-up criterion for regularity of order $C^{1+\delta}$ with $\delta>\al/2$. The regularity is expressed in terms of $L^p$ norms of two derivatives through Sobolev embedding and involve the important arc-chord condition. This new blow-up criterion is at a level of regularity lower than the one shown numerically in \cite{CFMR} and \cite{ScottDritschel}. Prior to this result, there was no rigorous proof of a blow-up criterion at the regularity level of these numerical results, with the best rigorously proven criteria being at the level of $C^{2+\delta}$ \cite{G}.
	
	\begin{thm}[\bfseries{Local existence for $0<\al<1$ in $H^2$}]\label{LE0a1H2}
		Suppose $x_0(\g)\in H^2(\T)$ with $F(x_0)\in L^\infty(\T^2)$. Then there exists a time $T>0$ such that there is a unique solution of (\ref{te}-\ref{iv}-\ref{patch}) for $0<\al<1$ with $\partial D(t)\in C([0,T],H^2)$, $F(x)\in L^\infty([0,T]\times\T^2)$ and $\partial D(0)=\{x_0(\g):\,\g\in\T\}$.
	\end{thm}
	
	\begin{thm}[\bfseries{Local existence for $1<\al<2$ in $H^3$}]\label{LE1a2H3}
		Suppose $x_0(\g)\in H^3(\T)$ with $F(x_0)\in L^\infty(\T^2)$. Then there exists a time $T>0$ such that there is a unique solution of (\ref{te}-\ref{iv}-\ref{patch}) for $1<\al<2$ with $\partial D(t)\in C([0,T],H^3)$, $F(x)\in L^\infty([0,T]\times\T^2)$ and $\partial D(0)=\{x_0(\g):\,\g\in\T\}$.
	\end{thm}
	
	\begin{thm}[\bfseries{Regularity criterion}]\label{RC}
		Consider $\partial D(t)$ a solution of (\ref{te}-\ref{iv}-\ref{patch}) with
		$\partial D(0)=\{x_0(\g):\,\g\in\T\}$, $x_0(\g)\in H^n(\T)$, $n\geq 2$ for $0<\al<1$, $n\geq 3$ for $1\leq\al<2$, and $F(x_0)\in L^\infty(\T^2)$. Assume that for $p>(1-\alpha/2)^{-1}$ and $T>0$ the following holds
		\begin{equation}\label{RCine}
		\int_0^T(\|\pg^2 x\|_{L^p}(s)+
		\|F(x)\|_{L^\infty}(s))\|\pg^2 x\|_{L^p}(s)\|F(x)\|_{L^\infty}^{2+\al}(s)ds<\infty.
		\end{equation} 
		Then $\partial D(t)\in C([0,T],H^n(\T))$ and $F(x)\in L^\infty([0,T]\times\T^2)$.
	\end{thm}
	
	The main nonlinear term in this case is given by the first one in \eqref{contourS}. Hence, doing energy estimates, symmetrization and integration by parts provide that one of the most singular characters to control is given by
	$$
	S=\frac\al4\intT\!\intT |\partial^k_{\gamma}x(\gamma)\!-\!\partial^k_{\gamma}x(\gamma\!-\!\eta)|^2K(\g,\e)d\eta d\g,
	$$ 
	with $K$ being the following kernel 
	$$
	K(\g,\e)=\frac{(x(\gamma)\!-\!x(\gamma\!-\!\eta))\cdot (\pg x(\gamma)\!-\!\pg x(\gamma\!-\!\eta))}{|x(\gamma)-x(\gamma-\eta)|^{\alpha+2}}.
	$$
	Above we suppress the time dependence for simplicity and consider $k\geq 2$. In the case $0<\al<1$, the kernel inside the integral above can be bound as follows 
	$$|K(\g,\e)|\leq \|F(x)\|_{L^\infty}^{\al+1}|\pg x|_{C^\delta}|\e|^{-\al-1+\delta},$$
	where the seminorm $|\cdot|_{C^\delta}$ is given by $\sup_{\g\neq\e}|f(\g)\!-\!f(\e)||\g\!-\!\e|^{-\delta}$. Then, for $\al<\delta$, it yields
	the following  control for $S$: $$S\leq \|F(x)\|_{L^\infty}^{\al+1}|\pg x|_{C^\delta}\|\pg^k x\|_{L^2}^2.$$
	Different approaches for different terms allow similar bounds which, using Sobolev embedding, provide local-existence for $H^k$ with $k\geq 2$ in the case $0<\alpha<1/2$ and local-existence for $H^k$ with $k\geq 3$ in the case $1/2 \leq \alpha <1$. A regularity criteria also follows in terms of $C^{1+\delta}$ regularity for $\al<\delta$. In the case $1\leq\al<2$, the same kernel needs to be rewritten as
	$$
	K(\g,\e)=\frac{(x(\gamma)\!-\!x(\gamma\!-\!\eta))\cdot(\pg x(\gamma)\!-\!\pg x(\gamma\!-\!\eta))-\pg x(\gamma)\cdot\pg^2 x(\gamma)\e^2}{|x(\gamma)-x(\gamma-\eta)|^{\alpha+2}}
	$$
	for a time independent tangent vector length reparameterization (see \cite{G} for more details).  
	The above expression yields $$|K(\g,\e)|\leq \|F(x)\|_{L^\infty}^{\al+1}(|\pg^2 x|_{C^\delta}+\|F(x)\|_{L^\infty}\|\pg^2 x\|_{L^\infty}|\pg x|_{C^\delta})|\e|^{-\al+\delta},$$
	yielding the following control for $S$: $$S\leq \|F(x)\|_{L^\infty}^{\al+1}(|\pg^2 x|_{C^\delta}+\|F(x)\|_{L^\infty}\|\pg^2 x\|_{L^\infty}|\pg x|_{C^\delta})\|\pg^k x\|_{L^2}^2,$$ given that $\al-1<\delta$. The same intuition shows local-existence for $H^k$ with $k\geq 4$ in the case $3/2\leq \alpha<2$ and a blow-up criterion involving $C^{2+\delta}$ regularity for $\al-1<\delta$. Below, we show new extra cancellations which overcome the difficulties explained above. In particular, the blow-up criterion is mostly in terms of $|\pg x|_{C^\delta}$, $\delta>\al/2$, but we use Sobolev embedding to put all the terms at the same level so that the norm $\|\pg^2 x\|_{L^p}\sim |\pg x|_{C^\delta}$ is used for $1-\delta=p^{-1}$. Therefore, this gives the relation $p>(1-\alpha/2)^{-1}$.
	
	The structure of the paper is as follows: In Section \ref{contoureq}, we set up the contour equation for the boundaries of each patch $D_{j}$ under the condition that the parametrization is solely dependent on time, i.e. $|\gd x(\gamma,t)|^{2} = A(t)$. In Section 3 we prove the appropriate a priori estimates to get the local existence result in Theorem \ref{localmodifiedSQG}, for the contours $x_{j}$ for each $j\in \{1,\ldots, n\}$. In doing so, we realize that the aforementioned choice of parametrization will allow us to perform the estimates required to control the arc chord term
	\bea\label{arcchord}F(x_{j})(\gamma,\eta,t) = \frac{|\eta|}{|x_{j}(\gamma,t)-x_{j}(\gamma-\eta,t)|}
	\eea
	for each $j\in 1,\ldots, n$. In Section 4, we prove the velocity estimates for the finite time singularity construction to prove Theorem \ref{FiniteTimeSingularity}. Finally in Section 5 we prove Theorems \ref{LE0a1H2}, \ref{LE1a2H3} and \ref{RC} through the use of the new cancellations.

	\section{Contour Equation with Constant Parametrization}\label{contoureq}
	
	In this section, we derive the contour equation under constant parametrization. Consider $n$ patches and $0< \alpha < 1$ each with constant parametrization given by $x_{k}(\gamma)$ for $\gamma \in [0,2\pi]$ and $ k = 1, \ldots, n$. The contour equation for the SQG equation is given by
	\begin{multline}\label{SQGnpatch}
	\partial_{t}x_{k}(\gamma,t) = \sum_{j=1}^{n} \frac{\theta_{j}}{2\alpha}\intT \frac{\partial_{\gamma}x_{k}(\gamma,t)-\partial_{\gamma}x_{j}(\gamma-\eta,t)}{|x_{k}(\gamma,t)-x_{j}(\gamma-\eta,t)|^{\alpha}}  d\eta \\+\frac{\theta_{j}}{2\alpha}\intT \frac{\partial_{\gamma}x_{k}(\gamma,t)-\partial_{\gamma}\bar{x}_{j}(\gamma-\eta,t)}{|x_{k}(\gamma,t)-\bar{x}_{j}(\gamma-\eta,t)|^{\alpha}}  d\eta
	\end{multline}
	where
	$\theta_{j}$ is the magnitude of the temperature patch inside the contour $x_{j}$ (see \cite{KYZ} for more details). For the purposes of our calculations, we will change the parametrization of the contour to depend only on time, i.e. $|\partial_{\gamma}x_{j}|^{2} = A_{j}(t)$ for each $j \in \{1,\ldots , n\}$. By the chain rule, changing the parametrization of the contour equation gives us the equation
	
	\begin{equation}\label{SQG}
	\partial_{t}x_{k}(\gamma,t) = NL_{k}(\gamma,t) + \lambda_k(\gamma,t)\partial_{\gamma}x_{k}(\gamma,t)
	\end{equation}
	where
	\begin{multline}\label{NL}
	NL_k(\gamma,t) = \sum_{j=1}^{n} \frac{\theta_{j}}{2\alpha}\intT \frac{\partial_{\gamma}x_{k}(\gamma,t)-\partial_{\gamma}x_{j}(\gamma-\eta,t)}{|x_{k}(\gamma,t)-x_{j}(\gamma-\eta,t)|^{\alpha}}  d\eta \\+\frac{\theta_{j}}{2\alpha}\intT \frac{\partial_{\gamma}x_{k}(\gamma,t)-\partial_{\gamma}\bar{x}_{j}(\gamma-\eta,t)}{|x_{k}(\gamma,t)-\bar{x}_{j}(\gamma-\eta,t)|^{\alpha}}  d\eta.
	\end{multline}
	To solve for $c(\gamma,t)$, we differentiate both sides of \ref{SQG} in the $\gamma$ variable:
	$$ \partial_{\gamma}\partial_{t}x_{k}(\gamma,t) = \partial_{\gamma} NL_{k}(\gamma,t) + \lambda_k(\gamma,t)\partial_{\gamma}^{2}x_{k}(\gamma,t) + \partial_{\gamma}\lambda_k(\gamma,t)\partial_{\gamma}x_{k}(\gamma,t).$$
	Taking an inner product of both sides with the tangential derivative $\partial_{\gamma}x(\gamma)$, we obtain
	\begin{align*}
	\partial_{\gamma}x_{k}(\gamma)\cdot\partial_{\gamma}\partial_{t}x_{k}(\gamma) &= \partial_{\gamma}x_{k}(\gamma)\cdot\partial_{\gamma} NL_{k}(\gamma) + \partial_{\gamma}x_{k}(\gamma)\cdot \lambda_k(\gamma)\partial_{\gamma}^{2}x_{k}(\gamma)\\
	&\quad + \partial_{\gamma}x_{k}(\gamma)\cdot\partial_{\gamma}\lambda_k(\gamma)\partial_{\gamma}x_{k}(\gamma)\\
	& = \partial_{\gamma}x_{k}(\gamma)\cdot\partial_{\gamma} NL_{k}(\gamma) +  \partial_{\gamma}\lambda_k(\gamma)|\partial_{\gamma}x_{k}(\gamma)|^{2}
	\end{align*}
	since by our choice of parametrization,
	$$\partial_{\gamma}x_{k}(\gamma)\cdot \lambda_k(\gamma)\partial_{\gamma}^{2}x_{k}(\gamma) = \lambda_k(\gamma) \partial_{\gamma}(|\partial_{\gamma}x_{k}(\gamma)|^{2}) =\lambda_k(\gamma) \partial_{\gamma}A(t) = 0.$$
	Now, integrating, we obtain that
	\begin{align*}
	\int_{-\pi}^{\pi} \partial_{\gamma}x_{k}(\gamma)\cdot\partial_{\gamma}\partial_{t}x_{k}(\gamma) d\gamma &= \frac{1}{2} \int_{-\pi}^{\pi}\partial_{t}(|\partial_{\gamma}x_{k}(\gamma)|^{2}) d\gamma=\frac{1}{2}\int_{-\pi}^{\pi} \partial_{t}(A(t)) d\gamma\\&=\pi \partial_{t}(A(t))=2\pi\partial_{\gamma}x_{k}(\gamma)\cdot\partial_{\gamma}\partial_{t}x_{k}(\gamma),
	\end{align*}
	and
	\begin{align*}
	\int_{-\pi}^{\pi} \partial_{\gamma}x_{k}(\gamma)\cdot\partial_{\gamma}\partial_{t}x_{k}(\gamma) d\gamma &= \int_{-\pi}^{\pi} (\partial_{\gamma}x_{k}(\gamma)\cdot\partial_{\gamma} NL_{k}(\gamma) +  \partial_{\gamma}\lambda_k(\gamma)|\partial_{\gamma}x_{k}(\gamma)|^{2}) d\gamma\\
	&= \int_{-\pi}^{\pi} \partial_{\gamma}x_{k}(\gamma)\cdot\partial_{\gamma} NL_{k}(\gamma)d\g +\int_{-\pi}^{\pi}\partial_{\gamma}\lambda_k(\gamma)A(t) d\gamma\\
	&= \int_{-\pi}^{\pi} \partial_{\gamma}x_{k}(\gamma)\cdot\partial_{\gamma} NL_k(\gamma)d\gamma.
	\end{align*}
	%since
	%$$\int_{-\pi}^{\pi} \partial_{\gamma}\lambda_k(\gamma) d\gamma= \lambda_k(2\pi) - \lambda_k(0) = 0.$$
	Hence,
	$$ \frac{1}{2\pi}\intT \partial_{\beta}x_{k}(\e) \cdot \partial_{\e}NL_{k}(\e) d\e = \partial_{\gamma}x_{k}(\gamma)\cdot\partial_{\gamma} NL_{k}(\gamma) + A(t)\partial_{\gamma}\lambda_k(\gamma)$$
	which implies that
	\begin{multline}\label{c}
	\lambda_k(\gamma,t) =  \frac{\g\!+\!\pi}{2\pi} \! \!\intT \frac{\partial_{\e}x_{k}(\e) \cdot \partial_{\eta}NL_{k}(\e)}{A(t)} d\e - \int_{-\pi}^{\gamma}\!\! \frac{\partial_{\e}x_{k}(\e)\cdot \partial_{\e}NL_{k}(\e)}{A(t)} d\e.
	\end{multline}
	
	Gathering equations (\ref{SQG}-\ref{NL}-\ref{c}), we find the contour evolution system on the half plane. Removing the second summation term of \eqref{NL}, we have the contour evolution system on the whole plane.

	\section{Proof of Theorem \ref{localmodifiedSQG}}
	
	In this section we prove Theorem \ref{localmodifiedSQG} by giving the main a priori estimates in the proof. The regularizing process to get bona fide estimate for the system can be done as in \cite{KYZ}, which gives us the existence part of the theorem as a consequence of the apriori estimates we prove here. As mentioned in the introduction, for the uniqueness arguemnt of the theorem, we will also defer to the proof of uniqueness in \cite{KYZ}; it can be adjusted for the desired regularity of our theorem.

	\subsection{Evolution of $\|x\|_{H^{2}}$}\label{existenceH2}
	In this section, we will consider the evolution of one patch $x_{k}$ in the Sobolev space $H^{2}$. In the end, by summing the estimates for each patch $x_{k}$, we will have the apriori estimate for the evolution of the $H^{2}$ regularity of the whole system of contours $\{x_{j}\}_{j=1,\ldots, n}$.
	
	We begin by differentiating the $\dot{H}^{2}$ norm in time:
	$$\frac12\frac{d}{dt}\|x_{k}\|_{\dot{H}^{2}}^{2} = \int_{\mathbb{T}} \gd^{2}x_{k}(\gamma)\cdot \gd^{2}\partial_{t}x_{k}(\gamma) d\gamma. $$
	
	We can write
	%\begin{multline*}
	%\partial_{\gamma}^{2}\partial_{t}x(\gamma) = \gd^{2}\int_{\mathbb{T}}\frac{\dxdif}{|\xdif|^{\alpha}}d\eta + \gd^{2}\int_{\mathbb{T}} \frac{\dzdif}{|\zdif|^{\alpha}} d\eta + \partial_{\gamma}^{2}(c(\gamma)\partial_{\gamma}x(\gamma))
	%= O + N + \text{tangential terms}
	%\end{multline*}
	\bea\label{SQGdecomp}
	\partial_{\gamma}^{2}\partial_{t}x_{k}(\gamma) = \partial_{\gamma}^{2}NL_{j=k} + \partial_{\gamma}^{2}NL_{j\neq k} + \partial_{\gamma}^{2}(T_{k})
	\eea
	where $NL_{j=k}$ is the term with $j=k$ in the sum of (\ref{SQGnpatch}), $NL_{j\neq k}$ are the other terms in the sum of (\ref{SQGnpatch}) and $T_{k}$ is the $\text{tangential terms}$, i.e. the terms that come from the choice of parametrization $|\partial_{\gamma}x_{k}(\gamma,t)|^{2} = A_{k}(t)$. In the upcoming subsections, we shall begin by bounding the $NL_{j=k}$ nonlinear terms. In particular, we focus on the more difficult to control terms that are due to the boundary. Following the estimates of the $j=k$ terms, we bound the $NL_{j\neq k}$ terms by controlling the distance between distinct contours. In the last subsection, we conclude the apriori estimates by controlling the tangential terms that appear due to the choice of parametrization. 
	\subsection{Controlling the $NL_{j=k}$ terms}\label{nljkterms}
	First, we will examine the nonlinear term $NL_{j=k}$. We split the nonlinear term $NL_{j=k}$ into
	\bea\label{NLkdecomp}
	NL_{j=k} = O_{k} + N_{k}
	\eea
	where $$O_{k}= \int_{\mathbb{T}}\frac{\dxdif}{|\xdif|^{\alpha}}d\eta$$ is the old term from the full space equation and $$N_{k} = \int_{\mathbb{T}} \frac{\dzdif}{|\zdif|^{\alpha}} d\eta$$ is the new term from the half space equation.
	We handle the new term, which is more singular. We consider $$\intT\pg^2 x_k(\g)\cdot\partial_{\gamma}^{2}N_{k}d\g=I_1+I_2+I_3$$
	where
	$$
	I_1=\intT\intT \gd^{2}x_{k}(\gamma)\cdot  \frac{ \gd^{3}x_{k}(\gamma)- \gd^{3}\bar{x}_{k}(\gamma-\eta)}{|\zdif|^{\alpha}} d\eta d\gamma,
	$$
	$$
	I_2=\intT \intT \gd^{2}x_{k}(\gamma)\cdot (\gd^{2}x_{k}(\gamma)-\gd^{2}\bar{x}_{k}(\gamma-\eta))\gd(|\zdif|^{-\alpha}) d\eta d\gamma,
	$$
	and
	$$
	I_3=\intT\intT \gd^{2}x_{k}(\gamma)\cdot(\dzdif)\gd^{2}(|\zdif|^{-\alpha})d\e.
	$$
	We first consider the highest order term in derivatives:
	\begin{align}
	\begin{split}
	I_{1} &=\intT\intT \gd^{2}\bar{x}_{k}(\gamma)\cdot  \frac{ \gd^{3}\bar{x}_{k}(\gamma)- \gd^{3}x(\gamma-\eta)}{|\zdif|^{\alpha}} d\eta d\gamma\\
	&= \intT\intT \gd^{2}\bar{x}_{k}(\gamma-\eta)\cdot  \frac{ \gd^{3}\bar{x}_{k}(\gamma-\eta)- \gd^{3}x_{k}(\gamma)}{|\zdif|^{\alpha}} d\eta d\gamma.
	\end{split}
	\end{align}
	Hence,
	\begin{align}
	\begin{split}
	I_{1} &= \frac{1}{2} \intT\intT (\gd^{2}x_{k}(\gamma)-\gd^{2}\bar{x}_{k}(\gamma-\eta))\cdot  \frac{ \gd^{3}\bar{x}_{k}(\gamma-\eta)- \gd^{3}x_{k}(\gamma)}{|\zdif|^{\alpha}} d\eta d\gamma\\
	=c_\al&\!\intT\intT|\gd^{2}x_{k}(\gamma)-\gd^{2}\bar{x}_{k}(\gamma-\eta)|^{2} \frac{(\dzdif)\!\cdot\! (\zdif)}{|\zdif|^{\alpha +2}} d\eta d\gamma\\
	&\lesssim \intT\intT(|\gd^{2}x_{k}(\gamma)|^{2}+|\gd^{2}\bar{x}_{k}(\gamma-\eta)|^{2}) \frac{|\dzdif|}{|\zdif|^{\alpha +1}} d\eta d\gamma.
	\end{split}
	\end{align}
	We will now use the following interpolation lemma for positive function, restated from \cite{KYZ}:
	\begin{lemma}\label{klemma}[Lemma 2.2, \cite{KYZ}]
		Suppose $\sigma \in [0,1]$ and $\partial_{\gamma}f \in C^{\sigma}(\T)$ such that $f(\gamma) \geq 0$. Then, for any $\gamma \in \T$, the following inequality holds:
		\bea\label{kineq}
		|\partial_{\gamma}f(\gamma)| \leq 2 \|\partial_{\gamma}f\|_{C^{\sigma}}^{1/(1+\sigma)}f(\gamma)^{\sigma/(1+\sigma)}.
		\eea
	\end{lemma}
	Using Lemma \ref{klemma}, we can see that
	\begin{align*}
	|\dzdif| &\leq |\dxdif| + |\gd x_{k}^{(2)}(\gamma-\eta)|\\
	&\leq |\eta|^{1/3}\|\pgx\|_{C^{\frac{1}{3}}} +  |\gd x_{k}^{(2)}(\gamma-\eta)| \\
	&\lesssim |\eta|^{1/3} \|\pgx\|_{C^{\frac{1}{3}}} + \|\gd x_{k}^{(2)}\|_{C^{\frac{1}{2}}}^{2/3} (x_{k}^{(2)}(\gamma-\eta))^{1/3}
	\end{align*}
	where $x_{k} = (x_{k}^{(1)},x_{k}^{(2)})$.
	Furthermore, we have
	\begin{align*}
	(0,2x^{(2)}_{k}(\gamma-\eta)) &= x_{k}(\gamma-\eta) - \bar{x}_{k}(\gamma-\eta)\\
	&= -(\xdif) + \zdif
	\end{align*}
	and therefore
	\begin{align*}
	2x^{(2)}_{k}(\gamma-\eta)&\leq |\xdif| + |\zdif|\\
	&\leq \|x_{k}\|_{C^{1}} |\eta| + |\zdif|.
	\end{align*}
	Using the definition of $F(x)(\gamma,\eta)$,
	$$ |\eta| \leq F(x)(\gamma,\eta) |\xdif| \leq F(x)(\gamma,\eta) |\zdif|$$
	and hence, using Sobolev embeddings
	\bea\label{useful}
	|\dzdif| \lesssim \|x\|_{H^{2}}^{\frac23} \supF^{\frac{1}{3}}|\zdif|^{\frac{1}{3}}.
	\eea
	Thus,
	\begin{align*}
	I_{1} &\lesssim \intT\intT (|\gd^{2}x_{k}(\gamma)|^{2}+|\gd^{2}\bar{x}_{k}(\gamma-\eta)|^{2}) \frac{\|x_{k}\|_{H^{2}}\supF^{\frac{1}{3}}}{|\zdif|^{\alpha +2/3}} d\eta d\gamma\\
	&\lesssim \|x_{k}\|_{H^{2}}\supF^{1+\alpha} \intT\intT (|\gd^{2}x_{k}(\gamma)|^{2}+|\gd^{2}\bar{x}_{k}(\gamma-\eta)|^{2}) \frac{1}{|\eta|^{\alpha+2/3}} d\eta d\gamma\\
	&\lesssim \|x_{k}\|_{H^{2}}^{3}\supF^{1+\alpha}
	\end{align*}
	for $\alpha < 1/3$.
	Next, the same argument holds for the second term for $\alpha < 1/3$:
	\begin{multline*}
	I_{2}= \intT \intT \gd^{2}x_{k}(\gamma)\cdot (\gd^{2}x_{k}(\gamma)-\gd^{2}\bar{x}_{k}(\gamma-\eta))(\zdif)\\ \cdot \frac{(\dzdif)}{|\zdif|^{\alpha+2}} d\eta d\gamma \\
	\leq \intT\intT (|\gd^{2}x_{k}(\gamma)|^{2}+|\gd^{2}x_{k}(\gamma)||\gd^{2}\bar{x}_{k}(\gamma-\eta)|) \frac{\|x_{k}\|_{H^{2}}\supF^{\frac{1}{3}}}{|\zdif|^{\alpha +2/3}} d\eta d\gamma\\
	\lesssim \|x_{k}\|_{H^{2}}^{3}\supF^{1+\alpha}.
	\end{multline*}
	Finally, the terms $I_3$ can be written as
	$$I_3=\intT \pg^2 x(\g)\cdot A_{3}d\g,\quad \mbox{ where }\quad  A_3= A_{31} + A_{32} + A_{33},$$
	and
	$$A_{31} =c_\al \intT (\dzdif) \frac{(\zdif) \cdot (\gd^{2}x_{k}(\gamma)\!-\! \gd^{2}\bar{x}_{k}(\gamma\!-\!\eta))}{|\zdif|^{\alpha+2}} d\eta,$$
	$$A_{32} =c_\al\intT (\dzdif) \frac{|\dzdif|^{2}}{|\zdif|^{\alpha+2}} d\eta $$
	and
	$$A_{33} =c_\al \intT (\dzdif) \frac{((\zdif) \cdot (\gd x_{k}(\gamma)\!-\!\gd \bar{x}_{k}(\gamma\!-\!\eta)))^{2}}{|\zdif|^{\alpha+4}} d\eta.$$
	The $A_{31}$ term is done exactly the same as $A_{2}$ for $\alpha < 1/3$. We now deal with the other terms $A_{32}$ and $A_{33}$. We first need the following lemma:
	\begin{lemma}\label{intbyparts}
		Suppose $f > 0$ $2\pi$-periodic and $f\in H^{2}$ and $ 1< \beta \leq 2$. Then
		$$ \intT \frac{|f'(x)|^{4}}{f(x)^{\beta}} dx\leq C \|f\|_{H^{2}}^{2}.$$
	\end{lemma}
	\begin{proof}
		Let $$P = \intT \frac{f'(x)^{4}}{f(x)^{\beta}} dx.$$ Then, by integration by parts,
		\begin{align*}
		P &= f'(x)^{3}f(x)^{1-\beta}\Big|_{0}^{2\pi} - 3 \intT \frac{f''(x) f'(x)^{2} f(x)}{f(x)^{\beta}} dx + \beta\intT \frac{f'(x)^{4}f(x)}{f(x)^{\beta+1}} dx\\
		&= P_{1} + P_{2} + \beta P\leq \frac{1}{|1-\beta|} (|P_{1}| + |P_{2}|).
		\end{align*}
		The term $P_{1}$ is zero using periodicity. $P_{2}$ is done as follows:
		$$|P_{2}| \leq 3 \|f\|_{\dot{H}^{2}}\Big(\intT \frac{f'(x)^{4}}{f(x)^{2\beta-2}}dx\Big)^{\frac{1}{2}} \leq C_{\epsilon}\|f\|_{\dot{H}^{2}}^{2} + \epsilon P $$
		since $2\beta - 2 \leq  \beta $ for $\beta \leq 2$. Combining the above estimates, we complete the proof.
	\end{proof}
	Let us now estimate the terms from $A_{32}$. By change of variables $\gamma \rightleftarrows \gamma-\eta$ and the property $u\cdot v = \bar{u} \cdot \bar{v}$, we have that
	\begin{align*}
	I_{32} &\eqdef \intT\intT \gd^{2}x_{k}(\gamma) \cdot  (\dzdif) \frac{|\dzdif|^{2}}{|\zdif|^{\alpha+2}} d\eta d\gamma\\
	&=-\intT\intT \gd^{2}\bar{x}_{k}(\gamma-\eta) \cdot  (\dzdif) \frac{|\dzdif|^{2}}{|\zdif|^{\alpha+2}} d\eta d\gamma.
	\end{align*}
	Adding half of the last two expression we find
	\begin{align*}
	I_{32}&= \frac{1}{4} \intT\intT \gd(|\dzdif|^{2})\frac{|\dzdif|^{2}}{|\zdif|^{\alpha+2}} d\eta d\gamma\\
	=&c_\al\!\intT\!\intT (\zdif)\!\cdot\! (\dzdif)\frac{|\dzdif|^{4}}{|\zdif|^{\alpha+4}} d\eta d\gamma.
	\end{align*}
	Hence, by the property $$|\dzdif| \leq 2|\gd x_{k}^{(2)}(\gamma-\eta)|+|\dxdif|,$$ we have that
	\begin{align*}
	|I_{32}| &\lesssim \intT\intT\frac{|\dzdif|^{5}}{|\zdif|^{\alpha+3}} d\eta d\gamma\\
	&\leq \intT\intT |\eta|^{-1+\epsilon} \frac{|\dzdif|^{5}}{|\zdif|^{\alpha+2+\epsilon}} d\eta d\gamma\\
	&\lesssim \intT\intT |\eta|^{-1+\epsilon} \frac{|\gd x_{k}^{(2)}(\gamma-\eta)|^{4} + |\dxdif|^{4}}{|\zdif|^{\alpha+\frac{5}{3}+\epsilon}} d\eta d\gamma
	\end{align*}
	and therefore
	\begin{align*}
	|I_{32}|&\lesssim \intT\frac{|\gd x_{k}^{(2)}(\gamma)|^{4}}{|x_{k}^{(2)}(\gamma)|^{\alpha+\frac{5}{3}+\epsilon}} d\gamma + \|x_{k}\|_{H^{2}}^{4}\intT\intT |\eta|^{-1/3-\alpha}  d\eta d\gamma\\
	&\lesssim I_{321} + \|x_{k}\|_{H^{2}}^{4}
	\end{align*}
	where in the last line we have used Young's inequality and
	$$I_{321} =\intT\frac{|\gd x_{k}^{(2)}(\gamma)|^{4}}{|x_{k}^{(2)}(\gamma)|^{\alpha+\frac{5}{3}+\epsilon}} d\gamma.$$ Since $x_{k}^{(2)} \geq 0$, we obtain that for all $\delta > 0$, the function $x_{k,\delta}^{(2)} = x_{k}^{(2)} + \delta > 0$. Hence,
	\begin{align*}
	I_{321} &= \lim_{\delta\rightarrow 0} \intT\frac{|\gd x_{k}^{(2)}(\gamma)|^{4}}{|x_{k,\delta}^{(2)}(\gamma)|^{\alpha+\frac{5}{3}+\epsilon}} d\gamma \\
	&= \lim_{\delta\rightarrow 0} \intT\frac{|\gd x_{k,\epsilon}^{(2)}(\gamma)|^{4}}{|x_{k,\delta}^{(2)}(\gamma)|^{\alpha+\frac{5}{3}+\epsilon}} d\gamma\\
	&\lesssim \lim_{\delta\rightarrow 0} \|x_{k,\delta}\|_{H^{2}}^{2}.
	\end{align*}
	where we applied Lemma \ref{intbyparts} in the last line for $\alpha < 1/3$. Taking the limit, we obtain that $I_{321}$ is indeed bounded by $\|x_{k}\|_{H^{2}}^{2}$. Hence, we have appropriately bounded the term $A_{32}$. The term $A_{33}$ can be handled similarly:
	
	$$\intT \gd^{2}x_{k}(\gamma)A_{33} d\gamma \leq \|x_{k}\|_{H^{2}}\|A_{33}\|_{L^{2}}$$
	and
	\begin{align*}
	\|A_{33}\|_{L^{2}} &\leq \Big\|\intT d\eta \frac{|\dzdif|^{3}}{|\zdif|^{2+\alpha}} \Big\|_{L^{2}}\\
	&\leq \|\gd x_{k}\|_{C^{1/2}}^{2/3}\Big\|\intT d\eta \frac{|\dzdif|^{2}}{|\zdif|^{5/3+\alpha}} \Big\|_{L^{2}}\\
	&\leq \|x_{k}\|_{H^{2}}^{2/3}(I_{331} + I_{332})
	\end{align*}
	where in the second line we used Lemma \ref{klemma} and
	$$I_{331} = \Big\|\intT d\eta \frac{|\dxdif|^{2}}{|\xdif|^{5/3+\alpha}} \Big\|_{L^{2}}$$
	and
	$$I_{332} = \|\intT d\eta \frac{|\gd x_{k}^{(2)}(\gamma)|^{2}}{|\zdif|^{5/3+\alpha}} \Big\|_{L^{2}}.$$ Then, we control $I_{331}$ and $I_{332}$ as follows:
	\begin{align*}
	I_{331} &\leq \|x_{k}\|_{H^{2}}^{2/3}\supF^{5/3+\alpha}\|\gd x\|_{C^{1/2}}^{2}\Big\|\intT d\eta \frac{1}{|\eta|^{2/3+\alpha}} \Big\|_{L^{2}}\\
	&\lesssim \|x_{k}\|_{H^{2}}^{8/3}
	\end{align*}
	and
	\begin{align*}
	I_{332} &\leq \supF^{1-\epsilon} \Big\|\intT d\eta |\eta|^{-1+\epsilon}\frac{|\gd x_{k}^{(2)}(\gamma-\eta)|^{2}}{|x_{k}^{(2)}(\gamma-\eta)|^{2/3+\alpha+\epsilon}} \Big\|_{L^{2}}\\
	&\lesssim \Big\| |\cdot|^{-1+\epsilon}\Big\|_{L^{1}} \Big( \intT d\gamma \frac{|\gd x_{k}^{(2)}(\gamma)|^{4}}{|x_{k}^{(2)}(\gamma)|^{4/3+2\alpha+2\epsilon}}
	\Big)^{1/2}\\
	&\lesssim \|x_{k}\|_{H^{2}}^{2}
	\end{align*}
	using Lemma \ref{intbyparts} as was done for controlling $I_{321}$.

	\subsection{Controlling the $NL_{j\neq k}$ terms}\label{nljneqkterms}
	We now turn to the second type of terms we need to control, $NL_{j\neq k}$. We begin by defining the quantity
	\bea\label{delta}
	\delta[x](t) \eqdef \min_{i\neq j}\min_{\gamma,\eta\in \T}|x_{i}(\gamma)-x_{j}(\eta)|.
	\eea
	We have the following proposition to control $\delta[x](t)$.
	\begin{prop}\label{deltacontrol}
		For every $k \in \{1,\ldots n\}$, we have the following control over $\delta[x]^{-1}:$
		\begin{multline*}
		\frac{d}{dt}\Big((\delta[x](t))^{-1}\Big)\\ \lesssim  \sum_{j\neq k}\delta[x](t)^{-2}(1+ \|F\|_{L^{\infty}}^{\alpha}\|x_{k}\|_{C^{1}}^{2}+ \delta[x](t)^{-\alpha}\|x_{j}\|_{C^{1}}\|x_{k}\|_{C^{1}})\|x_{k}\|_{C^{1}}.
		\end{multline*}
	\end{prop}
	
	\begin{proof}
		We first note that for $j\neq k$,
		\begin{equation*}
		\|\intT \frac{\partial_{\gamma}x_{k}(\gamma,t)-\partial_{\gamma}\bar{x}_{j}(\gamma-\eta,t)}{|x_{k}(\gamma,t)-\bar{x}_{j}(\gamma-\eta,t)|^{\alpha}}\|_{L^{\infty}} \lesssim (\|x_{k}\|_{C^{1}}+\|x_{j}\|_{C^{1}})\delta[x](t)^{-\alpha}
		\end{equation*}
		and for $j=k$,
		\begin{equation*}
		\|\intT \frac{\partial_{\gamma}x_{k}(\gamma,t)-\partial_{\gamma}\bar{x}_{k}(\gamma-\eta,t)}{|x_{k}(\gamma,t)-\bar{x}_{k}(\gamma-\eta,t)|^{\alpha}}\|_{L^{\infty}} \lesssim \|x\|_{C^{1}}\|F\|_{L^{\infty}}^{\alpha}.
		\end{equation*}
		We obtain similar bounds for the other terms in $NL_{j=k}$ and $NL_{j\neq k}$. For the tangential terms, we have
		\begin{multline*}
		\|\lambda_k(\gamma)\partial_{\gamma}x_{k}(\gamma)\|_{L^{\infty}}\\ \leq \|\lambda_k(\gamma)\|_{L^{\infty}} \|x_{k}\|_{C^{1}}
		\lesssim (1+ \|NL_{k}\|_{L^{\infty}}\|x_{k}\|_{C^{1}}) \|x_{k}\|_{C^{1}} \\ \lesssim \sum_{j\neq k}(1+ \|F\|_{L^{\infty}}^{\alpha}\|x_{k}\|_{C^{1}}^{2}+ \delta[x](t)^{-\alpha}\|x_{k}\|_{C^{1}}\|x_{j}\|_{C^{1}})\|x_{k}\|_{C^{1}}.
		\end{multline*}
		Using the above equations, we see that for any $k \in \{1,\ldots n\}$,
		$$\|\partial_t x_{k}\|_{L^{\infty}} \lesssim \sum_{j\neq k} (1+ \|F\|_{L^{\infty}}^{\alpha}\|x_{k}\|_{C^{1}}^{2}+ \delta[x](t)^{-\alpha}\|x_{k}\|_{C^{1}}\|x_{j}\|_{C^{1}})\|x_{k}\|_{C^{1}}.$$ Hence, twice that bound holds for $\frac{d}{dt}\delta[x](t)$. Finally,
		$$\frac{d}{dt} \Big(\delta[x](t) ^{-1}\Big) = \delta[x](t) ^{-2}\frac{d}{dt}\delta[x](t)$$
		thereby showing our claim.
	\end{proof}
	We can now apply Proposition \ref{deltacontrol} to control the $NL_{j\neq k}$ terms. We decompose 
	$$NL_{j\neq k} = O_{j\neq k} + N_{j\neq k}$$
	where 
	$$O_{j\neq k} =  \frac{\theta_{j}}{2\alpha}\intT \frac{\partial_{\gamma}x_{k}(\gamma,t)-\partial_{\gamma}x_{j}(\gamma-\eta,t)}{|x_{k}(\gamma,t)-x_{j}(\gamma-\eta,t)|^{\alpha}}  d\eta $$
	and 
	$$N_{j\neq k} = \intT \frac{\partial_{\gamma}x_{k}(\gamma,t)-\partial_{\gamma}\bar{x}_{j}(\gamma-\eta,t)}{|x_{k}(\gamma,t)-\bar{x}_{j}(\gamma-\eta,t)|^{\alpha}}$$
	for $j\neq k$. We first consider the $N_{j\neq k}$ terms, as the $O_{j\neq k}$ terms can be controlled similarly.
	$$
	\partial_{\gamma}^{2}N_{j\neq k} \eqdef B_{1} + B_{2} + B_{3}
	$$
	where
	$$B_{1} = \intT \frac{ \gd^{3}x_{k}(\gamma)- \gd^{3}\bar{x}_{j}(\gamma-\eta)}{|x_{k}(\gamma)- \bar{x}_{j}(\gamma-\eta)|^{\alpha}} ,$$
	$$B_{2} = \intT \frac{ (\gd^{2}x_{k}(\gamma)- \gd^{2}\bar{x}_{j}(\gamma-\eta)) (x_{k}(\gamma)- \bar{x}_{j}(\gamma-\eta))\cdot (\partial_{\gamma}x_{k}(\gamma)- \partial_{\gamma}\bar{x}_{j}(\gamma-\eta))}{|x_{k}(\gamma)- \bar{x}_{j}(\gamma-\eta)|^{\alpha+2}}$$
	and
	$$B_{3} = \intT (\dzdif)\gd^{2}(\frac{1}{|(x_{k}(\gamma)- \bar{x}_{j}(\gamma-\eta)|^{\alpha}}).$$
	
	We first consider the highest order term $B_{1}$:
	\begin{align}
	\begin{split}
	J_{1} \eqdef &\intT \gd^{2}x(\gamma)\cdot B_{1} d\gamma = \intT\intT \gd^{2}x_{k}(\gamma)\cdot  \frac{ \gd^{3}x_{k}(\gamma)- \gd^{3}\bar{x}_{j}(\gamma-\eta)}{|x_{k}(\gamma)- \bar{x}_{j}(\gamma-\eta)|^{\alpha}}d\eta d\gamma\\
	=&\intT\intT \gd(|\gd^{2}x_{k}(\gamma)|^{2})  \frac{1}{|x_{k}(\gamma)- \bar{x}_{j}(\gamma-\eta)|^{\alpha}}d\eta d\gamma \\
	&\qquad\qquad+ \intT\intT \gd^{2}x_{k}(\gamma)\cdot  \frac{ \partial_{\eta}\gd^{2}\bar{x}_{j}(\gamma-\eta)}{|x_{k}(\gamma)- \bar{x}_{j}(\gamma-\eta)|^{\alpha}}d\eta d\gamma\\
	=& J_{11} + J_{12}.
	\end{split}
	\end{align}
	For $J_{11}$, we integrate by parts in $\gamma$ to obtain
	\begin{align*}
	|J_{11}| &=\Big|\intT\intT |\gd^{2}x_{k}(\gamma)|^{2}  \frac{(x_{k}(\gamma)- \bar{x}_{j}(\gamma-\eta))\cdot \partial_{\gamma}x_{k}(\gamma)- \partial_{\gamma}\bar{x}_{j}(\gamma-\eta)}{|x_{k}(\gamma)- \bar{x}_{j}(\gamma-\eta)|^{\alpha+2}}d\eta d\gamma\Big|\\
	&\lesssim \|x_{k}\|_{H^{2}}^{2}\delta[x]^{-1-\alpha}(\|x_{j}\|_{C^{1}}+\|x_{k}\|_{C^{1}}).
	\end{align*}
	Integrating by parts in $\eta$, the same bound holds for $J_{22}$. The remaining terms $B_{2}$ and $B_{3}$ can be bounded in $L^{2}$. For example,
	\begin{multline*}
	\|B_{2}\|_{L^{2}_{\gamma}}\\\! = \! \Big\| \! \intT \! \frac{ (\gd^{2}x_{k}(\gamma)- \gd^{2}\bar{x}_{j}(\gamma-\eta)) (x_{k}(\gamma)- \bar{x}_{j}(\gamma-\eta))\cdot (\partial_{\gamma}x_{k}(\gamma)- \partial_{\gamma}\bar{x}_{j}(\gamma-\eta))}{|x_{k}(\gamma)- \bar{x}_{j}(\gamma-\eta)|^{\alpha+2}} d\eta\Big\|_{L^{2}_{\gamma}}\\
	\leq (\|x_{j}\|_{H^{2}}+\|x_{k}\|_{H^{2}})\delta[x]^{-\alpha-1}(\|x_{j}\|_{C^{1}}+\|x_{k}\|_{C^{1}})
	\end{multline*}
	and hence
	\begin{align*}
	\intT \gd^{2}x_{k}(\gamma) B_{2} d\gamma &\leq \|x_{k}\|_{H^{2}}\|B_{2}\|_L^{2}\\
	&\lesssim (\|x_{k}\|_{H^{2}}^{2}+\|x_{j}\|_{H^{2}}\|x_{k}\|_{H^{2}})\delta[x]^{-\alpha-1}(\|x_{j}\|_{C^{1}}+\|x_{k}\|_{C^{1}}).
	\end{align*}
	We can do the same for $B_{3}$ after differentiating it:
	\begin{multline*}
	B_{3} = \intT (\dzdif)\gd\Big(\frac{(x_{k}(\gamma)- \bar{x}_{j}(\gamma-\eta))\cdot (\partial_{\gamma}x_{k}(\gamma)- \partial_{\gamma}\bar{x}_{j}(\gamma-\eta))}{|x_{k}(\gamma)- \bar{x}_{j}(\gamma-\eta)|^{\alpha+2}}\Big)\\
	= \intT (\dzdif)\frac{((x_{k}(\gamma)- \bar{x}_{j}(\gamma-\eta))\cdot (\partial_{\gamma}x_{k}(\gamma)- \partial_{\gamma}\bar{x}_{j}(\gamma-\eta)))^{2}}{|x_{k}(\gamma)- \bar{x}_{j}(\gamma-\eta)|^{\alpha+4}}\\ + \intT (\dzdif)\frac{|\partial_{\gamma}x_{k}(\gamma)- \partial_{\gamma}\bar{x}_{j}(\gamma-\eta)|^{2}}{|x_{k}(\gamma)- \bar{x}_{j}(\gamma-\eta)|^{\alpha+2}}\\
	+\intT (\dzdif)\frac{(x_{k}(\gamma)- \bar{x}_{j}(\gamma-\eta))\cdot (\partial_{\gamma}^{2}x_{k}(\gamma)- \partial_{\gamma}^{2}\bar{x}_{j}(\gamma-\eta))}{|x_{k}(\gamma)- \bar{x}_{j}(\gamma-\eta)|^{\alpha+2}}\\
	\eqdef B_{31} + B_{32} + B_{33}.
	\end{multline*}
	First,
	$$|B_{31}| \leq  \intT |\dzdif|\frac{|\partial_{\gamma}x_{k}(\gamma)- \partial_{\gamma}\bar{x}_{j}(\gamma-\eta)|^{2}}{|x_{k}(\gamma)- \bar{x}_{j}(\gamma-\eta)|^{\alpha+2}}$$
	and
	$$|B_{32}| \leq  \intT |\dzdif|\frac{|\partial_{\gamma}x_{k}(\gamma)- \partial_{\gamma}\bar{x}_{j}(\gamma-\eta)|^{2}}{|x_{k}(\gamma)- \bar{x}_{j}(\gamma-\eta)|^{\alpha+2}}.$$
	Hence,
	\begin{align*}
	\|B_{31}\|_{L^{2}} &\leq (\|x_{k}\|_{C^{1}}+ \|x_{j}\|_{C^{1}})^{3}\delta[x]^{\alpha +2}\\
	&\lesssim  (\|x_{k}\|_{H^{2}}+ \|x_{j}\|_{H^{2}})^{3}\delta[x]^{\alpha +2}.
	\end{align*}
	The same bound holds for $B_{32}$. For $B_{33}$, we have
	\begin{align*}
	\|B_{31}\|_{L^{2}}&\leq \Big\|\intT |\dzdif|\frac{|\partial_{\gamma}^{2}x_{k}(\gamma)- \partial_{\gamma}^{2}\bar{x}_{j}(\gamma-\eta)|}{|x_{k}(\gamma)- \bar{x}_{j}(\gamma-\eta)|^{\alpha+1}}\Big\|_{L^{2}} \\
	&\leq (\|x_{k}\|_{C^{1}}+ \|x_{j}\|_{C^{1}})(\|x_{k}\|_{H^{2}}+ \|x_{j}\|_{H^{2}})\delta[x]^{\alpha +1}\\
	&\lesssim  (\|x_{k}\|_{H^{2}}+ \|x_{j}\|_{H^{2}})^{2}\delta[x]^{\alpha +1}.
	\end{align*}
	Finally, we have the estimate for $i=2,3$
	$$\intT \partial_{\gamma}^{2}x_{k}(\gamma) \cdot B_{i} d\gamma \leq \|x_{k}\|_{H^{2}}\|B_{i}\|_{L^{2}}$$
	and applying the estimates of $\|B_{i}\|_{L^{2}}$ from above, we have the appropriate bounds we want.
	
	\subsection{Controlling the tangential terms}\label{tangentsection}
	Finally, we deal with the tangential term: $c(\gamma) \partial_{\gamma}x_{k}(\gamma)$. We have
	
	\begin{align*}
	\intT \partial_{\gamma}^{2}x_{k}(\gamma)\cdot \lambda_k(\gamma) \partial_{\gamma}^{3}x_{k}(\gamma)d\gamma  &+ 2\intT \partial_{\gamma}^{2}x_{k}(\gamma)\cdot \partial_{\gamma}\lambda_k(\gamma) \partial_{\gamma}^{2}x_{k}(\gamma) d\gamma \\
	&\qquad\qquad+ \intT \partial_{\gamma}^{2}x_{k}(\gamma)\cdot \partial_{\gamma}^{2}\lambda_k(\gamma) \partial_{\gamma}x_{k}(\gamma) d\gamma\\
	%= \intT \frac{1}{2}\partial_{\gamma}(|\partial_{\gamma}^{2}x_{k}(\gamma)|^{2})\lambda_k(\gamma) d\gamma  + 2\intT |\partial_{\gamma}^{2}x_{k}(\gamma)|^{2}\partial_{\gamma}\lambda_k(\gamma)  d\gamma + \intT \partial_{\gamma}^{2}x_{k}(\gamma)\cdot \partial_{\gamma}^{2}\lambda_k(\gamma) \partial_{\gamma}x_{k}(\gamma) d\gamma \\
	= \frac{3}{2}\intT |\partial_{\gamma}^{2}x_{k}(\gamma)|^{2}\partial_{\gamma}\lambda_k(\gamma)  d\gamma &+ \intT \partial_{\gamma}^{2}x_{k}(\gamma)\cdot \partial_{\gamma}^{2}\lambda_k(\gamma) \partial_{\gamma}x_{k}(\gamma) d\gamma \eqdef K_{1} + K_{2}.
	\end{align*}
	First,
	$$K_{2} = \frac{1}{2}\intT \partial_{\gamma}(|\partial_{\gamma}x_{k}(\gamma)|^{2})\partial_{\gamma}^{2}\lambda_k(\gamma)  d\gamma = \frac{1}{2}\intT \partial_{\gamma}(A(t))\partial_{\gamma}^{2}\lambda_k(\gamma)  d\gamma = 0.$$
	Next,
	\begin{multline*}
	\partial_{\gamma}\lambda_k(\gamma) = -\frac{\partial_{\gamma}x_{k}(\gamma)\cdot \partial_{\gamma}NL_{k}(\gamma)}{A(t)} + \frac{1}{2\pi}\intT\frac{\gd x_{k}(\beta)\cdot \gd NL_{k}(\beta)}{A(t)}d\beta\eqdef D_{1} + D_{2}
	\end{multline*}
	Let us consider the conjugate terms of $NL_{k}$, as the other terms in $NL_{k}$ are similar and with more cancellation. For those terms in $\gd NL(\gamma)$, we have
	$$
	\gd \intT \frac{\partial_{\gamma}x_{k}(\gamma)-\partial_{\gamma}\bar{x}_{k}(\gamma-\xi)}{|x_{k}(\gamma)-\bar{x}_{k}(\gamma-\xi)|^{\alpha}} d\xi
	\eqdef D_{11}(\gamma) + D_{12}(\gamma)$$
	where
	$$D_{11} = \intT  - \frac{\alpha}{2}\frac{\partial_{\gamma}x_{k}(\gamma)-\partial_{\gamma}\bar{x}_{k}(\gamma-\xi)}{|x_{k}(\gamma)-\bar{x}_{k}(\gamma-\xi)|^{\alpha+2}} (x_{k}(\gamma)-\bar{x}_{k}(\gamma-\xi))\cdot (\partial_{\gamma}x_{k}(\gamma)-\partial_{\gamma}\bar{x}_{k}(\gamma-\xi))d\xi$$
	and
	$$D_{12} = \intT \frac{\partial_{\gamma}^{2}x_{k}(\gamma)-\partial_{\gamma}^{2}\bar{x}_{k}(\gamma-\xi)}{|x_{k}(\gamma)-\bar{x}_{k}(\gamma-\xi)|^{\alpha}} d\xi.$$
	So first,
	\begin{align*}
	|\gd x_{k}(\gamma) \cdot D_{12}(\gamma)|&=\Big|\partial_{\gamma}x_{k}(\gamma)\cdot\intT \frac{\partial_{\gamma}^{2}x_{k}(\gamma)-\partial_{\gamma}^{2}\bar{x}_{k}(\gamma-\xi)}{|x_{k}(\gamma)-\bar{x}_{k}(\gamma-\xi)|^{\alpha}} d\xi\\ &= \frac{1}{2}\intT \frac{\partial_{\gamma}(\partial_{\gamma}x_{k}(\gamma)^{2})-2\partial_{\gamma}x_{k}(\gamma)\cdot\partial_{\gamma}^{2}\bar{x}_{k}(\gamma-\xi)}{|x_{k}(\gamma)-\bar{x}_{k}(\gamma-\xi)|^{\alpha}} d\xi\Big|\\
	& =\Big| \intT \frac{\partial_{\gamma}x_{k}(\gamma)\cdot\partial_{\gamma}^{2}\bar{x}_{k}(\gamma-\xi)}{|x_{k}(\gamma)-\bar{x}_{k}(\gamma-\xi)|^{\alpha}} d\xi\Big|\\
	&\leq \|x_{k}\|_{C^{1}} \|F\|_{L^{\infty}}^{\alpha}\|x_{k}\|_{\dot{H}^{2}} \| |\xi|^{\alpha}\|_{L^{2}}\\
	&\lesssim \|x_{k}\|_{C^{1}} \|F\|_{L^{\infty}}^{\alpha}\|x_{k}\|_{\dot{H}^{2}}.
	\end{align*}
	For $D_{11}$, we have that it is bounded for $\alpha < 2/3$ using previous arguments since
	$$|D_{11}| \lesssim \intT \frac{|\dzdif|^{2}}{|\zdif|^{\alpha+1}}.$$
	For $NL_{j\neq k}$, the derivative of the conjugate terms can be written as the sum $E_{1} + E_{2}$ where
	$$E_{1} = \intT \frac{\partial_{\gamma}^{2}x_{k}(\gamma)-\partial_{\gamma}^{2}\bar{x}_{j}(\gamma-\xi)}{|x_{k}(\gamma)-\bar{x}_{j}(\gamma-\xi)|^{\alpha}} d\xi$$
	and 
	$$E_{2} = \intT\frac{\partial_{\gamma}x_{k}(\gamma)-\partial_{\gamma}\bar{x}_{j}(\gamma-\xi)}{|x_{k}(\gamma)-\bar{x}_{j}(\gamma-\xi)|^{\alpha+2}} (x_{k}(\gamma)-\bar{x}_{j}(\gamma-\xi))\cdot (\partial_{\gamma}x_{k}(\gamma)-\partial_{\gamma}\bar{x}_{j}(\gamma-\xi))d\xi.$$
	Consider $E_{1}$ first:
	\begin{align*}
	\gd x_{k}(\gamma)\cdot \intT \frac{\partial_{\gamma}^{2}x_{k}(\gamma)-\partial_{\gamma}^{2}\bar{x}_{j}(\gamma-\xi)}{|x_{k}(\gamma)-\bar{x}_{j}(\gamma-\xi)|^{\alpha}} d\xi &= 
	%\intT \frac{\frac{1}{2}\partial_{\gamma}(|\partial_{\gamma}x_{k}(\gamma)|^{2})-\partial_{\gamma}x_{k}(\gamma)\cdot\partial_{\gamma}^{2}\bar{x}_{j}(\gamma-\xi)}{|x_{k}(\gamma)-\bar{x}_{j}(\gamma-\xi)|^{\alpha}} d\xi\\&=
	-\intT \frac{\partial_{\gamma}x_{k}(\gamma)\cdot\partial_{\gamma}^{2}\bar{x}_{j}(\gamma-\xi)}{|x_{k}(\gamma)-\bar{x}_{j}(\gamma-\xi)|^{\alpha}} d\xi\\
	&\leq \|x_{k}\|_{C^{1}}\|x_{j}\|_{H^{2}}\delta[x]^{-1-\alpha}.
	\end{align*}
	Similarly, we can bound the term from $E_{2}$ by
	$$\gd x_{k}(\gamma)\cdot E_{2} \leq (\|x_{k}\|_{C^{1}}^{2}+\|x_{j}\|_{C^{1}}^{2})\|x_{k}\|_{C^{1}}\delta[x]^{-1-\alpha}.$$
	Hence, $K_{1}$ is also appropriately bounded since
	$$K_{1} \lesssim \|x_{k}\|_{H^{2}}^{2}\|\partial_{\gamma}\lambda_k\|_{L^{\infty}}\delta[x]^{-1-\alpha}.$$
	
	Summarizing, combining the estimates for the nonlinear terms from Sections \ref{nljkterms}, \ref{nljneqkterms} and \ref{tangentsection}, we obtain the estimate $$\intT \gd^{2} x_{k}(\gamma) \cdot \partial_{t}\gd^{2}x_{k}(\gamma) d\gamma \lesssim \mathcal{P}(\|x_{1}\|_{H^{2}}, \ldots, \|x_{n}\|_{H^{2}})$$ for a polynomial $\mathcal{P}$ with coefficients depending on $\supF$, $\delta[x]^{-1}$ and $\alpha$. Thus, we have the appropriate apriori estimate for the evolution of $\|x_{k}\|_{H^{2}}$.

	\subsection{Control of $\|F(x)\|_{L^{\infty}}$}\label{arcchordcontrol}
	
	Consider $$F(x_{k}) \eqdef F(x_{k})(\gamma,\eta) = \frac{|\eta|}{|\xdif|}.$$
	%and let $$\|F(x_{k})\|_{L^{{p}}} \eqdef \|F(x)\|_{L^{p}_{\gamma,\eta}} \ and \ \|F(x_{k})\|_{L^{{\infty}}} \eqdef sup_{\gamma,\eta \in \T} F(x_{k})(\gamma,\eta).$$
	Then
	$$ \partial_t F(x_{k})=-\frac{|\eta|(\xdif)\cdot(\partial_{t}x_{k}(\gamma)- \partial_{t}x_{k}(\gamma-\eta))}{|\xdif|^{3}}$$
	%We have the inequality
	%\bea\label{gammaregularity}
	%|\gdF| \leq \supF^{2}\holder |\eta|^{-\frac{1}{2}}
	%\eea
	We can write
	\begin{multline}\label{dtx}
	\partial_{t}x_{k}(\gamma)- \partial_{t}x_{k}(\gamma-\eta) = I_{1}+ I_{2} + I_{3}+ N_{1} + N_{2} + N_{3} \\ +J_{1} + J_{2} + J_{3}+ M_{1} + M_{2} + M_{3}\\ + \lambda_{k}(\gamma)\partial_{\gamma}x_{k}(\gamma) - \lambda_{k}(\gamma-\eta)\partial_{\gamma}x_{k}(\gamma-\eta) 
	\end{multline}
	where
	\bea\label{I1}
	I_{1} \eqdef (\dxdif) \int_{\T} \frac{1}{|\xigamma|^{\alpha}} d\xi,
	\eea
	\bea\label{I2}
	I_{2} \eqdef \int_{T} -\frac{\gd x_{k}(\gamma-\xi) - \gd x_{k}(\gamma-\xi-\eta)}{|\xigamma|^{\alpha}} d\xi
	\eea
	and
	\begin{equation}\label{I3}
	I_{3} \\ \eqdef \int_{\T} (\dxigammaeta) (g_{k}(\gamma,\xi)- g_{k}(\gamma-\eta,\xi))d\xi.
	\end{equation}
	where 
	\begin{equation*}
	g_{k}(\g,\xi) =|x_k(\gamma)-x_k(\gamma-\xi)|^{-\al},
	\end{equation*}
	and where
	\bea\label{N1}
	N_{1} \eqdef (\dxdif) \int_{\T} \frac{1}{|\zxigamma|^{\alpha}} d\xi,
	\eea
	\bea\label{N2}
	N_{2} \eqdef \int_{T} -\frac{\gd \bar{x}_{k}(\gamma-\xi) - \gd \bar{x}_{k}(\gamma-\xi-\eta)}{|\zxigamma|^{\alpha}} d\xi
	\eea
	and
	\begin{equation}\label{N3}
	N_{3} \\ \eqdef \int_{\T} (\dzxigammaeta) \Big(h_{k}(\gamma) - h_{k}(\gamma -\eta)\Big) d\xi
	\end{equation}
	where
	$$ h_{k}(\gamma) = \frac{1}{|\zxigamma|^{\alpha}}.$$
	The terms $J_{i}$ correspond to the terms analogous to $I_{i}$ for the terms in the sum \eqref{SQG} where $j\neq k$ and the terms $M_{i}$ correspond to the terms analogous to $N_{i}$ for the terms in the sum \eqref{SQG} where $j\neq k$. For reference, we explicitly write these terms $M_{i}$:
	\bea\label{M1}
	M_{1} \eqdef (\dxdif) \int_{\T} \frac{1}{|x_{k}(\gamma) - \bar{x}_{j}(\gamma-\xi)|^{\alpha}} d\xi,
	\eea
	\bea\label{M2}
	M_{2} \eqdef \int_{T} -\frac{\gd \bar{x}_{j}(\gamma-\xi) - \gd \bar{x}_{j}(\gamma-\xi-\eta)}{|x_{k}(\gamma) - \bar{x}_{j}(\gamma-\xi)|^{\alpha}} d\xi
	\eea
	and
	\begin{equation}\label{M3}
	M_{3} \\ \eqdef \int_{\T} (\gd x_{k}(\gamma-\eta) - \gd \bar{x}_{j}(\gamma-\xi-\eta)) \Big(h_{j,k}(\gamma) - h_{j,k}(\gamma -\eta)\Big) d\xi
	\end{equation}
	where
	$$h_{j,k} = \frac{1}{|x_{k}(\gamma) - \bar{x}_{j}(\gamma-\xi)|^{\alpha}}.$$
	Now, we have the evolution of $F(x)(\gamma,\eta)$ given by
	\begin{align*}
	\frac{d}{dt}F(x)(\gamma,\eta) &= -\frac{|\eta| (\xdif)\cdot(\partial_{t}x(\gamma) - \partial_{t}x(\gamma-\eta))}{|\xdif|^{3}}\\
	&=\! - \!\sum_{j=1}^{3}\! \! \frac{|\eta| (\xdif)\cdot(I_{i} + N_{i}+J_{i}+M_{i} + \text{tangential terms})}{|\xdif|^{3}}.
	\end{align*}
	We will do the estimates for the terms $N_{i}$. The $I_{i}$ terms have more cancellation and can be done similarly. For $N_{1}$,
	\begin{multline*}
	\frac{|\eta| (\xdif)\cdot(\dxdif)}{|\xdif|^{3}}\intT \frac{1}{|\zxigamma|^{\alpha}} d\xi\\ \eqdef N_{11} + N_{12}
	\end{multline*}
	where
	\begin{multline*}
	N_{11} =\frac{|\eta| (\xdif - \eta \partial_{\gamma}x_{k}(\gamma))\cdot(\dxdif)}{|\xdif|^{3}}\\ \cdot \intT \frac{1}{|\zxigamma|^{\alpha}} d\xi
	\end{multline*}
	and
	$$N_{12} = \frac{|\eta| (\eta \partial_{\gamma}x_{k}(\gamma))\cdot(\dxdif)}{|\xdif|^{3}}\intT \frac{1}{|\zxigamma|^{\alpha}} d\xi.$$
	We can bound $N_{11}$ as follows:
	\begin{align*}
	|N_{11}| &\leq \|F\|_{L^{\infty}}^{3}\|\partial_{\gamma}x_{k}\|_{C^{1/2}}^{2}\Big|\intT \frac{1}{|\zxigamma|^{\alpha}} d\xi\Big|\\
	&\leq \|F\|_{L^{\infty}}^{3+\alpha}\|\partial_{\gamma}x_{k}\|_{C^{1/2}}^{2}\intT |\xi|^{-\alpha}\lesssim \|F\|_{L^{\infty}}^{3+\alpha}\|\partial_{\gamma}x_{k}\|_{\dot{H}^{2}}^{2}.
	\end{align*}
	Next, using the fact that $|\partial_{\gamma}x_{k}(\gamma)|^{2} = A(t)$ is constant with respect to $\gamma$, we obtain that
	$$|\partial_{\gamma}x_{k}(\gamma)|^{2} = \frac{1}{2}|\partial_{\gamma}x_{k}(\gamma)|^{2} + \frac{1}{2}|\partial_{\gamma}x_{k}(\gamma-\eta)|^{2}.$$ Using this, we see that
	\begin{align*}
	|N_{12}| &\leq \Big|\frac{|\eta| (\eta \partial_{\gamma}x_{k}(\gamma))\cdot(\dxdif)}{|\xdif|^{3}}\intT \frac{1}{|\zxigamma|^{\alpha}} d\xi\Big|\\
	&\leq |\eta|^{-1}\|F\|_{L^{\infty}}^{2+\alpha}(|\partial_{\gamma}x_{k}(\gamma)|^{2} - \partial_{\gamma}x_{k}(\gamma)\cdot \partial_{\gamma}x_{k}(\gamma-\eta))\intT |\xi|^{-\alpha}\\
	&\lesssim |\eta|^{-1}\|F\|_{L^{\infty}}^{2+\alpha}(\frac{1}{2}|\partial_{\gamma}x_{k}(\gamma)|^{2} + \frac{1}{2}|\partial_{\gamma}x_{k}(\gamma-\eta)|^{2}- \partial_{\gamma}x_{k}(\gamma)\cdot \partial_{\gamma}x_{k}(\gamma-\eta))\\
	&= \frac{1}{2}|\eta|^{-1}\|F\|_{L^{\infty}}^{2+\alpha}|\partial_{\gamma}x_{k}(\gamma)-\partial_{\gamma}x_{k}(\gamma-\eta)|^{2}\\
	&\lesssim \|F\|_{L^{\infty}}^{2+\alpha}\|\partial_{\gamma}x\|_{C^{1/2}}^{2}\\
	&\lesssim \|F\|_{L^{\infty}}^{2+\alpha}\|\partial_{\gamma}x\|_{\dot{H}^{2}}^{2}.
	\end{align*}
	We move onto the $N_{2}$ term.
	\begin{align*}
	\Big|\!\frac{|\eta|(\xdif)\cdot N_{2}}{|\xdif|^{3}}\!\Big|\! &= \! \Big|\!\frac{|\eta|(\xdif)}{|\xdif|^{3}} \! \\ &\cdot\! \! \intT \! d\xi \frac{\partial_{\gamma}\bar{x}_{k}(\gamma-\xi)-\partial_{\gamma}\bar{x}_{k}(\gamma-\xi-\eta)}{|x_{k}(\gamma)-\bar{x}_{k}(\gamma-\eta)|^{\alpha}}\!\Big|\\
	&\leq \|F\|_{L^{\infty}}^{2+\alpha}\int_{0}^{1} ds \intT d\xi |\frac{\partial^{2}_{\gamma}\bar{x}_{k}(\gamma-\xi-(s-1)\eta)|}{|\xi|^{\alpha}}\\
	&\leq \|F\|_{L^{\infty}}^{2+\alpha}\int_{0}^{1} ds \|\partial^{2}_{\gamma}\bar{x}_{k}(\gamma-\xi-(s-1)\eta)\|_{L^{2}_{\xi}}\||\xi|^{\alpha}\|_{L^{2}_{\xi}}\\
	&\lesssim  \|F\|_{L^{\infty}}^{2+\alpha}\|x_{k}\|_{\dot{H}^{2}}
	\end{align*}
	for $\alpha < 1/2$.
	Finally, to deal with $N_{3}$, we use the fact that for some $\gamma_{1}$ between $\gamma$ and $\eta$,
	$$ |h_{k}(\gamma)- h_{k}(\gamma-\eta)| = |\eta| |\partial_{\gamma}h_{k}(\gamma_{1})|.$$
	Differentiating $h_{k}(\gamma)$ to get
	$$\partial_{\gamma}h_{k}(\gamma) = -\frac{\alpha}{2}\frac{(\dxdif)\cdot (\xdif)}{|\xdif|^{\alpha+2}},$$
	we see that
	\begin{align*}
	\intT |\partial_{\gamma}h_{k}(\gamma_{1})|d\xi &\lesssim \intT d\xi \frac{|\partial_{\gamma}x_{k}(\gamma_{1})- \partial_{\gamma}\bar{x}_{k}(\gamma_{1}-\xi)|}{|x_{k}(\gamma_{1})-\bar{x}_{k}(\gamma_{1}-\xi)|^{\alpha+1}}\\
	&\lesssim \|x_{k}\|_{H^{2}}\|F\|_{L^{\infty}}^{1+\alpha}
	\end{align*}
	for $\alpha < 1/3$ as we have done in previous calculations.
	Hence, the $N_{3}$ term is bounded by
	\begin{align*}
	|N_{3}| &\lesssim \intT |\dzxigammaeta| |\eta| |\partial_{\gamma}h_{k}(\gamma_{1})| d\xi \\
	&\lesssim \|x_{k}\|_{C^{1}} |\eta| \intT|\partial_{\gamma}h_{k}(\gamma_{1})| d\xi \\ &\lesssim  |\eta| \|x_{k}\|_{C^{1}}\|x_{k}\|_{H^{2}}\|F\|_{L^{\infty}}^{1+\alpha}.
	\end{align*}
	Therefore,
	$$\frac{|\eta|(\xdif)\cdot N_{3}}{|\xdif|^{3}} \lesssim \|F\|^{3+\alpha}\|x\|_{H^{2}}\|x\|_{C^{1}}.$$
	Now we look to the terms with $j\neq k$. Let us analyze the $M_{i}$ terms, as the $J_{i}$ terms can be handled similarly. Let us first look at $M_{1}$:
	\begin{multline*}
	\frac{|\eta| (\xdif)\cdot(\dxdif)}{|\xdif|^{3}}\intT \frac{1}{|\partial_{\gamma}x_{k}(\gamma)-\partial_{\gamma}\bar{x}_{j}(\gamma-\xi)|^{\alpha}} d\xi \\ \eqdef M_{11} + M_{12}.
	\end{multline*}
	where
	\begin{multline*}
	M_{11} = \frac{|\eta| (\xdif - \eta \partial_{\gamma}x_{k}(\gamma))\cdot(\dxdif)}{|\xdif|^{3}}\\ \cdot\intT \frac{1}{|\partial_{\gamma}x(\gamma)-\partial_{\gamma}\bar{x}_{j}(\gamma-\xi)|^{\alpha}} d\xi
	\end{multline*}
	and
	\begin{multline*}
	M_{12} = \frac{|\eta| (\eta \partial_{\gamma}x(\gamma))\cdot(\dxdif)}{|\xdif|^{3}}\intT \frac{1}{|\partial_{\gamma}x(\gamma)-\partial_{\gamma}\bar{x}_{j}(\gamma-\xi)|^{\alpha}} d\xi.
	\end{multline*}
	Now, $M_{11}$ is bounded similarly to $N_{11}$ except the integral term is bounded as follows:
	$$\intT \frac{1}{|\partial_{\gamma}x_{k}(\gamma)-\partial_{\gamma}\bar{x}_{j}(\gamma-\xi)|^{\alpha}} d\xi \leq \delta[x]^{-\alpha}.$$ Hence,
	$$M_{11}\leq \supF^{2}\delta[x]^{-\alpha}\|x\|_{H^{2}}^{2}.$$
	We can follow the techniques from the terms $N_{12}$ and $N_{2}$ with the adjustment above to obtain similar bounds for $M_{12}$ and $M_{2}$:
	$$M_{12} \leq \delta[x]^{-\alpha}\supF^{2}\|x_{k}\|_{H^{2}}^{2}$$ and 
	$$M_{2} \leq \delta[x]^{-\alpha}\supF^{2}\|x_{j}\|_{H^{2}}.$$
	The last term $M_{3}$ is done similarly to $N_{3}$ except we replace
	$h_{k}(\gamma)$ with
	$h_{j,k}(\gamma)$.
	Finally, we have to take care of the tangential terms given by
	\begin{multline*}
	\lambda_k(\gamma)\partial_{\gamma}x_{k}(\gamma) - \lambda_k(\gamma-\eta)\partial_{\gamma}x_{k}(\gamma-\eta)\\ =  \lambda_k(\gamma)(\partial_{\gamma}x_{k}(\gamma) - \partial_{\gamma}x_{k}(\gamma-\eta)) +  (\lambda_k(\gamma) - \lambda_k(\gamma-\eta))\partial_{\gamma}x_{k}(\gamma-\eta)\\ \eqdef C_{1} + C_{2}.
	\end{multline*}
	We can bound the term with $C_{1}$ as follows. First, decompose it
	$$\Big|\frac{|\eta|(\xdif)\cdot c(\gamma)(\partial_{\gamma}x_{k}(\gamma) - \partial_{\gamma}x_{k}(\gamma-\eta))}{|\xdif|^{3}}\Big| \leq C_{11} + C_{12}$$
	where
	\begin{align*}
	C_{11} &\eqdef \Big|\frac{|\eta|(\xdif - \eta\partial_{\gamma}x(\gamma))\cdot \lambda_k(\gamma)(\partial_{\gamma}x_{k}(\gamma) - \partial_{\gamma}x_{k}(\gamma-\eta))}{|\xdif|^{3}}\Big|\\
	&\leq \|F\|_{L^{\infty}}^{3}\|\partial_{\gamma}x_{k}(\gamma)\|_{C^{1/2}}^{2} |c(\gamma)|\\
	&\leq \|F\|_{L^{\infty}}^{3}\|x\|_{H^{2}}^{2} |\lambda_k(\gamma)|.
	\end{align*}
	We have the bound
	\begin{align*}
	|\lambda_k(\beta)| &\lesssim 2 \int_{-\pi}^{\pi} \frac{|\partial_{\gamma}NL_{k}(\gamma)|}{A(t)^{1/2}} d\gamma.
	\end{align*}
	Now, for any $\beta$, we consider the nonlinear terms from the conjugate terms. The other terms are similar and have more cancellation.
	\begin{multline*}
	\int_{-\pi}^{\pi}|\partial_{\gamma}NL_{k}(\gamma)| \leq  \int_{-\pi}^{\pi} d\gamma\intT d\xi \frac{|\partial_{\gamma}^{2}x_{k}(\gamma) - \partial_{\gamma}^{2}\bar{x}_{k}(\gamma-\xi)|}{|\xigamma|^{\alpha}} \\+\int_{-\pi}^{\gamma}d\gamma \intT d\xi \frac{|\partial_{\gamma}x_{k}(\gamma) - \partial_{\gamma}\bar{x}_{k}(\gamma-\xi)|^{2}}{|\xigamma|^{\alpha+1}}\\
	\lesssim \|x_{k}\|_{H^{2}}\|F\|_{L^{\infty}}^{\alpha} + \|x_{k}\|_{H^{2}}^{2}\|F\|_{L^{\infty}}^{\alpha+1}
	\end{multline*}
	for $\alpha < 2/3$ and for any $\beta$. Hence,
	$$| \lambda_k(\beta)| \lesssim \|x_{k}\|_{H^{2}}\|F\|_{L^{\infty}}^{\alpha} + \|x_{k}\|_{H^{2}}^{2}\|F\|_{L^{\infty}}^{\alpha+1}.$$
	For $C_{12}$, we have
	\begin{align*}
	C_{12} &\eqdef \Big|\frac{|\eta|(\eta\partial_{\gamma}x_{k}(\gamma))\cdot \lambda_k(\gamma)(\partial_{\gamma}x_{k}(\gamma) - \partial_{\gamma}x_{k}(\gamma-\eta))}{|\xdif|^{3}}\Big|\\
	&\leq \|F\|_{L^{\infty}}^{3}(|\partial_{\gamma}x_{k}(\gamma)|^{2} -\partial_{\gamma}x(\gamma)\cdot \partial_{\gamma}x_{k}(\gamma-\eta) ) |\lambda_k(\gamma)|\\
	&\leq \|F\|_{L^{\infty}}^{3}\|x_{k}\|_{H^{2}}^{2} |\lambda_{k}(\gamma)|
	\end{align*}
	where the second step resembles previous calculations. Finally by the bound on $|c(\gamma)|$, we have finished $C_{1}$.
	For $C_{2}$, we use the fact that
	$$\lambda_{k}(\gamma) - \lambda_{k}(\gamma-\eta) = \eta \partial_{\gamma}\lambda_{k}(\sigma)$$
	for some $\sigma$ between $\gamma$ and $\gamma - \eta$.
	Now, using the estimate from the previous section for $\partial_{\gamma}\lambda_{k}$, we are done since
	\begin{align*}
	\Big|\frac{|\eta|(\xdif)\cdot C_{2}}{|\xdif|^{3}}\Big| &\leq \frac{|\eta|^{2}\|x_{k}\|_{C^{1}} \|\partial_{\gamma}\lambda_k(\gamma)\|_{L^{\infty}}}{|\xdif|^{2}}\\
	&\leq \supF^{2} \|x_{k}\|_{C^{1}} \|\partial_{\gamma}\lambda_k(\gamma)\|_{L^{\infty}}
	\end{align*}
	and $ \|\partial_{\gamma}\lambda_k(\gamma)\|_{L^{\infty}} $ is bounded by a polynomial of the quantities $\|x_{j}\|_{H^{2}}$ for $j = 1,\ldots , n$, $\supF$ and $\delta[x]^{-1}$.
	
	\subsection{Uniqueness}
	
	In this section, we present the argument for uniqueness of solutions to the SQG system. We consider any patch type solution with $\partial D_j(t)\in C([0,T],H^2)$ non self-intersecting and $D_j(t)\cap D_k(t)=\phi$ for $k\neq j$. Given any parameterization of the boundary of the patches, we perform changes of variables to find $D_j(t)=\{x_{j}(\gamma,t),\gamma\in\T\}$ with $|\gd x_{j}(\gamma)|^{2} = A_{j}(t)$ only depending on time for $j = 1, \ldots, n$ and solutions of the contour equations (\ref{SQG}-\ref{NL}-\ref{c}) (see \cite{CCG} for more details). Then, suppose $y(\xi,t)$ is a contour reparametrization such that 
	$$ x_{j}(\gamma,t) = y_{j}(\phi_{j}(\gamma,t),t)$$
	for $j = 1, \ldots, n$. Then,
	\bea\label{changevolution}\partial_{t}x_{k}(\gamma,t) = \partial_{t}y_{k}(\phi_{k}(\gamma,t),t) + \partial_{\xi}y_{k}(\phi_{k}(\gamma,t),t)\cdot \partial_{t}\phi_{k}(\gamma,t).
	\eea
	From the contour equation of $x(\gamma,t)$, we also have that
	$$\partial_{t}x_{k}(\gamma,t) =  A_{1} + A_{2} + A_{3}$$
	where
	$$A_{1}= \sum_{j=1}^{n}\intT \frac{\partial_{\xi}y_{k}(\phi_{k}(\gamma,t),t)\cdot \gd\phi(\gamma,t) - \partial_{\xi}y_{j}(\phi_{j}(\gamma-\eta),t)\gd\phi_{j}(\gamma-\eta,t)}{|x_{k}(\gamma,t)-x_{j}(\gamma-\eta,t)|^{\alpha}} d\eta,$$
	$$A_{2} = \sum_{j=1}^{n}\intT \frac{\partial_{\xi}y_{k}(\phi_{k}(\gamma,t),t)\cdot \gd\phi(\gamma,t) -  \partial_{\xi}\bar{y}_{j}(\phi_{j}(\gamma-\eta),t)\gd\phi_{j}(\gamma-\eta,t)}{|x_{k}(\gamma,t)-\bar{x}_{j}(\gamma-\eta,t)|^{\alpha}} d\eta$$
	and
	$$ A_{3} = \lambda_{k}(\gamma)\gd y_{k}(\phi_{k}(\gamma,t),t)\gd \phi_{k}(\gamma,t).$$
	Then, we can write
	$$A_{1} = \sum_{j=1}^{n}\partial_{\xi}y_{k}(\phi_{k}(\gamma,t),t)\cdot  A^{(j)}_{11} + A^{(j)}_{12} \text{   and   } A_{2} = \sum_{j=1}^{n} \partial_{\xi}y_{k}(\phi_{k}(\gamma,t),t)\cdot A^{(j)}_{21} + A^{(j)}_{22}$$
	where
	$$ A^{(j)}_{11} = \intT \frac{ \gd\phi_{k}(\gamma,t) -\gd\phi_{j}(\g-\eta,t)}{|x_{k}(\gamma,t)-x_{j}(\gamma-\eta,t)|^{\alpha}} d\eta,$$
	$$ A^{(j)}_{12} = \intT \frac{ \partial_{\xi}y_{k}(\phi_{k}(\gamma,t),t) - \partial_{\xi}y_{j}(\phi_{j}(\g-\eta,t),t)}{|x_{k}(\gamma,t)-x_{j}(\gamma-\eta,t)|^{\alpha}} \cdot \gd\phi_{j}(\g-\eta,t)d\eta,$$
	$$ A^{(j)}_{21} = \intT \frac{ \gd\phi_{k}(\gamma,t) -\gd\phi_{j}(\g-\eta,t)}{|x_{k}(\gamma,t)-\bar{x}_{j}(\gamma-\eta,t)|^{\alpha}} d\eta,$$
	and
	$$ A^{(j)}_{22} = \intT \frac{ \partial_{\xi}y_{k}(\phi_{k}(\gamma,t),t) - \partial_{\xi}\bar{y}_{j}(\phi_{j}(\g-\eta,t),t)}{|x_{k}(\gamma,t)-x_{j}(\gamma-\eta,t)|^{\alpha}} \cdot \gd\phi_{j}(\g-\eta,t)d\eta.$$
	Then, comparing with \eqref{changevolution}, we see that
	\bea\label{phievolve}\partial_{t} \phi_{k}(\gamma,t) = \sum_{j=1}^{n} A^{(j)}_{11}+ A^{(j)}_{21} + \lambda_k\pg\phi_k.
	\eea
	
	Above changes of variables $\phi_j$ allow to find $\partial D_j(t)=\{y_j(\xi,t):\xi\in\T\}$ and $y_j(\xi,t)$ as solutions of the contour equations \eqref{SQGnpatch}. We then show that there is uniqueness for the system \eqref{SQGnpatch} and therefore uniqueness of the problem. First we give the appropriate regularity for the changes of variables.
	
	\begin{prop}
		The change of parametrization $\phi_{k}(\gamma,t)-\gamma \in C([0,T];H^{2})$.
	\end{prop}
	
	\begin{proof}
		Differentiating in time,
		\bea\label{phih2evo}\frac12\frac{d}{dt} \|\phi_{k}-\gamma\|_{H^{2}}^{2} =\intT (\phi_{k}(\gamma) -\gamma)\partial_{t}\phi_{k}(\gamma) d\gamma+ \intT \gd^{2}\phi_{k}(\gamma) \partial_{t}\gd^{2}\phi_{k}(\gamma) d\gamma.
		\eea
		The first term on the right is of low order and not difficult to handle. We provide details for the most singular ones. Differentiating \eqref{phievolve} twice in $\gamma$, we obtain that
		$$\partial_{t}\gd^{2} \phi_k(\gamma,t) = \sum_{j=1}^{n}\gd^{2}A^{(j)}_{11}+ \gd^{2}A^{(j)}_{21} + \gd^{2}(\lambda_{k}(\gamma)\gd\phi_{k}(\gamma)).$$
		We will only consider the estimates for $ \gd^{2}A^{(j)}_{21}$ and $ \gd^{2}(\lambda_{k}(\gamma)\gd\phi_{k}(\gamma))$, as the first term is easier. Throughout the estimates of the proof, the implicit constant in an inequality ''$\lesssim$" depends continuously on $\|F(x_{j})\|_{L^{\infty}}$, $\delta[x]^{-1}$, $\|x_{j}\|_{H^{2}}$ and $\alpha$ for $j= 1, \ldots , n$. First,

		$$\gd^{2} A^{(j)}_{21} = B^{(j)}_{1}+B^{(j)}_{2}+B^{(j)}_{3}+B^{(j)}_{4}+B^{(j)}_{5}$$
		where
		$$B^{(j)}_{1} = \intT \frac{ \gd^{3}\phi_{k}(\gamma) -\gd^{3}\phi_{j}(\gamma-\eta)}{|x_{k}(\gamma)-\bar{x}_{j}(\gamma-\eta)|^{\alpha}} d\eta$$
		$$B^{(j)}_{2} = 2c_{\alpha} \intT \frac{ \gd^{2}\phi_{k}(\gamma) -\gd^{2}\phi_{j}(\gamma-\eta)}{|x_{k}(\gamma)-\bar{x}_{j}(\gamma-\eta)|^{\alpha+2}} b_{2}^{(j)}(\gamma,\eta) d\eta$$
		$$B^{(j)}_{3} = c_{\alpha}\intT \frac{ \gd\phi_{k}(\gamma) -\gd\phi_{j}(\gamma-\eta)}{|x_{k}(\gamma)-\bar{x}_{j}(\gamma-\eta)|^{\alpha+2}}  b_{3}^{(j)}(\gamma,\eta)d\eta$$
		$$B^{(j)}_{4} = c_{\alpha}\intT \frac{ \gd\phi_{k}(\gamma) -\gd\phi_{j}(\gamma-\eta)}{|x_{k}(\gamma)-\bar{x}_{j}(\gamma-\eta)|^{\alpha+2}} b_{4}^{(j)}(\gamma,\eta) d\eta$$
		and
		$$B^{(j)}_{5} = \tilde{c}_{\alpha}\intT \frac{ \gd\phi_{k}(\gamma) -\gd\phi_{j}(\gamma-\eta)}{|x_{k}(\gamma)-\bar{x}_{j}(\gamma-\eta)|^{\alpha+4}} |b_{2}^{(j)}(\gamma,\eta)|^{2}d\eta$$
		where $c_{\alpha}$ are constants depending on $\alpha$ 
		$$
		b_{2}^{(j)}(\gamma,\eta) = (x_{k}(\gamma)-\bar{x}_{j}(\gamma-\eta))\cdot (\gd x_{k}(\gamma)- \gd \bar{x}_{j}(\gamma-\eta)),
		$$
		$$
		b_{3}^{(j)}(\gamma,\eta) = |\gd x_{k}(\gamma)-\gd \bar{x}_{j}(\gamma-\eta)|^{2},
		$$
		and
		$$
		b_{4}^{(j)}(\gamma,\eta) = (x_{k}(\gamma)-\bar{x}_{j}(\gamma-\eta))\cdot (\gd^{2}x_{k}(\gamma)-\gd^{2}\bar{x}_{j}(\gamma-\eta)).
		$$
		We first consider the $j=k$ terms in \eqref{phih2evo}. From $B_{1}^{(k)}$, we integrate by parts and then a symmetrization argument to obtain
		\begin{align*}
		I_{1}^{(k)} &= \intT \gd^{2}\phi_{k}(\gamma)  B_{1}^{(k)} d\gamma \\
		&= - \frac{c_{\alpha}}{2}\intT\intT \frac{ \gd^{2}\phi_{k}(\gamma) (\gd^{2}\phi_{k}(\gamma)-\gd^{2}\phi_{k}(\gamma-\eta))}{|x_k(\gamma))-x_{k}(\gamma-\eta)|^{\alpha+2}} b_{2}^{(k)}(\gamma,\eta) d\eta\\
		&= - \frac{c_{\alpha}}{4}\intT\intT \frac{|\gd^{2}\phi_{k}(\gamma) -\gd^{2}\phi_{k}(\gamma-\eta)|^{2}}{|x_{k}(\gamma)-\bar{x}_{k}(\gamma-\eta)|^{\alpha+2}} b_{2}^{(k)}(\gamma,\eta) d\eta.
		\end{align*}
		Hence, we obtain that
		$$|I_{1}^{(k)}| \leq \|\phi_{k}\|_{H^{2}}^{2} \intT \frac{|\dzdif|}{|\zdif|^{\alpha+1}} d\eta \lesssim \|\phi_{k}\|_{H^{2}}^{2}.$$
		Next, inserting $B_{2}^{(k)}$ into \eqref{phih2evo},
		\begin{align*}
		I_{2}^{(k)}&=\intT \gd^{2}\phi_{k}(\gamma) B_{2}^{(k)} d\gamma \\
		&\leq \intT \Big(|\gd^{2}\phi_{k}(\gamma)|^{2} + |\gd^{2}\phi_{k}(\gamma)||\gd^{2}\phi_{k}(\gamma-\eta)|\Big) \frac{|b_{2}^{(k)}(\gamma,\eta)|}{|\zdif|^{\alpha+2}} d\eta \\
		&\lesssim \|\phi_{k}\|_{H^{2}}^{2}.
		\end{align*}
		For $B_{3}^{(k)}$, we use that
		\bea\label{phiest} |\gd\phi_{k}(\gamma,t) -\gd\phi_{k}(\gamma-\eta,t)| \leq |\eta| \int_{0}^{1} |\gd^{2}\phi_{k}(\gamma-(s-1)\eta)| ds
		\eea
		to obtain
		\begin{align*}
		I_{3}^{(k)} &= \intT \gd^{2}\phi_{k}(\gamma) \cdot B_{3}^{(k)} d\gamma \\
		&\leq \int_{0}^{1}\intT |\gd^{2}\phi_{k}(\gamma)||\gd^{2}\phi_{k}(\gamma-(s-1)\eta)| \intT \frac{|\eta||\dzdif|^{2}}{|\zdif|^{\alpha+2}} d\gamma d\eta ds\\
		&\lesssim \|\phi_{k}\|_{H^{2}}^{2}.
		\end{align*}
		For
		$$I_{4}^{(k)} =  \intT \gd^{2}\phi_{k}(\gamma) \cdot B_{4}^{(k)} d\gamma$$
		we do the following bounds:
		\begin{align*}
		I_{4}^{(k)} &\approx  \int_{\T^{2}} \gd^{2}\phi_{k}(\gamma) \cdot \frac{ \gd\phi_{k}(\gamma) -\gd\phi_{k}(\gamma-\eta)}{|x_{k}(\gamma)-\bar{x}_{k}(\gamma-\eta)|^{\alpha+2}} b_{4}^{(k)}(\gamma,\eta) d\eta d\gamma\\
		&\lesssim \|\pg\phi\|_{C^{\frac12}}\int_{\T^{2}} \frac{|\gd^{2}\phi_{k}(\gamma)||\gd^{2}x_{k}(\gamma)-\gd^{2}\bar{x}_{k}(\gamma-\eta)|}{|\zdif|^{\alpha+\frac12}} d\eta d\gamma\\
		&\lesssim \|\phi\|_{H^2}\int_{\T}|\eta|^{-\alpha-\frac12}  \int_{\T}|\gd^{2}\phi_{k}(\gamma)|(|\gd^{2}x_{k}(\gamma)|+|\gd^{2}\bar{x}_{k}(\gamma-\eta)|)d\gamma d\eta \lesssim \|\phi\|_{H^2}^2.
		\end{align*}
		Finally, for
		$$I_{5}^{(k)} =  \intT \gd^{2}\phi_{k}(\gamma) \cdot B_{5}^{(k)} d\gamma,$$ we use \eqref{phiest} and bound as in $I_{3}^(k)$.
		This concludes the estimates coming from the term $\gd^{2} A_{21}^{(k)}$. For $j\neq k$, the terms are bounded due to the control of $\delta[x]^{-1}$ as proven earlier. For the most singular integral, $I_{1}^{(j)}$, we have
		\begin{align*}
		I_{1}^{(j)} &= \intT \gd^{2}\phi_{k}(\gamma) \cdot B_{1}^{(j)} d\gamma \\
		&= I_{11}^{(j)} + I_{12}^{(j)}
		\end{align*}
		where
		$$I_{11}^{(j)}= \intT\intT \frac{\gd^{2}\phi_{k}(\gamma)\cdot \gd^{3}\phi_{k}(\gamma)}{|x_{k}(\gamma)-\bar{x}_{j}(\gamma-\eta)|^{\alpha}}  d\eta d\gamma$$
		and
		$$ I_{12}^{(j)} = -\intT\intT \frac{\gd^{2}\phi_{k}(\gamma)\cdot \gd^{3}\phi_{k}(\gamma-\eta)}{|x_{k}(\gamma)-\bar{x}_{j}(\gamma-\eta)|^{\alpha}}d\eta d\gamma.$$
		An integration by parts in $\gamma$ and the usual estimate methods bound the $I_{11}^{(j)}$ term:
		\begin{align*}
		|I_{11}^{(j)}| &= \Big| c_{\alpha}\intT\intT \frac{|\gd^{2}\phi_{k}(\gamma)|^{2}}{|x_{k}(\gamma)-\bar{x}_{j}(\gamma-\eta)|^{\alpha+2}} |b_{2}^{(j)}(\gamma,\eta)| d\eta d\gamma\\
		&\lesssim \delta[x]^{-1+\alpha}(\|x_{j}\|_{C^{1}}+\|x_{k}\|_{C^{1}})\|\phi_{j}\|_{H^{2}}\|\phi_{k}\|_{H^{2}} \lesssim \|\phi_{j}\|_{H^{2}}\|\phi_{k}\|_{H^{2}}.
		\end{align*}
		For $I_{12}^{(j)}$, we integrate by parts in $\eta$:
		\begin{align*}
		I_{12}^{(j)} &=\intT\intT \frac{\gd^{2}\phi_{k}(\gamma)\cdot \gd^{2}\partial_{\eta}\phi_{k}(\gamma-\eta)}{|x_{k}(\gamma)-\bar{x}_{j}(\gamma-\eta)|^{\alpha}}d\eta d\gamma\\
		&= - c_{\alpha}\intT\intT \frac{\gd^{2}\phi_{k}(\gamma)\cdot \gd^{2}\phi_{k}(\gamma-\eta)}{|x_{k}(\gamma)-\bar{x}_{j}(\gamma-\eta)|^{\alpha+2}} \gd\bar{x}_{j}(\gamma-\eta)\cdot(x_{k}(\gamma)-\bar{x}_{j}(\gamma-\eta))d\eta d\gamma.
		\end{align*}
		Hence,
		$$|I_{12}^{(j)} |\lesssim \delta[x]^{-\alpha-2}\|x_{j}\|_{C^{1}}\|\phi_{j}\|_{H^{2}}\|\phi_{k}\|_{H^{2}} \lesssim \|\phi_{j}\|_{H^{2}}\|\phi_{k}\|_{H^{2}}.$$
		The rest are done similarly. Next, we now move onto the last term. Hence,
		$$ J = \intT \gd^{2}\phi_{k}(\gamma) \cdot \gd^{2}(\lambda_k\pg x_k) d\gamma = J_{1} + J_{2} + J_{3},$$
		where
		$$J_{1} = \intT \gd^{2}\phi_{k}(\gamma)\cdot \gd^{2} \lambda_{k}(\gamma)\gd\phi_{k}(\gamma) d\gamma,$$
		$$J_{2} = \intT |\gd^{2}\phi_{k}(\gamma)|^{2} \gd\lambda_{k}(\gamma)d\gamma$$
		and
		$$J_{3} = \intT \gd^{2}\phi_{k}(\gamma) \cdot \lambda_{k}(\gamma)\gd^{3}\phi_{k}(\gamma) d\gamma= -\frac{1}{2}\intT |\gd^{2}\phi_{k}(\gamma)|^{2} \gd\lambda_{k}(\gamma) d\gamma.$$
		We first examine $J_{1}$. Differentiating,
		\begin{align*}
		\gd^{2}\lambda_{k}(\gamma) &= -\gd\Big( \frac{\gd x_{k}(\gamma)}{|\gd x_{k}(\gamma)|^{2}} \cdot \gd NL_{k}(\gamma)  \Big) \\
		&= \gd\Big(\sum_{j=1}^{n} C_{1}^{(j)} + C_{2}^{(j)}\Big)
		\end{align*}
		where due to $\gd x(\gamma)\cdot \gd^{2}x(\gamma) = 0$, we have that 
		\begin{align*}
		C_{1}^{(k)} &= -\frac{\gd x_{k}(\gamma)}{A_{k}(t)} \cdot \intT \frac{\gd^{2}x_{k}(\gamma)-\gd^{2}\bar{x}_{k}(\gamma-\eta)}{|\zdif|^{\alpha}} d\eta\\
		&= \frac{\gd x_{k}(\gamma)}{A_{k}(t)} \cdot \intT \frac{ \gd^{2}\bar{x}_{k}(\gamma-\eta)}{|\zdif|^{\alpha}} d\eta
		\end{align*}
		and
		$$C_{2}^{(k)}\!\approx\!\frac{\gd x_{k}(\gamma)}{A_{k}(t)} \cdot \!\intT \frac{\gd x_{k}(\gamma) -\gd\bar{x}_{k}(\gamma-\eta)}{|\zdif|^{\alpha+2}} (\zdif)\cdot (\dzdif) d\eta.$$
		The terms in the sum are for $j\neq k$:
		$$C_{1}^{(j)} = -\frac{\gd x_{k}(\gamma)}{A_{k}(t)} \cdot \intT \frac{\gd^{2}x_{k}(\gamma)-\gd^{2}\bar{x}_{j}(\gamma-\eta)}{|x_{k}(\gamma)-\bar{x}_{j}(\gamma-\eta)|^{\alpha}} d\eta$$
		and
		$$C_{2}^{(j)}\!\approx\!\frac{\gd x_{k}(\gamma)}{A_{k}(t)} \cdot\! \intT \frac{\gd x_{k}(\gamma) -\gd\bar{x}_{j}(\gamma-\eta)}{|x_{k}(\gamma)-\bar{x}_{j}(\gamma-\eta)|^{\alpha+2}} (\zdif)\cdot (\dzdif) d\eta.$$
		
		Since $J_{1} \leq \|\phi_{k}\|_{H^{2}}\|\gd^{2}\lambda_{k}\|_{L^{2}}$, it suffices to prove that $\|\gd^{2}\lambda_{k}\|_{L^{2}} \lesssim 1+  \|\phi_{k}\|_{H^{2}}$. The terms in the sum $j\neq k$ can be bounded by the control of $\delta[x]$. The more singular terms remaining are from $C_{1}^{(k)}$ and $C_{2}^{(k)}$. First,
		$$\gd C_{1}^{(k)} = C_{11}^{(k)} + C_{12}^{(k)} + C_{13}^{(k)}$$ where
		$$C_{11}^{(k)} = \frac{\gd^{2}x_{k}(\gamma)}{A_{k}(t)} \cdot \intT \frac{ \gd^{2}\bar{x}_{k}(\gamma-\eta)}{|\zdif|^{\alpha}} d\eta,$$
		$$C_{12}^{(k)} \!= \! c_{\alpha}\frac{\gd^{2}x_{k}(\gamma)}{A_{k}(t)} \cdot \!\intT\! \frac{\gd \bar{x}_{k}(\gamma-\eta)}{|\zdif|^{\alpha+2}} (\zdif)\cdot(\dzdif) d\eta$$
		and
		$$C_{13}^{(k)} = \frac{\gd x_{k}(\gamma)}{A_{k}(t)} \cdot \intT \frac{ \gd^{3}\bar{x}_{k}(\gamma-\eta)}{|\zdif|^{\alpha}} d\eta.$$
		For $C_{11}$, we have that
		$$\|C_{11}^{(k)}\|_{L^{2}} \leq  \|x\|_{H^{2}}\Big\| \intT \frac{ \gd^{2}\bar{x}_{k}(\gamma-\eta)}{|\eta|^{\alpha}} d\eta \Big\|_{L^{\infty}}\leq  \|x\|_{H^{2}}^{2}$$
		where we used Young's convolution inequality in the final step for $\alpha < 1/2$. For $C_{12}^{(k)}$, the bound for $\alpha < 1/3$ is given by 
		$$\|C_{12}^{(k)}\|_{L^{2}} \leq \|x_{k}\|_{H^{2}}^{2}$$ by the usual methods involving \eqref{useful} and Young's convolution inequality as above. Finally for $C_{13}^{(k)}$, we integrate by parts:
		\begin{align*}
		|C_{13}^{(k)}| =& \Big|\frac{\gd^{2}x_{k}(\gamma)}{A_{k}(t)} \cdot \intT \frac{ \partial_{\eta}\gd^{2}\bar{x}_{k}(\gamma-\eta)}{|\zdif|^{\alpha}} d\eta\Big|\\
		\approx& \Big|\frac{\gd x_{k}(\gamma)}{A_{k}(t)} \cdot \intT \frac{ \gd^{2}\bar{x}_{k}(\gamma-\eta)(\zdif)\cdot \gd \bar{x}_{k}(\gamma-\eta)}{|\zdif|^{\alpha+2}} d\eta\Big|\\
		=& \frac{1}{A_{k}(t)} \Big| \intT (\dzdif)\cdot\gd^{2}\bar{x}_{k}(\gamma-\eta)\\
		&\qquad\qquad\qquad \frac{ (\zdif)\cdot (\gd x_{k}(\gamma)-\gd \bar{x}_{k}(\gamma-\eta))}{|\zdif|^{\alpha+2}} d\eta\Big|\\
		\lesssim &\intT \frac{|\gd^{2}\bar{x}_{k}(\gamma-\eta)||\gd\bar{x}_{k}(\gamma-\eta)|}{|\eta|^{\alpha+2/3}} d\eta.
		\end{align*}
		Hence, $\|C_{13}^{(k)}\|_{L^{2}} \lesssim 1$ by Holder's inequality and the integrability of $|\eta|^{-\alpha-2/3}$ for $\alpha < 1/3$.
		Next, $$\gd C_{2}^{(k)} = C_{21}^{(k)} + C_{22}^{(k)} + C_{23}^{(k)} + C_{24}^{(k)}$$ where
		\begin{multline*}
		C_{21}^{(k)}\! \\\approx \!\frac{\gd^{2} x_{k}(\gamma)}{A_{k}(t)}\! \cdot \!\!\!\intT \!\frac{\gd x_{k}(\gamma) -\gd\bar{x}_{k}(\gamma-\eta)}{|\zdif|^{\alpha+2}}\! (\zdif)\!\!\cdot\! \! (\dzdif) d\eta\\
		= \!\frac{\gd^{2} x_{k}(\gamma)}{A_{k}(t)}\! \cdot \!\!\!\intT \!\frac{-\gd\bar{x}_{k}(\gamma-\eta)}{|\zdif|^{\alpha+2}}\! (\zdif)\!\!\cdot\! \! (\dzdif) d\eta
		\end{multline*}
		and similarly
		$$C_{22}^{(k)} \!\approx \!\frac{\gd x_{k}(\gamma)}{A_{k}(t)}\! \cdot \!\!\!\intT \!\frac{-\gd^{2}\bar{x}_{k}(\gamma-\eta)}{|\zdif|^{\alpha+2}}\! (\zdif)\!\!\cdot\! \! (\dzdif) d\eta$$
		and
		$$C_{23}^{(k)} \!\approx \!\frac{\gd x_{k}(\gamma)}{A_{k}(t)}\! \cdot \!\!\!\intT \!\frac{\gd x_{k}(\gamma)-\gd\bar{x}_{k}(\gamma-\eta)}{|\zdif|^{\alpha+2}}\! (\zdif)\!\!\cdot\! \! (\gd^{2} x(\gamma)-\gd^{2}\bar{x}_{k}(\gamma-\eta)) d\eta$$
		and
		$$C_{24}^{(k)} \!\approx \!\frac{\gd x_{k}(\gamma)}{A_{k}(t)}\! \cdot \!\!\!\intT \!\frac{\gd x_{k}(\gamma)-\gd\bar{x}_{k}(\gamma-\eta)}{|\zdif|^{\alpha+2}}|\dzdif|^{2}d\eta$$
		and
		\begin{multline*}
		C_{25}^{(k)} \approx \frac{\gd x_{k}(\gamma)}{A_{k}(t)} \\\cdot \intT \!\frac{\gd x_{k}(\gamma)-\gd\bar{x}_{k}(\gamma-\eta)}{|\zdif|^{\alpha+4}}((\zdif)\cdot (\dzdif))^{2}d\eta.
		\end{multline*}
		By the usual methods, it can be seen that $\|C_{2i}\|_{L^{2}} \lesssim \|x_{k}\|_{H^{2}}$. For example, for $C_{25}$, we use the equivalence
		$$\gd x_{k}(\gamma) \cdot (\gd x_{k}(\gamma)-\gd\bar{x}_{k}(\gamma-\eta)) = \frac12|\gd x_{k}(\gamma)-\gd\bar{x}_{k}(\gamma-\eta)|^{2}$$
		due to the constant parametrization to obtain that
		\begin{align*}
		\|C_{25}\|_{L^{2}} &\lesssim \intT \frac{|\gd x_{k}(\gamma)-\gd\bar{x}_{k}(\gamma-\eta)|^{4}}{| x_{k}(\gamma)-\bar{x}_{k}(\gamma-\eta)|^{\alpha+2}} d\eta\\
		&\lesssim \intT \frac{1}{| x_{k}(\gamma)-\bar{x}_{k}(\gamma-\eta)|^{\alpha+2/3}} d\eta\lesssim 1
		\end{align*}
		where in the first step, we used \eqref{useful} and in the second step, we used the control of $\supF$.
		Hence, in summary, \eqref{phih2evo} is given by
		$$\frac{d}{dt} \sum_{j=1}^{n}\|\phi_{j}\|_{H^{2}}^{2} \lesssim \sum_{j=1}^{n}\|\phi_j\|_{H^{2}}^{2}$$
		where the implicit constant depends continuously on $\|F(x_{j})\|_{L^{\infty}}$, $\delta[x]^{-1}$, $\|x_{j}\|_{H^{2}}$ and $\alpha$.
	\end{proof}
	
	Next we give uniqueness for the system \eqref{SQGnpatch}.
	
	\begin{prop}
		Suppose $\{x_{j}(\gamma,t)\}_{j=1,\ldots,n}$ and $\{y_{j}(\gamma,t)\}_{j=1,\ldots,n}$ are both  solutions to the contour equation \eqref{SQGnpatch} in $C([0,T],H^2)$ with initial data $x_{j}(\gamma,0)=y_{j}(\gamma,0)$ and $z_j=x_j-y_j$. Then,
		$$ \frac{d}{dt} \Big(\sum_{j=1}^{n} \|z_{j}\|_{L^{2}}^{2}\Big) \lesssim  \sum_{j=1}^{n} \|z_{j}\|_{L^{2}}^{2}$$
		where the implicit constant depends continuously on $\delta[x]^{-1}$, $\delta[y]^{-1}$, $\|F(x_{j})\|_{L^{\infty}}$, $\|F(y_{j})\|_{L^{\infty}}$, $\|x_{j}\|_{H^{2}}$, $\|y_{j}\|_{H^{2}}$ and $\alpha$. Above inequality together with Gronwall's lemma provides $x_j=y_j$ on $[0,T]$.
	\end{prop}
	
	\begin{proof}
		Define
		$$z_{j}(\gamma, t) = x_{j}(\gamma,t) - y_{j}(\gamma,t),$$
		for each $j = 1, \ldots, n$. Then,
		\begin{multline*}
		\frac{1}{2}\frac{d}{dt} \|z_{k}\|_{L^{2}}^{2} = \intT \partial_{t}z_{k}(\gamma)\cdot z_{k}(\gamma) d\gamma\\
		= \sum_{j=1}^{n}\int_{\T^{2}}  d\eta d\gamma \Big( \frac{\partial_{\gamma}x_{k}(\gamma,t)-\partial_{\gamma}\bar{x}_{j}(\gamma-\eta,t)}{|x_{k}(\gamma,t)-\bar{x}_{j}(\gamma-\eta,t)|^{\alpha}} - \frac{\partial_{\gamma}y_{k}(\gamma,t)-\partial_{\gamma}\bar{y}_{j}(\gamma-\eta,t)}{|y_{k}(\gamma,t)-\bar{y}_{j}(\gamma-\eta,t)|^{\alpha}} \Big) \cdot z_{k}(\gamma)\\+\Big(\frac{\partial_{\gamma}x_{k}(\gamma,t)-\partial_{\gamma}x_{j}(\gamma-\eta,t)}{|x_{k}(\gamma,t)-x_{j}(\gamma-\eta,t)|^{\alpha}} - \frac{\partial_{\gamma}y_{k}(\gamma,t)-\partial_{\gamma}y_{j}(\gamma-\eta,t)}{|y_{k}(\gamma,t)-y_{j}(\gamma-\eta,t)|^{\alpha}}    \Big)\cdot z_{k}(\gamma)\\
		\eqdef \sum_{j=1}^{n}K_{j} + L_{j}.
		\end{multline*}
		Consider the terms $K_{j}$:
		$$K_{j} = K_{j1} + K_{j2}$$
		where
		$$K_{j1} = \int_{\T^{2}}  \Big( \frac{\partial_{\gamma}x_{k}(\gamma,t)-\partial_{\gamma}\bar{x}_{j}(\gamma-\eta,t)}{|x_{k}(\gamma,t)-\bar{x}_{j}(\gamma-\eta,t)|^{\alpha}} - \frac{\partial_{\gamma}y_{k}(\gamma,t)-\partial_{\gamma}\bar{y}_{j}(\gamma-\eta,t)}{|x_{k}(\gamma,t)-\bar{x}_{j}(\gamma-\eta,t)|^{\alpha}} \Big) \cdot z_{k}(\gamma)d\eta d\gamma $$
		and
		$$K_{j2} = \int_{\T^{2}} \Big( \frac{\partial_{\gamma}y_{k}(\gamma,t)-\partial_{\gamma}\bar{y}_{j}(\gamma-\eta,t)}{|x_{k}(\gamma,t)-\bar{x}_{j}(\gamma-\eta,t)|^{\alpha}} - \frac{\partial_{\gamma}y_{k}(\gamma,t)-\partial_{\gamma}\bar{y}_{j}(\gamma-\eta,t)}{|y_{k}(\gamma,t)-\bar{y}_{j}(\gamma-\eta,t)|^{\alpha}} \Big) \cdot z_{k}(\gamma) d\eta d\gamma .$$
		For $K_{j1}$, we write $K_{j1} = K_{j11} + K_{j12}$ where
		$$K_{j11} = \int_{\T^{2}}  \Big( \frac{\partial_{\gamma}x_{k}(\gamma)-\partial_{\gamma}y_{k}(\gamma)}{|x_{k}(\gamma)-\bar{x}_{j}(\gamma-\eta)|^{\alpha}} \Big) \cdot z_{k}(\gamma)d\eta d\gamma$$
		and
		$$K_{j12} = \int_{\T^{2}}  \Big( \frac{\partial_{\gamma}\bar{x}_{j}(\gamma-\eta)-\partial_{\gamma}\bar{y}_{j}(\gamma-\eta)}{|x_{k}(\gamma)-\bar{x}_{j}(\gamma-\eta)|^{\alpha}} \Big) \cdot z_{k}(\gamma)d\eta d\gamma.$$
		For the more singular terms in which $j=k$, we obtain that
		\begin{align*}
		K_{k11} &= \int_{\T^{2}}  \Big( \frac{\partial_{\gamma}x_{k}(\gamma)-\partial_{\gamma}y_{k}(\gamma)}{|x_{k}(\gamma)-\bar{x}_{k}(\gamma-\eta)|^{\alpha}} \Big) \cdot z_{k}(\gamma)d\eta d\gamma \\
		&=\frac{1}{2}\intT \gd(|z_{k}(\gamma)|^{2}) d\gamma \intT \frac{1}{|\zdif|^{\alpha}} d\eta\\
		&= -\frac{c_{\alpha}}{2}\intT |z_{k}(\gamma)|^{2} d\gamma \intT \frac{(\zdif)\cdot(\dzdif)}{|\zdif|^{\alpha+2}} d\eta.
		\end{align*}
		Hence, by the usual methods,
		$$|K_{k11}| \lesssim \|z_{k}\|_{L^{2}}^{2}.$$
		Similarly, we can control $K_{k12}$ by the same bounds:
		\begin{align*}
		K_{k12} &= \int_{\T^{2}}  \Big( \frac{\partial_{\gamma}\bar{x}_{k}(\gamma-\eta)-\partial_{\gamma}\bar{y}_{k}(\gamma-\eta)}{|x_{k}(\gamma)-\bar{x}_{k}(\gamma-\eta)|^{\alpha}} \Big) \cdot z_{k}(\gamma)d\eta d\gamma\\
		&= -\int_{\T^{2}}  \Big( \frac{\bar{z}_{k}(\gamma-\eta)}{|x_{k}(\gamma)-\bar{x}_{k}(\gamma-\eta)|^{\alpha}} \Big) \cdot \gd z_{k}(\gamma)d\eta d\gamma\\ &- c_{\alpha}\int_{\T^{2}}  \frac{(\zdif)\cdot (\dzdif)}{|x_{k}(\gamma)-\bar{x}_{k}(\gamma-\eta)|^{\alpha+2}}\bar{z}(\gamma-\eta)\cdot  z_{k}(\gamma)   d\eta d\gamma\\
		&=-\int_{\T^{2}}  \Big( \frac{z_{k}(\gamma)}{|x_{k}(\gamma)-\bar{x}_{k}(\gamma-\eta)|^{\alpha}} \Big) \cdot \gd \bar{z}_{k}(\gamma-\eta)d\eta d\gamma\\ &- c_{\alpha} \int_{\T^{2}}  \frac{(\zdif)\cdot (\dzdif) }{|x_{k}(\gamma)-\bar{x}_{k}(\gamma-\eta)|^{\alpha+2}}  \bar{z}_{k}(\gamma-\eta)\cdot z_{k}(\gamma)d\eta d\gamma\\
		& = -\frac{c_{\alpha}}{2}\int_{\T^{2}}  \frac{(\zdif)\cdot (\dzdif)}{|x_{k}(\gamma)-\bar{x}_{k}(\gamma-\eta)|^{\alpha+2}}  \bar{z}(\gamma-\eta)\cdot z_{k}(\gamma) d\eta d\gamma
		\end{align*}
		where we have integrating by parts in $\gamma$ and then performed a change of variables and used the equality $\bar{u}\cdot \bar{v} = u\cdot v$ for vectors $u$ and $v$. Next,
		\begin{align*}
		|K_{k12}| &\lesssim \int_{\T^{2}}  \frac{|\dzdif|}{|x_{k}(\gamma)-\bar{x}_{k}(\gamma-\eta)|^{\alpha+1}}  |\bar{z}(\gamma-\eta)||z_{k}(\gamma)| d\eta d\gamma\\
		&\lesssim \|z_{k}\|_{L^{2}}  \Big\|\int_{\T}  |\eta|^{-2/3-\alpha} |\bar{z}(\gamma-\eta)| d\eta\Big\|_{L^{2}}\\
		&\lesssim \|z_{k}\|_{L^{2}}^{2}
		\end{align*}
		where we use the control of $\supF$ in the second line and Young's inequality in the third line.
		Next, for the term $K_{k2}$, we use the fact that $|1-x^{s}| \leq |1-x|$ for $0\leq s < 1$ to obtain:
		\begin{align*}
		\Big|\frac{1}{|x_{k}(\gamma)-\bar{x}_{k}(\gamma-\eta)|^{\alpha}}-\frac{1}{|y_{k}(\gamma)-\bar{y}_{k}(\gamma-\eta)|^{\alpha}}\Big| &\lesssim \Big|1-\frac{|x_{k}(\gamma)-\bar{x}_{k}(\gamma-\eta)|^{\alpha}}{|y_{k}(\gamma)-\bar{y}_{k}(\gamma-\eta)|^{\alpha}}\Big| |\eta|^{-\alpha}\\
		&\lesssim \Big|1-\frac{|x_{k}(\gamma)-\bar{x}_{k}(\gamma-\eta)|}{|y_{k}(\gamma)-\bar{y}_{k}(\gamma-\eta)|}\Big||\eta|^{-\alpha}.
		\end{align*}
		Hence,
		\begin{multline*}
		|K_{k2}| \lesssim \int_{\T^{2}} |\partial_{\gamma}y_{k}(\gamma)-\partial_{\gamma}\bar{y}_{k}(\gamma-\eta)|\Big|1-\frac{|x_{k}(\gamma)-\bar{x}_{k}(\gamma-\eta)|}{|y_{k}(\gamma)-\bar{y}_{k}(\gamma-\eta)|}\Big||\eta|^{-\alpha} |z_{k}(\gamma)| d\eta d\gamma\\
		\lesssim \int_{\T^{2}} |y_{k}(\gamma)-\bar{y}_{k}(\gamma-\eta)|^{1/3}\Big|1-\frac{|x_{k}(\gamma)-\bar{x}_{k}(\gamma-\eta)|}{|y_{k}(\gamma)-\bar{y}_{k}(\gamma-\eta)|}\Big||\eta|^{-\alpha} |z_{k}(\gamma)| d\eta d\gamma\\
		\lesssim \int_{\T^{2}} |y_{k}(\gamma)-\bar{y}_{k}(\gamma-\eta)|^{-2/3}\Big||y_{k}(\gamma)-\bar{y}_{k}(\gamma-\eta)|-|x_{k}(\gamma)-\bar{x}_{k}(\gamma-\eta)|\Big|\\\cdot|\eta|^{-\alpha} |z_{k}(\gamma)| d\eta d\gamma\\
		\lesssim \int_{\T^{2}}|y_{k}(\gamma)-\bar{y}_{k}(\gamma-\eta)-x_{k}(\gamma)+\bar{x}_{k}(\gamma-\eta)||\eta|^{-\alpha-2/3} |z_{k}(\gamma)| d\eta d\gamma\\
		\lesssim \|z_{k}\|_{L^{2}}^{2}
		\end{multline*}
		using the usual Young's inequality and Holder's inequality arguments in the final line. For $K_{j}$ with $j\neq k$, we utilize the control of $\delta[x]^{-1}$ and $\delta[y]^{-1}$ to control the terms $K_{j1}$ and $K_{j2}$. For example, for $K_{j12}$, we integrate by parts in $\eta$:
		\begin{align*}
		K_{j12} &= \int_{\T^{2}}  \Big( \frac{\partial_{\gamma}\bar{x}_{j}(\gamma-\eta)-\partial_{\gamma}\bar{y}_{j}(\gamma-\eta)}{|x_{k}(\gamma)-\bar{x}_{j}(\gamma-\eta)|^{\alpha}} \Big) \cdot z_{k}(\gamma)d\eta d\gamma\\ 
		&= \int_{\T^{2}}  \Big( \frac{-\partial_{\eta}\bar{z}_{j}(\gamma-\eta)}{|x_{k}(\gamma)-\bar{x}_{j}(\gamma-\eta)|^{\alpha}} \Big) \cdot z_{k}(\gamma)d\eta d\gamma\\
		&= \int_{\T^{2}}  \Big( \frac{\bar{z}_{j}(\gamma-\eta)(x_{k}(\gamma)-\bar{x}_{j}(\gamma-\eta))\cdot \gd\bar{x}_{j}(\gamma-\eta) }{|x_{k}(\gamma)-\bar{x}_{j}(\gamma-\eta)|^{\alpha+2}} \Big) \cdot z_{k}(\gamma)d\eta d\gamma\\
		&\leq \delta[x]^{-1-\alpha}\|x_{j}\|_{C^{1}} \|z_{j}\|_{L^{2}} \|z_{k}\|_{L^{2}} \lesssim  \|z_{j}\|_{L^{2}}^{2}+  \|z_{k}\|_{L^{2}}^{2}.
		\end{align*}
		For $K_{j2}$, we have that 
		\begin{align*}
		\Big|\frac{1}{|x_{k}(\gamma)-\bar{x}_{j}(\gamma-\eta)|^{\alpha}}-\frac{1}{|y_{k}(\gamma)-\bar{y}_{j}(\gamma-\eta)|^{\alpha}}\Big| &\lesssim \Big|1-\frac{|x_{k}(\gamma)-\bar{x}_{j}(\gamma-\eta)|}{|y_{k}(\gamma)-\bar{y}_{j}(\gamma-\eta)|^{\alpha}}^{\alpha}\Big| \delta[x]^{-\alpha}\\
		&\lesssim \Big|1-\frac{|x_{k}(\gamma)-\bar{x}_{j}(\gamma-\eta)|}{|y_{k}(\gamma)-\bar{y}_{j}(\gamma-\eta)|}\Big|,
		\end{align*}
		and hence,
		\begin{multline*}
		|K_{j2}| \lesssim \int_{\T^{2}} |\partial_{\gamma}y_{k}(\gamma)-\partial_{\gamma}\bar{y}_{j}(\gamma-\eta)|\Big|1-\frac{|x_{k}(\gamma)-\bar{x}_{j}(\gamma-\eta)|}{|y_{k}(\gamma)-\bar{y}_{j}(\gamma-\eta)|}\Big||z_{k}(\gamma)| d\eta d\gamma\\
		\lesssim \int_{\T^{2}}(|y_{j}\|_{C^{1}}+|y_{k}\|_{C^{1}})\Big||y_{k}(\gamma)-\bar{y}_{j}(\gamma-\eta)|-|x_{k}(\gamma)-\bar{x}_{j}(\gamma-\eta)|\Big||z_{k}(\gamma)| d\eta d\gamma\\
		\lesssim \int_{\T^{2}}|y_{k}(\gamma)-\bar{y}_{j}(\gamma-\eta)-x_{k}(\gamma)+\bar{x}_{j}(\gamma-\eta)||z_{k}(\gamma)| d\eta d\gamma\\
		\lesssim \|z_{j}\|_{L^{2}}\|z_{k}\|_{L^{2}}+\|z_{k}\|_{L^{2}}^{2}\lesssim \|z_{j}\|_{L^{2}}^{2}+\|z_{k}\|_{L^{2}}^{2}.
		\end{multline*}
		The remaining terms for the estimates are less singular and can be bounded similarly or more easily.
	\end{proof}
	
	\section{Proof of Theorem \ref{FiniteTimeSingularity}}
	
	For the range of $\alpha \in (0,1/3)$, we construct solutions that exhibit singularities in finite time. For the purposes of this section, we introduce the parameter $\beta = \alpha/2$ for $\beta \in (0,1/6)$ and we change the notation for the velocity function to $u(t,x) = (u_{1}(t,x), u_{2}(t,x))$ and the spatial coordinate to $x = (x_{1}, x_{2})$.

	We consider initial data $\theta_{0}(x)$ composed of two odd symmetric patches:
	$$\theta_{0}(x) = \chi_{D_{0}}(x) - \chi_{\tilde{D_{0}}}(x)$$
	where  $\chi_{D_{0}}(x_{1},x_{2}) = - \chi_{\tilde{D}_{0}}(-x_{1},x_{2})$ are the characteristic functions on domains $D_{0}$ and $ \tilde{D_{0}}$ with $H^{2}$ boundaries in the upper half plane such that $D$ and $ \tilde{D_{0}}$ are symmetric about the vertical axis. Additionally, we impose the physical constraint that $(2\epsilon, 3) \times (0,3)\subset D_{0}\subset (\epsilon,4)\times(0,4) $ for small enough $\epsilon$  to be determined later in the construction. By Theorem \ref{localmodifiedSQG}, there exists a unique patch solution $$\theta(t,x) = \chi_{D(t)}(x) - \chi_{\tilde{D}(t)}(x)$$ for some time $ T > 0$ with patch boundaries in $H^{2}$. The patches remain odd symmetric in time by the patch evolution equation.
	
	Let $T_{\ast}$ be the maximal time of existence of this unique solution to the SQG patch problem for $\beta \in (0,1/6)$. We will demonstrate that indeed $T_{\ast} < \infty$ via a contradiction argument. To see the singularity, let $K(t)$ be the trapezoid
	$$K(t) = \{ (x_{1},x_{2}) \ | \ x_{1}\in (X(t), a) \text{ and } x_{2} \in (0, mx_{1})\}$$
	for $$X(t) = \Big( (3\epsilon)^{2\beta} - 2\beta C t\Big)^{\frac{1}{2\beta}}.$$
	Then, $X(t)$ satisfies $X'(t) = -CX(t)^{1-2\beta}$, $X(0) = 3\epsilon$ and $X(T) = 0$ for $T = \frac{(3\epsilon)^{2\beta}}{2\beta C}$. The positive constants $C$, $m$ and $a$ are to be determined later in the construction. The estimates on $u(x,t)$  given by the following lemmas below shall demonstrate that the trapezoid remains within the patch. Due to the definition of $X(t)$, the trapezoid touches the origin in finite time, and therefore, a finite time singularity must occur.

Let us recall the integral representations of the horizontal and vertical velocity functions:
	$$ u_{i} = (-1)^{i} \int_{\mathbb{R}_{+}\times\mathbb{R}_{+}} K_{i}(x,y)\theta(y) dy$$
	where for $\bar{y} = (y_{1},-y_{2})$ and $\tilde{y}= (-y_{1},y_{2})$
$$
K_{1}(x,y) = \frac{y_{2}-x_{2}}{|x-y|^{2+2\beta}} -  \frac{y_{2}-x_{2}}{|x-\tilde{y}|^{2+2\beta}} - \frac{y_{2}+x_{2}}{|x+y|^{2+2\beta}}+ \frac{y_{2}+x_{2}}{|x-\bar{y}|^{2+2\beta}} = \sum_{j=1}^{4} K_{1j}(x,y)
$$
and
$$
K_{2}(x,y) = \frac{y_{1}-x_{1}}{|x-y|^{2+2\beta}} + \frac{y_{1}+x_{1}}{|x-\tilde{y}|^{2+2\beta}}  - \frac{y_{1}+x_{1}}{|x+y|^{2+2\beta}}- \frac{y_{1}-x_{1}}{|x-\bar{y}|^{2+2\beta}} = \sum_{j=1}^{4} K_{2j}(x,y).
$$
The following estimates hold on the upper right quadrant:
\begin{equation}\label{K1}
K_{1}(x,y) \geq K_{11}(x,y) + K_{12}(x,y) \text{; } \text{sgn}(y_{2}-x_{2})(K_{11}(x,y) + K_{12}(x,y)) \geq 0
\end{equation}
and
\begin{equation}\label{K2}
K_{2}(x,y) \geq K_{21}(x,y) + K_{24}(x,y) \text{; } \text{sgn}(y_{1}-x_{1})(K_{21}(x,y) + K_{24}(x,y)) \geq 0.
\end{equation}
Using these estimates, we can separately estimate the good and bad parts of $u_{1}$ and $u_{2}$ which are given by the decomposition:
\bea\label{u1goodintegral}
u_{1}^{good} =  - \int_{\mathbb{R}_{+}\times (x_{2},\infty)} K_{1}(x,y)\theta(y) dy,
\eea
\bea\label{u1badintegral}u_{1}^{bad} =  - \int_{\mathbb{R}_{+}\times (0,x_{2})} K_{1}(x,y)\theta(y) dy,
\eea
\bea\label{u2goodintegral}u_{2}^{good} = \int_{ (x_{1},\infty)\times\mathbb{R}_{+}} K_{2}(x,y)\theta(y) dy
\eea
and
\bea\label{u2badintegral}u_{2}^{bad} =  \int_{(0,x_{1})\times\mathbb{R}_{+}} K_{2}(x,y)\theta(y) dy.
\eea

The estimates for $u_{1}^{bad}$ and $u_{2}^{bad}$ follow similarly as in \cite{KRYZ}, with modifications for a general slope of $m$:
	
	\begin{lemma}\label{u1bad}
		For $x_{2}\leq mx_{1}$, we have
		$$ u_{1}^{bad}\leq \frac{1}{\beta}\Big( \frac{1}{1-2\beta} - (1+m^{2})^{-\beta}\Big) x_{1}^{1-2\beta}.$$
	\end{lemma}
	
	\begin{proof}
By \eqref{K1} and \eqref{u1badintegral},  and since $0\leq \theta \leq 1$, we have the estimate
$$u_{1}^{bad} \leq - \int_{0}^{2x_{1}} dy_{1}\int_{0}^{x_{2}} dy_{2} \frac{y_{2}-x_{2}}{|x-y|^{2+2\beta}} . $$
Integrating, we obtain
\begin{align*}
u_{1}^{bad} &\leq \frac{1}{2\beta} \int_{0}^{2x_{1}} dy_{1} \Big(\frac{1}{(x_{1}-y_{1})^{2\beta}} -  \frac{1}{((x_{1}-y_{1})^{2} + x_{2}^{2})^{\beta}}\Big)\\
&= \frac{1}{\beta} \int_{0}^{x_{1}} dy_{1} \Big(\frac{1}{(x_{1}-y_{1})^{2\beta}} -  \frac{1}{((x_{1}-y_{1})^{2} + x_{2}^{2})^{\beta}}\Big)\\
& = \frac{1}{\beta} \Big(\frac{1}{1-2\beta}x_{1}^{1-2\beta} -  \int_{0}^{x_{1}} dy_{1} \frac{1}{((x_{1}-y_{1})^{2} + x_{2}^{2})^{\beta}}\Big).
\end{align*}
For the second integral, we use the inequality $x_{2} \leq mx_{1}$ to obtain that
\begin{align*}
u_{1}^{bad} &\leq \frac{1}{\beta} \Big(\frac{1}{1-2\beta}x_{1}^{1-2\beta} -  \frac{x_{1}^{1-2\beta}}{(1+m^{2})^{\beta}}\Big).
\end{align*}
\end{proof}
	
\begin{lemma}\label{u2bad}
For $mx_{1} \leq x_{2}$, we have
$$ u_{2}^{bad}\geq -\frac{1}{\beta}\Big( \frac{1}{1-2\beta} - (1+\frac{1}{m^{2}})^{-\beta}\Big) x_{2}^{1-2\beta}.$$
\end{lemma}

\begin{proof}
We follow similarly to the proof for $u_{1}^{bad}$. By \eqref{K2}, \eqref{u2badintegral} and since $0 \leq \theta \leq 1$, we have that
\begin{align*}
u_{2}^{bad} &\geq \int_{0}^{\infty} dy_{1}\int_{0}^{x_{1}} dy_{1} \Big(\frac{y_{1}-x_{1}}{|x-y|^{2+2\beta}} -  \frac{y_{1}-x_{1}}{|x-\bar{y}|^{2+2\beta}}\Big)\\
&= \int_{0}^{2x_{2}} dy_{1}\int_{0}^{x_{1}} dy_{1} \frac{y_{1}-x_{1}}{|x-y|^{2+2\beta}}.
\end{align*}
Next,
\begin{align*}
u_{2}^{bad} &\geq -\frac{1}{2\beta} \int_{0}^{2x_{2}} dy_{2} \Big(\frac{1}{(x_{2}-y_{2})^{2\beta}} -  \frac{1}{((x_{2}-y_{2})^{2} + x_{1}^{2})^{\beta}}\Big)\\
& = -\frac{1}{\beta} \Big(\frac{1}{1-2\beta}x_{2}^{1-2\beta} -  \int_{0}^{x_{2}} dy_{2} \frac{1}{((x_{2}-y_{2})^{2} + x_{1}^{2})^{\beta}}\Big)\\
 &\geq -\frac{1}{\beta} \Big(\frac{1}{1-2\beta}x_{2}^{1-2\beta} -  \frac{x_{2}^{1-2\beta}}{(1+\frac{1}{m^{2}})^{\beta}}\Big).
\end{align*}
\end{proof}

For the extended range $\beta \in (0,1/6)$, we need new control on the good parts of the patch velocity.
	
\begin{lemma}\label{u1good}
We have the following estimate for $u_{1}^{good}(t,x)$ for $x_{1} < \delta_{\beta} << 1$:
		
\begin{multline*}
u_{1}^{good}(t,x) \leq - \frac1{\beta2^{2\beta}}\Big(\frac{1-(m^{2}+1)^{-\beta}}{(1-2\beta)} + \frac{1}{m^{2\beta}(1+\frac{4}{m^{2}})^{1+\beta}}\Big) x_{1}^{1-2\beta}\\ - \frac{1}{2\beta}  \Big( \frac{1}{(9+ m^{2})^{\beta}} - \frac{1}{(4 + 4m^{2})^{\beta}}  \Big)x_{1}^{1-2\beta} + O(x_{1}).
		\end{multline*}
	\end{lemma}
	\begin{proof}
		We have the estimate from \eqref{K1} and \eqref{u1goodintegral}
		$$u_{1}^{good}(x) \leq - \int_{A(x)}\frac{y_{2}-x_{2}}{|x-y|^{2+2\beta}} dy + \int_{A(x)+(2x_{1},0)}\frac{y_{2}-x_{2}}{|x-y|^{2+2\beta}} dy$$
		where $$A(x) = \{ (y_{1},y_{2}) \ | \ y_{1}\in (x_{1},x_{1}+a), y_{2}\in (x_{2}, x_{2} + m(y_{1}-x_{1}))   \}.$$ 
		Comparing the domains of the two integrals, we can reduce this inequality to
		$$u_{1}^{good}(x) \leq - \int_{A_{1}}\frac{y_{2}-x_{2}}{|x-y|^{2+2\beta}} dy + \int_{A_{2}}\frac{y_{2}-x_{2}}{|x-y|^{2+2\beta}} dy$$
		where $$ A_{1} = \{ (y_{1} + \frac{1}{m}(y_{2}-x_{2}),y_{2}) \ | \ y_{1}\in (x_{1},3x_{1}) \text{ and } y_{2}\in (x_{2}, x_{2} + ma)\} $$
		and 
		$$A_{2} = (x_{1}+a, 3x_{1}+a) \times (x_{2}, x_{2}+ma) .$$
		Now, we can bound $A_{2}$ by $O(x_{1})$:
		\begin{align*}
		\int_{A_{2}}\frac{y_{2}-x_{2}}{|x-y|^{2+2\beta}} dy &= \int_{x_{1}+a}^{3x_{1}+a} dy_{1} \int_{x_{2}}^{x_{2}+ma} dy_{2}\frac{y_{2}-x_{2}}{|x-y|^{2+2\beta}}\\
		&=\int_{x_{1}+a}^{3x_{1}+a} dy_{1} \frac{1}{2\beta} \Big( \frac{1}{(y_{1}-x_{1})^{2\beta}} -  \frac{1}{(m^{2}a^{2}+(y_{1}-x_{1})^{2})^{\beta}}\Big)\\
		&\leq \frac{x_{1}}{\beta}\Big( \frac{1}{a^{2\beta}}- \frac{1}{((m^{2}+1)a^{2})^{\beta}}  \Big)= O(x_{1}).
		\end{align*}
		
		To bound $A_{1}$, we consider two disjoint subsets of $A_{1}$. Let
		$A_{11} = A_{1} \cap (x_{1},3x_{1})\times \mathbb{R}$ and $A_{12} = A_{1} \cap \mathbb{R} \times (x_{2}+2mx_{1},\infty)$. On the triangle $A_{11}$, we have
		\begin{align*}
		\int_{A_{11}}\frac{y_{2}-x_{2}}{|x-y|^{2+2\beta}} dy &= \int_{x_{1}}^{3x_{1}} dy_{1} \int_{x_{2}}^{x_{2}+m(y_{1}-x_{1})} dy_{2}\frac{y_{2}-x_{2}}{|x-y|^{2+2\beta}}\\
		&= \int_{x_{1}}^{3x_{1}} dy_{1} \frac{1}{2\beta}\Big( \frac{1}{(y_{1}-x_{1})^{2\beta}} - \frac{1}{((m^{2}+1)(y_{1}-x_{1})^{2})^{\beta}}  \Big)\\
		&= \frac{1-(m^{2}+1)^{-\beta}}{2\beta}\int_{x_{1}}^{3x_{1}} \frac{dy_{1}}{(y_{1}-x_{1})^{2\beta}} \\
		&= \frac{2^{1-2\beta}(1-(m^{2}+1)^{-\beta})}{2\beta(1-2\beta)} x_{1}^{1-2\beta}.
		\end{align*}
		On the parallelogram $A_{12}$, we compare $|x-y|$ and $y_{2}-x_{2}$. Suppose $(y_{1},y_{2})\in A_{12}$. Then, $y_{2}-x_{2} \geq 2mx_{1}$ by the vertical cut-off and due to being in $A_{1}$, we have $y_{2}-x_{2} \geq m(y_{1}-3x_{1}) = m(y_{1}-x_{1}) - 2mx_{1}$. Combining the two estimates, we have $2(y_{2}-x_{2}) \geq m(y_{1}-x_{1})$. Thus,
		$$ |x-y|^{2} = (y_{1}-x_{1})^{2} + (y_{2}-x_{2})^{2} \leq (\frac{4}{m^{2}}+1)(y_{2}-x_{2})^{2}.$$
		This yields the estimates:
		\begin{align*}
		\int_{A_{12}}\frac{y_{2}-x_{2}}{|x-y|^{2+2\beta}} dy &\geq \int_{A_{12}}\frac{y_{2}-x_{2}}{((\frac{4}{m^{2}}+1)(y_{2}-x_{2})^{2})^{1+\beta}} dy \\
		&= \frac{1}{(\frac{4}{m^{2}}+1)^{1+\beta}}\int_{x_{2}+2mx_{1}}^{x_{2}+ma} dy_{2} \int_{x_{1}+\frac1m(y_{2}-x_{2})}^{3x_{1}+\frac1m(y_{2}-x_{2})} dy_{1} (y_{2}-x_{2})^{-1-2\beta}\\
		&= \frac{2x_{1}}{(\frac{4}{m^{2}}+1)^{1+\beta}} \frac{1}{2\beta} \Big( (2mx_{1})^{-2\beta} - (ma)^{-2\beta}  \Big)\\
		&=\frac{(2m)^{-2\beta}}{\beta(\frac{4}{m^{2}}+1)^{1+\beta}}x_{1}^{1-2\beta} + O(x_{1}).
		\end{align*}

	We integrate over another piece of $A_{1}$, which we call $A_{13} = \{ (y_{1},y_{2}) \ | \ y_{1}\in (3x_{1},4x_{1}) \text{  and  } y_{2} \in (x_{2}+m(y_{1}-3x_{1}), x_{2} + 2mx_{1})  \}.$ On this region, we have the estimate
	
	\begin{align*}
	\int_{A_{13}}\frac{y_{2}-x_{2}}{|x-y|^{2+2\beta}} dy &= \int_{3x_{1}}^{4x_{1}} dy_{1} \int_{x_{2}+m(y_{1}-3x_{1})}^{x_{2}+2mx_{1}} dy_{2}\frac{y_{2}-x_{2}}{|x-y|^{2+2\beta}}\\
	=\int_{3x_{1}}^{4x_{1}} dy_{1} \frac{1}{2\beta}&\Big(  \frac{1}{((y_{1}-x_{1})^{2}+ m^{2}(y_{1}-3x_{1})^{2})^{\beta}} - \frac{1}{((y_{1}-x_{1})^{2}+ 4m^{2}x_{1}^{2}) ^{\beta}} \Big)\\
	&\geq \frac{x_{1}}{2\beta} \Big( \frac{1}{(9x_{1}^{2}+ m^{2}x_{1}^{2})^{\beta}} - \frac{1}{(4x_{1}^{2} + 4m^{2}x_{1}^{2})^{\beta}}  \Big)\\
	&= \frac{1}{2\beta} \Big( \frac{1}{(9+ m^{2})^{\beta}} - \frac{1}{(4 + 4m^{2})^{\beta}}  \Big)x_{1}^{1-2\beta}.
	\end{align*}

	Combining the above estimates, we obtain the lemma.
	\end{proof}
	
	\begin{lemma}\label{u2good}
		We have the following estimate for $u_{2}^{good}(t,x)$ for $x_{2} < \delta_{\beta} << 1$:
		
		$$u_{2}^{good}(t,x) \geq \frac{1-(m^{2}+1)^{-1-\beta}}{2\beta(m^{2}+1)^{\beta}} \Big( 1 +\frac{2^{1-2\beta}-1}{1-2\beta}\Big) x_{2}^{1-2\beta} + O(x_{2}) $$
	\end{lemma}
	
	\begin{proof}
		We have the estimate on $u_{2}^{good}$:
		$$u_{2}^{good}(x) \geq \int_{A(x)}\frac{y_{1}-x_{1}}{|x-y|^{2+2\beta}} dy -  \int_{A(x)+(0,2x_{2})}\frac{y_{1}-x_{1}}{|x-y|^{2+2\beta}} dy.$$
		By comparing the domains of $A(x)$ and $A(x) + (0,2x_{2})$, we see that
		$$u_{2}^{good}(x) \geq \int_{B_{1}}\frac{y_{1}-x_{1}}{|x-y|^{2+2\beta}} dy -  \int_{B_{2}}\frac{y_{1}-x_{1}}{|x-y|^{2+2\beta}} dy,$$
		where $$B_{1} = (x_{1},x_{1}+a) \times (x_{2}, 3x_{2})$$ and $$B_{2} = \{ (y_{1}, y_{2}+ m(y_{1}-x_{1})) \ | \ y_{1}\in(x_{1},x_{1}+a) \text{ and } y_{2}\in(x_{2},3x_{2})\}.$$
		Changing variables $y_{2} \rightleftarrows y_{2} - m(y_{1}-x_{1})$ in the $B_{2}$ integral, we obtain
		$$u_{2}^{good}(x) \geq \int_{B_{1}}\Big(\frac{y_{1}-x_{1}}{|x-y|^{2+2\beta}}  -  \frac{y_{1}-x_{1}}{|x-(y_{1}, y_{2}+m(y_{1}-x_{1}))|^{2+2\beta}}\Big) dy.$$
		Next, notice that
		\begin{multline*}
		|x-(y_{1}, y_{2}+m(y_{1}-x_{1}))|^{2} = (y_{1}-x_{1})^{2} + (x_{2}-(y_{2}+m(y_{1}-x_{1})))^{2}\\
		= (y_{1}-x_{1})^{2} + (x_{2}-y_{2})^{2}-2m(x_{2}-y_{2})(y_{1}-x_{1})+m^{2}(y_{1}-x_{1})^{2}\\
		=(m^{2}+1)(y_{1}-x_{1})^{2} + (m^{2}+1)(y_{2}-x_{2})^{2}+2m(y_{2}-x_{2})(y_{1}-x_{1})-m^{2}(y_{2}-x_{2})^{2}\\
		= (m^{2}+1)|x-y|^{2}+m(y_{2}-x_{2})(2(y_{1}-x_{1})-m(y_{2}-x_{2})).
		\end{multline*}
		Now, when $(y_{1},y_{2})\in (x_{1}+ mx_{2}, x_{1}+a)\times (x_{2},3x_{2})$, we have that $y_{2}-x_{2} \leq 2x_{2}$. Hence,
		$$2(y_{1}-x_{1}) \geq 2m x_{2} \geq m(y_{2}-x_{2}).$$
		Thus,
		$$m(y_{2}-x_{2})(2(y_{1}-x_{1})-m(y_{2}-x_{2})) \geq 0$$ and
		$$|x-(y_{1}, y_{2}+m(y_{1}-x_{1}))|^{2}\geq (m^{2}+1)|x-y|^{2}.$$ 
		Thus, on $\tilde{B}_{1} = B_{1} \cap (x_{1}+ mx_{2}, x_{1}+a)\times (x_{2},3x_{2})$, we have
		\begin{align*}
		u_{2}^{good}(x) &\geq \int_{\tilde{B}_{1}}\Big(\frac{y_{1}-x_{1}}{|x-y|^{2+2\beta}}  -  \frac{y_{1}-x_{1}}{(m^{2}+1)^{1+\beta}|x-y|^{2+2\beta}}\Big) dy\\
		&= (1-(m^{2}+1)^{-1-\beta})\int_{x_{2}}^{3x_{2}} dy_{2} \int_{x_{1}+mx_{2}}^{x_{1}+a} dy_{1}\frac{y_{1}-x_{1}}{|x-y|^{2+2\beta}}\\
		&=\frac{1\!-\!(m^{2}\!+\!1)^{-\!1\!-\!\beta}}{2\beta}\int_{x_{2}}^{3x_{2}}\!\!\!\! dy_{2}\Big(\frac{1}{(m^{2}x_{2}^{2}\!+\!(y_{2}\!-\!x_{2})^{2})^{\beta}}\!-\!\frac{1}{(a^{2}\!+\!(y_{2}\!-\!x_{2})^{2})^{\beta}}\Big).
		\end{align*}
		The second integral in the last line is $O(x_{2})$ so it remains to estimate the first integral.
		\begin{align*}
		\int_{x_{2}}^{3x_{2}} dy_{2} \frac{1}{(m^{2}x_{2}^{2}+ (y_{2}-x_{2})^{2})^{\beta}} =& \int_{x_{2}}^{2x_{2}} dy_{2} \frac{1}{(m^{2}x_{2}^{2}+ (y_{2}-x_{2})^{2})^{\beta}}\\
		& + \int_{2x_{2}}^{3x_{2}} dy_{2} \frac{1}{(m^{2}x_{2}^{2}+ (y_{2}-x_{2})^{2})^{\beta}}\\
		\geq \int_{x_{2}}^{2x_{2}} dy_{2} \frac{1}{(m^{2}x_{2}^{2}+ x_{2}^{2})^{\beta}}& + \int_{2x_{2}}^{3x_{2}} dy_{2} \frac{1}{(m^{2}(y_{2}-x_{2})^{2}+ (y_{2}-x_{2})^{2})^{\beta}}\\
		=& \frac{1}{(m^{2}+1)^\beta}\Big(x_{2}^{1-2\beta} + \frac{(2x_{2})^{1-2\beta} - x_{2}^{1-2\beta}}{1-2\beta}\Big).
		\end{align*}
		Combining these estimates yields the lemma.
	\end{proof}
	By combining the estimates of Lemma \ref{u1bad} and Lemma \ref{u1good}, we obtain for $x_{2} \leq mx_{1} < \delta_{\beta}$
	
	\begin{multline}\label{u1bound}
	u_{1}(x) \leq -\frac1{\beta2^{2\beta}}\Big(\frac{1-(m^{2}+1)^{-\beta}}{1-2\beta} + \frac{1}{m^{2\beta}(1+\frac{4}{m^{2}})^{1+\beta}}\Big) x_{1}^{1-2\beta}\\ - \frac{1}{2\beta}  \Big( \frac{1}{(9+ m^{2})^{\beta}} - \frac{1}{(4 + 4m^{2})^{\beta}}  \Big)x_{1}^{1-2\beta} + \frac{1}{\beta}\Big( \frac{1}{1-2\beta} - (1+m^{2})^{-\beta} \Big)x_{1}^{1-2\beta}+ O(x_{1}).
	\end{multline}
	Plugging in $m=5$, we obtain that $u_{1}(x) < -\tilde{C}_{1}x_{1}^{1-2\beta} + O(x_{1})$ for $0 < \beta < 0.168$ and a positive constant $\tilde{C}_{1}$. Hence, by continuity of the expression on the right hand side, for $0<\beta < 1/6$, there exists a fixed positive constant $C_{1}$, depending on $\beta$, such that
	$$u_{1} \leq -C_{1}x_{1}^{1-2\beta}$$
	for $x_{1} < \delta_{\beta}$ sufficiently small.
	Similarly, combining the Lemma \ref{u2bad} and Lemma \ref{u2good}, we obtain the estimate for $mx_{1} \leq x_{2} < \delta_{\beta} $
	\begin{multline}\label{u2bound}
	u_{2}(x) \geq \frac{1-(m^{2}+1)^{-1-\beta}}{2\beta(m^{2}+1)^{\beta}} \Big( 1 +\frac{2^{1-2\beta}-1}{1-2\beta}\Big) x_{2}^{1-2\beta}\\ - \frac{1}{\beta}\Big( \frac{1}{1-2\beta} - (1+\frac{1}{m^{2}})^{-\beta} \Big)x_{2}^{1-2\beta}+ O(x_{2}).
	\end{multline}
	Again, plugging in $m = 5$, we obtain that $u_{2} > \tilde{C}_{2}x_{2}^{1-2\beta}+ O(x_{2})$ for $0<\beta < .167$ and positive constant $\tilde{C}_{2}$. Hence, for $0< \beta < 1/6$, for $x_{2} < \delta_{\beta}$ small enough, there exists a fixed positive constant $C_{2}$, depending on $\beta$, such that $$u_{2} \geq C_{2}x_{2}^{1-2\beta}.$$ We shall now let the constant $C$ in the ordinary differential equation for $X(t)$ be a fixed positive number that is smaller than $C_{1}$ and let $a$ small enough so the trapezoid is inside the patch.
	
	We demonstrate that the trapezoid remains within the patch using a proof by contradiction. A sketch of the proof is as follows. Suppose the trapezoid $K(t)$ crosses out of the patch before touching the origin. Then, there exists an initial time $t_{0} < T$ where the patch and trapezoid first intersect.
	
	If we choose the initial patch $D_{0}$ with sufficiently small $\epsilon > 0$, then the time $T$ where $X(T) = 0$ is also sufficiently small. Because $\|u\|_{L^{\infty}}$ is bounded and the velocity is a continuous function, the intersection of the trapezoid and patch must occur on either the vertical segment $\{X(t)\} \times (0,mX(t))$ or on the sloped segment connecting $(X(t),mX(t))$ and $(\delta_{\beta}, m\delta_{\beta})$. (See \cite{KRYZ} for details)
	
	Suppose the intersection occurs on the vertical segment. Because $u_{1}(x_{1},x_{2}) \leq X'(t) < 0$ for $x_{2}\leq mx_{1}$, the part of the patch boundary intersecting the trapezoid is moving towards the origin faster than the trapezoid. Hence, by continuity of the velocity function, there must have been a prior time at which the patch boundary already crossed the trapezoid. This is a contradiction of $t_{0}$ being the first moment where an intersection occurred.
	
	Similarly, if the intersection occurs on the sloped segment, then since the sloped segment of $K(t)$ remains on the line $x_{2} = mx_{1}$ and the component of the velocity normal to the sloped segment is positive, we have a contradiction using the same logic as above.
%	 the fact that $u_{1} \leq X'(t) < 0$ and $u_{2} \geq C_{2}x_{2}^{1-2\beta} \geq C_{2} m^{1-2\beta}x_{1}^{1-2\beta} \geq  - mX'(t) > 0$ yields a contradiction using the same logic.
%	
%	Now, consider the function $f(t) = \text{dist}(\overline{D(t)}^{c}, K(t))$ where $\overline{D(t)}^{c}$ denotes the complement of the closure of the patch $D(t)$ in the upper right quadrant of $\mathbb{R}^{2}$. Since the velocity $u(t,x)$ is continuous on $[0,T]$, the function $f(t)$ is also continuous on $[0,T]$. Hence, there exists a first time $t_{0} > 0$ such that $f(t_{0}) = 0$ and $K(t_{0}) \subset \overline{D(t_{0})}$.
%	
%	Set $D_{3} = (\delta_{\beta}, 5/2) \times (0,5/2)$. Since $\|u \|_{L^{\infty}} \leq 100$, $T \leq \delta_{\beta}/200$, $\epsilon < \delta_{\beta}/2$ and $D_{0} \subset (\epsilon, 4) \times (0,4)$, we have that $D_{3} \subset \overline{D(t)}$ for any $t\in [0,T]$ since $$\epsilon + 100T < \epsilon + \delta_{\beta}/2 < \delta_{\beta}.$$ Specifically $D_{3} \subset \overline{D(t_{0})}$.
%	
%	Next, since $f(t_{0}) = 0$ for the first time at $t_{0}$, we have that there exists some point $x = (x_{1},x_{2}) \in K(t_{0})$ such that $x$ is in the boundary of $D(t_{0})$. Due to $D_{3} \subset D(t_{0})$, we can specify that $x$ lies in either the line segment $$I_{1}= \{X(t_{0})\} \times [0,X(t_{0}))$$ or the segment $$I_{2} = \{(\tilde{x},\tilde{x}) \ | \ \tilde{x}\in [X(t_{0}), \delta_{\beta}) \}.$$
	
	\section{Proof of Theorems \ref{LE0a1H2}, \ref{LE1a2H3} and \ref{RC}}
	
	This section is devoted to prove Theorems \ref{LE0a1H2}, \ref{LE1a2H3} and \ref{RC}. We show below the main part of the argument: energy estimates. We consider the cases $H^2$ and $H^3$ as the rest of them are analogous. At the end of the section we collect all the necessary bounds to prove each result. Due to the size of formulas, we consider a more compact notation, denoting $f(\gamma,t)=f$ and $f_{-}=f(\g,t)-f(\g-\e,t)$ when there is no danger of confusion. We also denote $c_\alpha$ an universal constant only depending on $\alpha$. The existence results passes thorough an approximation method to get from the a priori energy estimates bona fide solutions. This part of the strategy can be found in \cite{G} and references therein.\\

	\emph{Proof:}\\
	
	First, the lower order terms in the energy estimate provide
	\begin{align*}
	\frac{d}{dt}\|x\|^2_{L^2}&\leq \|\pg\lambda \|_{L^\infty}\|x\|_{L^2}^2,
	\end{align*}
	where we have symmetrized the first nonlinear term to make it zero and we have integrated by parts on the second nonlinear term. We consider
	$$
	\pg \lambda=\frac1{2\pi}\intT\frac{\pg x}{|\pg x|^2}\pg\Big(\intT\frac{\pg x_-}{|x_-|^{\al}}d\e\Big) d\g +A_1+A_2,
	$$
	with
	$$
	A_1=-\frac{\pg x}{|\pg x|^2}\cdot\intT \frac{\pg^2 x_-}{|x_-|^{\al}} d\e,\quad\mbox{and}\quad
	A_2=c_\al\frac{\pg x}{|\pg x|^2}\cdot\intT \frac{\pg x_- (x_-\cdot\pg x_-)}{|x_-|^{2+\al}} d\e.
	$$
	The identity
	$$
	A_1=\intT \frac{\pg x_-\cdot\pg^2 \xe  }{|\pg x|^2|x_-|^{\al}} d\e,
	$$
	allows us to obtain
	$$
	\|A_1\|_{L^\infty}\leq \|F(x)\|_{L^\infty}^{2+\al}|\pg x|_{C^{\delta}}\intT\frac{|\pg^2 \xe|}{|\e|^{\alpha-\delta}}d\e\leq c_\al  \|F(x)\|_{L^\infty}^{2+\al}|\pg x|_{C^{\delta}}\|\pg^2 x\|_{L^p},
	$$
	with $p^{-1}+\delta=1$. Sobolev embedding gives the desired bound for the most singular term:
	$$
	\|A_1\|_{L^\infty}\leq c_\al \|F(x)\|_{L^\infty}^{2+\al}\|\pg^2 x\|^2_{L^p}.
	$$
	For $A_2$, we proceed as follows
	$$
	\|A_2\|_{L^\infty}\leq \|F(x)\|_{L^\infty}^{2+\al}|\pg x|_{C^{\delta}}^2\intT\frac{ |\e|^{2\delta}}{|\e|^{1+\al}}d\e\leq c_\al  \|F(x)\|_{L^\infty}^{2+\al}\|\pg^2 x\|^2_{L^p},
	$$
	Same approach for the remainder term in $\pg\lambda$ provides finally
	\begin{equation}\label{bnlipl}
	\|\pg\lambda\|_{L^\infty}\leq c_\al  \|F(x)\|_{L^\infty}^{2+\al}\|\pg^2 x\|^2_{L^p},
	\end{equation}
	and therefore
	\begin{equation}\label{nL2x}
	\frac{d}{dt}\|x\|^2_{L^2}\leq c_\al  \|F(x)\|_{L^\infty}^{2+\al}\|\pg^2 x\|^2_{L^p}\|x\|_{L^2}^2.
	\end{equation}
	In order to control a higher order Sobolev norm we consider first the evolution of the $H^2$ norm. It yields
	\begin{align*}
	\frac12\frac{d}{dt}\|\pg^2x\|^2_{L^2}&=I_1+I_2+I_3+I_4,
	\end{align*}
	where
	$$
	I_1=\intT\intT \pg^2 x\cdot\frac{\pg^3 x_{-}}{|x_-|^{\al}}d\g d\e,\qquad
	I_2=c_\al\intT\intT \pg^2 x\cdot \pg^2 x_- \frac{ x_-\cdot\pg x_-}{|x_-|^{2+\al}} d\g d\e,
	$$
	
	$$
	I_3=c_\al \intT\intT \pg^2 x \cdot\pg x_-\pg\Big(\frac{ x_-\cdot\pg x_-}{|x_-|^{2+\al}}\Big)d\g d\e,\quad\mbox{and}\quad I_4=\frac32\intT |\pg^2 x|\pg\lambda d\g.
	$$
	The term $I_1$ can be symmetrized as before so that
	$$
	I_1=c_\al\intT\intT |\pg^2 x_-|^2\frac{x_-\cdot \pg x_-}{|x_-|^{2+\al}}d\g d\e.
	$$
	Next we decompose as follows
	\begin{align}
	\begin{split}\label{maindecom}
	x_-\cdot \pg x_-=&
	(x_--\pg x(\g)\e)\cdot \pg x_-+\e\pg x(\g)\cdot\pg x_-.
	\end{split}
	\end{align}
	It allows to obtain
	\begin{equation}\label{tivtlo}
	\pg x(\g)\cdot\pg x_-=\frac12|\pg x_-|^2
	\end{equation}
	to get finally extra order in $|\e|$:
	\begin{equation}\label{mca}
	|x_-\cdot \pg x_-|\leq \frac32|\eta|^{1+2\delta}|\pg x|^2_{C^{\delta}}.
	\end{equation}
	Hence
	\begin{align*}
	I_1&\leq c_\al\|F(x)\|_{L^\infty}^{2+\al}|\pg x|^2_{C^{\delta}}\|\pg^2 x\|_{L^2}^2\leq c_\al\|F(x)\|_{L^\infty}^{2+\al}\|\pg^2 x\|^2_{L^p}\|\pg^2 x\|_{L^2}^2.
	\end{align*}
	It is possible to get a similar bound for $I_2$ so that 
	\begin{align*}
	I_2&\leq c_\al\|F(x)\|_{L^\infty}^{2+\al}\|\pg^2 x\|^2_{L^p}\|\pg^2 x\|_{L^2}^2.
	\end{align*}
	To deal with $I_3$ we decompose it further, $I_3=I_{31}+I_{32}+I_{33}$ so that
	$$
	I_{31}=c_\al \intT\intT \pg^2 x\cdot \pg x_-
	\frac{x_-\cdot \pg^2 x_-}{|x_-|^{2+\al}},
	d\g d\e,
	\quad 
	I_{32}=c_\al \intT\intT \pg^2 x \cdot  \pg x_-\frac{|\pg x_-|^2}{|x_-|^{2+\al}}d\g d\e,
	$$
	and
	$$I_{33}=c_\al \intT\intT \pg^2 x\cdot\pg x_-\frac{|x_-\cdot \pg x_-|^2}{|x_-|^{4+\al}}d\g d\e.
	$$
	Inside $I_{31}$ we take
	\begin{align*}
	x_-\cdot \pg^2 x_-=(x_--\pg x(\g)\e)\cdot \pg^2 x_- -\e \pg x_- \cdot\pg^2x(\g-\e),
	\end{align*}
	to find
	\begin{align*}
	I_{31}&\leq  c_\al \|F(x)\|_{L^\infty}^{2+\al}|\pg x|^2_{C^\delta}\|\pg^2 x\|_{L^2}^2.
	\end{align*}
	Writing $\pg x_-=\pg^2 x(\g-s\e)\e$, for $s\in(0,1)$, it is possible to obtain
	\begin{align*}
	I_{32}&\leq  c_\al \|F(x)\|_{L^\infty}^{2+\al}|\pg x|^2_{C^\delta}\|\pg^2 x\|_{L^2}^2.
	\end{align*}
	Analogous approach yields
	\begin{align*}
	I_{33}&\leq c_\al \|F(x)\|_{L^\infty}^{2+\al}|\pg x|^2_{C^\delta}\|\pg^2 x\|_{L^2}^2.
	\end{align*}
	Hence, we are done with $I_3$ using Sobolev injection as before. It remains to control $I_4$ but estimate \eqref{bnlipl} gives the desired bound:
	$$
	I_4\leq c_\al\|F(x)\|_{L^\infty}^{2+\al}\|\pg^2 x\|^2_{L^p}\|\pg^2 x\|_{L^2}^2.
	$$
	Gathering all the $I_j$ estimates we obtain
	\begin{equation}\label{cnh2jf}
	\frac{d}{dt}\|\pg^2 x\|^2_{L^2}\leq c_\al\|F(x)\|_{L^\infty}^{2+\al}\|\pg^2 x\|^2_{L^p}\|\pg^2 x\|_{L^2}^2.
	\end{equation}

	For the higher order Sobolev norm it is possible to find
	\begin{align*}
	\frac12\frac{d}{dt}\|\pg^3x\|^2_{L^2}&=J_1+J_2+J_3+J_4+J_5,
	\end{align*}
	where
	$$
	J_1=\intT\intT \pg^3 x \cdot\frac{\pg^4 x_-}{|x_-|^{\al}}d\g d\e,\qquad
	J_2=c_\al\intT\intT \pg^3 x \cdot \pg^3 x_- \frac{ x_-\cdot\pg x_-}{|x_-|^{2+\al}}d\g d\e,
	$$
	
	$$
	J_3=c_\al \intT\intT \pg^3 x\cdot \pg^2 x_-\pg\Big(\frac{ x_-\cdot\pg x_-}{|x_-|^{2+\al}}\Big)d\g d\e,
	$$
	$$
	J_4=c_\al \intT\intT \pg^3 x\cdot\pg x_-\pg^2 \Big(\frac{ x_-\cdot\pg x_-}{|x_-|^{2+\al}}\Big)d\g d\e,
	\quad\mbox{and}\quad
	J_5=\intT \pg^3 x\cdot \pg^3(\lambda\pg x)d\g.
	$$
	The term $J_1$ can be symmetrized to be bound as $I_1$ so that
	\begin{align*}
	J_1&\leq c_\al\|\pg^2 x\|^2_{L^p}\|F(x)\|_{L^\infty}^{2+\al}\|\pg^3 x\|_{L^2}^2.
	\end{align*}
	It is 
	possible to get a similar bound for $J_2$: 
	\begin{align*}
	J_2&\leq c_\al\|\pg^2 x\|^2_{L^p}\|F(x)\|_{L^\infty}^{2+\al}\|\pg^3 x\|_{L^2}^2.
	\end{align*}
	To deal with $J_3$ we decompose it further, $J_3=J_{31}+J_{32}+J_{33}$ so that
	$$
	J_{31}=c_\al \intT\intT \pg^3 x \cdot \pg^2 x_-\frac{x_-\cdot \pg^2x_-}{|x_-|^{2+\al}}d\g d\e,\quad
	J_{32}=c_\al \intT\intT \pg^3 x\cdot \pg^2 x_-\frac{|\pg x_-|^2}{|x_-|^{2+\al}}d\g d\e,
	$$
	and
	$$J_{33}=c_\al \intT\intT \pg^3 x \cdot \pg^2 x_-\frac{|x_-\cdot \pg x_-|^2}{|x_-|^{4+\al}}d\g d\e.
	$$
	In the term $J_{31}$ we split further to find $J_{31}=J_{311}+J_{312}$ with
	$$
	J_{311}=c_\al \intT\intT \pg^3 x(\g) \cdot \pg^2 x_-\frac{x_-\cdot \pg^2x(\gamma)}{|x_-|^{2+\al}}d\g d\e,
	$$
	and
	$$
	J_{312}=-c_\al \intT\intT \pg^3 x(\g) \cdot \pg^2 x_-\frac{x_-\cdot \pg^2\xe}{|x_-|^{2+\al}}d\g d\e.
	$$
	We rewrite $J_{311}$ as follows
	\begin{equation}\label{J311}
	J_{311}=c_\al \int_0^1\!\!\intT\intT \pg^3 x(\g) \cdot \pg^3 x(\g\!+\!(s-1)\e)\e\frac{(x_-\!-\!\pg x(\g)\e)\cdot \pg^2x(\gamma)}{|x_-|^{2+\al}}d\g d\e ds,
	\end{equation}
	to get
	\begin{align*}
	J_{311}&\leq c_\al\|F(x)\|_{L^\infty}^{2+\al}|\pg x|_{C^\delta} \int_0^1\!\!\intT\intT \frac{|\pg^3 x(\g)||\pg^3 x(\g\!+\!(s-1)\e)||\pg^2x(\gamma)| }{|\e|^{\al-\delta}}d\g d\e ds\\
	&\leq c_\al\|F(x)\|_{L^\infty}^{2+\al}|\pg x|_{C^\delta} \int_0^1\Big\|\intT |\pg^3 x(\g)||\pg^3 x(\g\!+\!(s-1)\cdot)||\pg^2x(\gamma)|d\g\Big\|_{L^p} ds\\
	&\leq c_\al
	\|F(x)\|_{L^\infty}^{2+\al}|\pg x|_{C^\delta}
	\|\pg^3x\|_{L^2}\int_0^1\Big\|\intT |\pg^3 x(\g\!+\!(s-1)\cdot)|^2|\pg^2x(\gamma)|^2d\g\Big\|_{L^{\frac p2}}^{\frac12} ds
	\end{align*}
	using H\"older inequalities. A change of variables provides a convolution denoting $\tilde{f}(\g)=f(-\g)$, so that using Young inequality next we finally get
	\begin{align*}
	J_{311}&\leq c_\al
	\|F(x)\|_{L^\infty}^{2+\al}|\pg x|_{C^\delta}
	\|\pg^3x\|_{L^2}\int_0^1\frac{ds}{(1-s)^{\frac1p}}\Big\| |\widetilde{\pg^3 x}|^2*|\pg^2x|^2\Big\|_{L^{\frac p2}}^{\frac12}\\
	&\leq  c_\al
	\|F(x)\|_{L^\infty}^{2+\al}|\pg x|_{C^\delta}\|\pg^2x\|_{L^p}
	\|\pg^3x\|_{L^2}^2.
	\end{align*}
	In $J_{312}$ we take $\pg^2 x_-=\pg^3 x(\g\!-\!r\e)\e$ with $r\in(0,1)$ to find
	\begin{equation}\label{J312}
	J_{312}=c_\al\intT\intT \pg^3 x(\g) \cdot \pg^3 x(\g\!-\!r\e)\e\frac{(x_-\!-\!\pg \xe \e)\cdot \pg^2\xe}{|x_-|^{2+\al}}d\g d\e ds,
	\end{equation}
	and therefore
	\begin{align*}
	J_{312}&\leq c_\al\|F(x)\|_{L^\infty}^{2+\al}|\pg x|_{C^\delta} \Big\|\intT |\pg^3 x(\g)||\pg^3 x(\g\!-\!r\cdot)||\pg^2x(\gamma\!-\!\cdot)|d\g\Big\|_{L^p}\\
	&\leq c_\al
	\|F(x)\|_{L^\infty}^{2+\al}|\pg x|_{C^\delta}
	\|\pg^3x\|_{L^2}\Big\| |\pg^3 x|^2*|\widetilde{\pg^2x}|^2\Big\|_{L^{\frac p2}}^{\frac12}\\
	&\leq c_\al
	\|F(x)\|_{L^\infty}^{2+\al}|\pg x|_{C^\delta}\|\pg^2x\|_{L^p}
	\|\pg^3x\|_{L^2}^2.
	\end{align*}
	Above estimate gives finally the desired estimate for $J_{31}$:
	$$
	J_{31}\leq c_\al
	\|F(x)\|_{L^\infty}^{2+\al}\|\pg^2x\|^2_{L^p}
	\|\pg^3x\|_{L^2}^2.
	$$ 
	We can bound $J_{32}$ and $J_{33}$ as before to obtain
	\begin{align*}
	J_{32}+J_{33}&\leq  c_\al \|F(x)\|_{L^\infty}^{2+\al}|\pg x|^2_{C^\delta}\|\pg^3 x\|_{L^2}^2\leq c_\al
	\|F(x)\|_{L^\infty}^{2+\al}\|\pg^2x\|^2_{L^p}
	\|\pg^3x\|_{L^2}^2,
	\end{align*}
	and finally
	\begin{align*}
	J_3&\leq  c_\al \|F(x)\|_{L^\infty}^{2+\al}\|\pg^2x\|^2_{L^p}
	\|\pg^3x\|_{L^2}^2.
	\end{align*}
	
	Hence, we are done with $J_3$. We then continue dealing with $J_4$ with the further splitting $J_4=J_{41}+J_{42}+J_{43}+J_{44}+J_{45}$ where
	$$
	J_{41}=c_\al\intT\intT \pg^3x\cdot\pg x_-\frac{x_-\cdot\pg^3 x_-}{|x_-|^{2+\al}}d\g d\e,$$
	$$
	J_{42}=c_\al\intT\intT \pg^3x\cdot\pg x_-\frac{\pg x_-\cdot\pg^2 x_-}{|x_-|^{2+\al}}d\g d\e,
	$$
	$$
	J_{43}=c_\al\intT\intT \pg^3x\cdot\pg x_-\frac{x_-\cdot\pg^2 x_-\, x_-\cdot\pg x_-}{|x_-|^{4+\al}}d\g d\e,
	$$ 
	$$
	J_{44}=c_\al\intT\intT \pg^3x\cdot\pg x_-\frac{|\pg x_-|^2 x_-\cdot\pg x_-}{|x_-|^{4+\al}}d\g d\e,
	$$
	and 
	$$
	J_{45}=c_\al\intT\intT \pg^3x\cdot\pg x_-\frac{(x_-\cdot\pg x_-)^3}{|x_-|^{6+\al}}d\g d\e.
	$$
	In $J_{41}$ we need to split further to find $J_{41}=J_{411}+J_{412}+J_{413}$ where
	$$
	J_{411}=c_\al\intT\intT \pg^3x(\g)\cdot\pg x_-\frac{(x_--\pg x(\g)\e)\cdot\pg^3 x_-}{|x_-|^{2+\al}}d\g d\e,$$
	$$
	J_{412}=c_\al\intT\intT \pg^3x\cdot\pg x_-\e\frac{(\pg x\cdot\pg^3x)_-}{|x_-|^{2+\al}}d\g d\e,
	$$
	and finally
	$$
	J_{413}=-c_\al\intT\intT \pg^3x(\gamma)\cdot\pg x_-\e\frac{\pg x_-\cdot\pg^3x(\g\!-\!\e)}{|x_-|^{2+\al}}d\g d\e.
	$$
	It gives
	\begin{align*}
	J_{411}&\leq  c_\al \|F(x)\|_{L^\infty}^{2+\al} |\pg x|^2_{C^\delta}\|\pg^3 x\|_{L^2}^2.
	\end{align*} 
	as desired. We could write
	$$
	(\pg x\cdot\pg^3x)_-=-(|\pg^2x|^2)_-=-2\eta\int_0^1\pg^3x(\g+(s-1)\e)\cdot\pg^2x(\g+(s-1)\e)ds
	$$ 
	with $s\in(0,1)$. Hence, for $J_{412}$ we obtain
	\begin{align*}
	J_{412}&\leq c_\al\|F(x)\|_{L^\infty}^{2+\al}|\pg x|_{C^\delta} \int_0^1\Big\|\intT |\pg^3 x(\g)||\pg^3 x\cdot\pg^2x|(\g\!+\!(s-1)\cdot)d\g\Big\|_{L^p} ds\\
	&\leq c_\al
	\|F(x)\|_{L^\infty}^{2+\al}|\pg x|_{C^\delta}
	\|\pg^3x\|_{L^2}\int_0^1\Big\|\intT |\pg^3 x(\g)|^2|\pg^2x(\gamma\!+\!(s-1)\cdot)|^2d\g\Big\|_{L^{\frac p2}}^{\frac12} ds\\
	&\leq c_\al
	\|F(x)\|_{L^\infty}^{2+\al}|\pg x|_{C^\delta}
	\|\pg^3x\|_{L^2}\Big\| |\pg^3 x|^2*|\widetilde{\pg^2x}|^2\Big\|_{L^{\frac p2}}^{\frac12}
	\end{align*}
	and therefore
	\begin{align*}
	J_{412}&\leq  c_\al \|F(x)\|_{L^\infty}^{2+\al}\|\pg^2 x\|^2_{L^p}\|\pg^3 x\|_{L^2}^2.
	\end{align*} 
	Bound
	\begin{align*}
	J_{413}&\leq  c_\al \|F(x)\|_{L^\infty}^{2+\al}|\pg x|^2_{C^\delta}\|\pg^3 x\|_{L^2}^2,
	\end{align*} 
	gives the desired control for $J_{41}$:
	\begin{align*}
	J_{41}&\leq  c_\al \|\pg^2 x\|^2_{L^p}\|F(x)\|_{L^\infty}^{2+\al}\|\pg^3 x\|_{L^2}^2.
	\end{align*} 
	Moving to $J_{42}$ and $J_{43}$ they can be estimate as desired
	\begin{align*}
	J_{42}+J_{43}&\leq  c_\al \|F(x)\|_{L^\infty}^{2+\al}|\pg x|^2_{C^\delta}\|\pg^3 x\|_{L^2}^2.
	\end{align*}
	The approach for $J_{44}$ and $J_{45}$ is different so that
	\begin{align}
	\begin{split}\label{J44y45}
	J_{44}+J_{45}&\leq  c_\al \|F(x)\|_{L^\infty}^{3+\al}|\pg^2 x|_{C^\delta}\|\pg^3 x\|_{L^2}\|\pg^2 x\|_{L^6}^3\\
	&\leq c_\al \|F(x)\|_{L^\infty}^{3+\al}|\pg^2 x|_{C^\delta}
	\|\pg^3 x\|_{L^2}^2\|\pg^2 x\|_{L^2}^2,
	\end{split}
	\end{align}
	by Gagliardo–Nirenberg interpolation inequality. Using \eqref{cnh2jf} it is possible to control the $L^2$ norm of two derivatives in such a way that 
	\begin{align*}
	J_{44}(t)+J_{45}(t)&\leq c_\al \|\pg^2 x\|_{L^\infty_TL^2}^2 \|F(x)\|_{L^\infty}^{3+\al}(t)|\pg^2 x|_{C^\delta}(t)
	\|\pg^3 x\|_{L^2}^2(t).
	\end{align*}
	We are then done with $J_{4}$. We continue dealing with $J_5$ with the splitting $J_5=J_{51}+J_{52}$ where
	$$
	J_{51}=\frac52\intT |\pg^3 x|^2 \pg\lambda d\g,\quad J_{52}=5\intT \pg^3 x\cdot \pg^2x \pg^2\lambda d\g
	$$
	Using \eqref{bnlipl} it is possible to get
	\begin{align*}
	J_{51}&\leq  c_\al \|F(x)\|_{L^\infty}^{2+\al}\|\pg^2 x\|^2_{L^p}\|\pg^3 x\|_{L^2}^2.
	\end{align*}
	Considering
	$$
	\pg^2 \lambda=\pg\Big(\intT \frac{\pg^2 \xe \cdot\pg x_- }{|\pg x|^2|x_-|^{\al}} d\e+c_\al\intT \frac{\pg x \cdot \pg x_-\, x_-\cdot \pg x_-}{|\pg x|^2|x_-|^{2+\al}} d\e\Big),
	$$
	for $J_{52}$ we split $J_{52}=\sum_{m=1}^8J_{52j}$ where
	$$
	J_{521}=5\intT \pg^3 x\cdot \pg^2x\intT\frac{\pg^3 \xe \cdot\pg x_- }{|\pg x|^2|x_-|^{\al}} d\e d\g,
	$$
	$$
	J_{522}=5\intT \pg^3 x\cdot \pg^2x\intT\frac{\pg^2 \xe \cdot\pg^2 x_- }{|\pg x|^2|x_-|^{\al}} d\e d\g,$$
	$$
	J_{523}=c_\al\intT \pg^3 x\cdot \pg^2x\intT\frac{\pg^2 \xe \cdot\pg x_-\, x_-\cdot\pg x_- }{|\pg x|^2|x_-|^{2+\al}} d\e d\g,
	$$
	$$ 
	J_{524}=c_\al\intT \pg^3 x\cdot \pg^2x\intT\frac{\pg^2 x\cdot\pg x_- \,x_-\cdot\pg x_- }{|\pg x|^2|x_-|^{2+\al}} d\e d\g,
	$$
	$$
	J_{525}=c_\al\intT \pg^3 x\cdot \pg^2x\intT \frac{\pg x \cdot \pg^2 x_-\, x_-\cdot \pg x_-}{|\pg x|^2|x_-|^{2+\al}} d\e d\g,
	$$ 
	$$
	J_{526}=c_\al\intT \pg^3 x\cdot \pg^2x\intT\frac{\pg x\cdot\pg x_- \,|\pg x_-|^2 }{|\pg x|^2|x_-|^{2+\al}} d\e d\g,
	$$
	$$
	J_{527}=c_\al\intT \pg^3 x\cdot \pg^2x\intT \frac{\pg x \cdot \pg x_-\, x_-\cdot \pg^2 x_-}{|\pg x|^2|x_-|^{2+\al}} d\e d\g,
	$$
	and
	$$ 
	J_{528}=c_\al\intT \pg^3 x\cdot \pg^2x\intT \frac{\pg x \cdot \pg x_-\, (x_-\cdot \pg x_-)^2}{|\pg x|^2|x_-|^{4+\al}} d\e d\g.
	$$
	Bound
	$$
	J_{521}\leq c_\al\|F(x)\|_{L^\infty}^{2+\al}|\pg x|_{C^\delta}\intT\intT \frac{|\pg^3 \xg||\pg^2\xg||\pg^3\xe|}{|\e|^{\al-\delta}}d\g d\e
	$$
	allows to treat above term as $J_{312}$ \eqref{J312} to get
	$$
	J_{521}\leq c_\al
	\|F(x)\|_{L^\infty}^{2+\al}|\pg x|_{C^\delta}\|\pg^2x\|_{L^p}
	\|\pg^3x\|_{L^2}^2.
	$$
	Integration by parts in $\e$ allows to rewrite
	\begin{align*}
	J_{522}=&-5\intT \pg^3 x\cdot \pg^2x\intT\frac{\pg x_- \cdot\pg^3 \xe }{|\pg x|^2|x_-|^{\al}}d\e d\g\\
	&+c_\al\intT \pg^3 x\cdot \pg^2x\intT\frac{\pg x_- \cdot\pg^2 x_- x_-\cdot\pg \xe }{|\pg x|^2|x_-|^{2+\al}} d\e d\g
	\end{align*}
	and to find the following bound
	\begin{align*}
	J_{522}\leq& c_\al\|F(x)\|_{L^\infty}^{2+\al}|\pg x|_{C^\delta}\intT\intT \frac{|\pg^3 \xg||\pg^2\xg||\pg^3\xe|}{|\e|^{\al-\delta}}d\g d\e \\
	&+c_\al\|F(x)\|_{L^\infty}^{2+\al}|\pg x|_{C^\delta}\int_0^1\!\!\intT\intT \frac{|\pg^3 \xg||\pg^2\xg||\pg^3x(\g\!+\!(s\!-\!1)\e)|}{|\e|^{\al-\delta}}d\g d\e ds.
	\end{align*}
	It allows to consider $J_{522}$ as $J_{312}$ \eqref{J312} and $J_{311}$ \eqref{J311} to obtain the desired bound
	$$
	J_{522}\leq c_\al
	\|F(x)\|_{L^\infty}^{2+\al}|\pg x|_{C^\delta}\|\pg^2x\|_{L^p}
	\|\pg^3x\|_{L^2}^2.
	$$
	For $J_{523}$, $J_{524}$, $J_{526}$ and $J_{528}$ we can get
	$$
	J_{523}+J_{524}+J_{526}+J_{528}\leq c_\al\|F(x)\|_{L^\infty}^{3+\al}|\pg x|_{C^\delta}\|\pg^3 x\|_{L^2}\|\pg^2x\|_{L^6}^3,
	$$
	so that it can be estimate as $J_{44}$ and $J_{45}$ \eqref{J44y45}. Finally, the terms $J_{525}$ and $J_{527}$ are bounded as follows
	\begin{align*}
	J_{525}\!+\!J_{527}\leq&c_\al\|F(x)\|_{L^\infty}^{2+\al}|\pg x|_{C^\delta}\!\int_0^1\!\!\!\intT\intT\! \frac{|\pg^3 \xg||\pg^2\xg||\pg^3x(\g\!+\!(s\!-\!1)\e)|}{|\e|^{\al-\delta}}d\g d\e ds,
	\end{align*}  
	so that they can be estimated as $J_{522}$:
	$$
	J_{525}\leq c_\al
	\|F(x)\|_{L^\infty}^{2+\al}|\pg x|_{C^\delta}\|\pg^2x\|_{L^p}
	\|\pg^3x\|_{L^2}^2.
	$$
	All the $J_{52j}$ bounds provide
	$$
	J_{52}\leq c_\al(\|\pg^2 x\|_{L^p}+\|\pg^2 x\|_{L^2}^2 \|F(x)\|_{L^\infty})\|\pg^2 x\|_{L^p}\|F(x)\|_{L^\infty}^{2+\al}
	\|\pg^3 x\|_{L^2}^2,
	$$
	as desired. Then, we are done with $J_5$ and therefore gathering all the $J_l$ bounds we find
	\begin{align}\label{H3evol}
	\frac{d}{dt}\|\pg^3x\|_{L^2}&\leq c_\al(\|\pg^2 x\|_{L^p}+\|\pg^2 x\|_{L^2}^2 \|F(x)\|_{L^\infty})\|\pg^2 x\|_{L^p}\|F(x)\|_{L^\infty}^{2+\al}
	\|\pg^3 x\|_{L^2}.
	\end{align}
	
	Next we deal with the important arc-chord condition in a different manner for the $\al$ values.\\
	
	\emph{Case $0<\al<1$:}\\
	
	For the evolution of the arc-chord constant
	$$
	\partial_t F(x)=-\frac{|\e|(\xd)}{|\xd|^3}\cdot(\partial_t\xg\!-\!\partial_t\xe),
	$$
	we consider all the terms in $\partial_t\xg\!-\!\partial_t\xe$ to find
	$$
	\partial_t F(x)=B_1+B_2+B_3+B_4+B_5,
	$$
	where
	$$
	B_1=-\frac{|\e|(\xd)}{|\xd|^3}\cdot(\xdp)\intT\frac{d\xi}{|\xdx|^{\al}},
	$$
	$$
	B_2=\frac{|\e|(\xd)}{|\xd|^3}\cdot\intT\frac{(\xdpxe)}{|\xdx|^{\al}}d\xi,
	$$
	$$
	B_3=-\frac{|\e|(\xd)}{|\xd|^3}\cdot\intT(\xdpex) (g(\g,\xi)-g(\g-\eta,\xi))d\xi,
	$$
	$$
	B_4=-\frac{|\e|(\xd)}{|\xd|^3}\cdot (\lambda(\g)-\lambda(\g-\e))\pg\xg,
	$$
	$$
	B_5=-\frac{|\e|(\xd)}{|\xd|^3}\cdot \lambda(\g-\e)(\xdp),
	$$
	with \begin{equation}\label{gformula}
	g(\g,\xi) =|x(\gamma)-x(\gamma-\xi)|^{-\al}.
	\end{equation}
	
	Inequality \eqref{mca} yields
	$$
	|B_1|\leq \frac32F(x)\|F(x)\|_{L^\infty}^{2+\al}|\pg x|_{C^{\frac12}}^2\intT|\xi|^{-\al}d\xi\leq 
	c_\al F(x)\|F(x)\|_{L^\infty}^{2+\al}\|\pg^2 x\|_{L^2}^2
	$$
	In $B_2$ we need to split further $B_2=B_{21}+B_{22}+B_{23}$ where
	$$
	B_{21}=\frac{|\e|(\xd-\pg\xg\e)}{|\xd|^3}\cdot\intT\frac{(\xdpxe)}{|\xdx|^{\al}}d\xi,
	$$
	$$
	B_{22}=\frac{|\e|\e}{|\xd|^3}\intT(\pg\xg-\pg x(\g-\xi))\cdot\frac{(\xdpxe)}{|\xdx|^{\al}}d\xi,
	$$
	and
	$$
	B_{23}=\frac{|\e|\e}{|\xd|^3}\intT\pg x(\g-\xi)\cdot\frac{(\xdpxe)}{|\xdx|^{\al}}d\xi.
	$$ 
	The use of 1/2-H\"older norms provides as for $B_1$:
	$$
	|B_{21}|\leq 2F(x)\|F\|_{L^\infty}^{2+\al}\|\pg x\|_{C^{\frac12}}^2\intT|\xi|^{-\al}d\xi\leq c_\al F(x)\|F(x)\|_{L^\infty}^{2+\al}\|\pg^2 x\|_{L^2}^2.
	$$
	In $B_{22}$ we use the mean value theorem to find
	$$
	|B_{22}|\leq F(x)\|F\|_{L^\infty}^{2+\al}|\pg x|_{C^{\frac12}}\!\intT\frac{|\pg^2 x(\g\!-\!\xi\!-\!s\e)|}{|\xi|^{\al-\frac12}}d\xi \leq c_\al F(x)\|F(x)\|_{L^\infty}^{2+\al}\|\pg^2 x\|_{L^2}^2.
	$$
	For the last term in the splitting we use \eqref{tivtlo}
	$$
	|B_{23}|\leq \frac12F(x)\|F\|_{L^\infty}^{2+\al}|\pg x|_{C^{\frac12}}^2\intT|\xi|^{-\al}d\xi\leq c_\al F(x)\|F(x)\|_{L^\infty}^{2+\al}\|\pg^2 x\|_{L^2}^2.
	$$
	We are done with $B_2$. Using the mean value theorem in $g$ we find
	$$
	|g(\g,\xi)-g(\g-\e,\xi)|=\al|\e|\frac{|(x(\g_s)\!-\!x(\g_s\!-\!\xi))\cdot (\partial_{\gamma}x(\g_s)\!-\!\partial_{\gamma}x(\g_s\!-\!\xi))|}{|x(\g_s)\!-\!x(\g_s\!-\!\xi)|^{\al+2}}
	$$
	with $\g_s=\g-s\e$, $s\in(0,1)$. Therefore
	$$
	|B_3|\leq c_\al F(x)\|F\|_{L^\infty}^{2+\al}|\pg x|_{C^{\frac12}}^2\intT|\xi|^{-\al}d\xi\leq c_\al F(x)\|F(x)\|_{L^\infty}^{2+\al}\|\pg^2 x\|_{L^2}^2.
	$$
	In $B_4$ we use the bound \eqref{bnlipl} for $p\leq 2$ to get
	$$
	|B_4|\leq c_\al F(x)\|F\|_{L^\infty}^{3+\al}\|\pg^2x\|_{L^p}^2\|\pg x\|_{L^\infty}\leq c_\al F(x)\|F(x)\|_{L^\infty}^{3+\al}\|\pg^2x\|_{L^p}\|\pg^2 x\|_{L^2}^2.
	$$
	Using \eqref{mca} in $B_5$ we find for the last term
	$$
	|B_5|\leq c_\al F(x) \|F(x)\|_{L^\infty}^2|\pg x|_{C^{\frac12}}^2 \|\lambda\|_{L^\infty},
	$$
	so that it remains to control $\|\lambda\|_{L^\infty}$. Taking
	\begin{equation}\label{cliag}
	|\lambda(\xi,t)|\leq 2\intT\Big|\frac{\pg x}{|\pg x|^2}\cdot\intT\frac{\pg^2 x_-}{|x_-|^\al}d\e\Big|d\g+2\intT\Big|\frac{\pg x}{|\pg x|^2}\cdot\intT\frac{\pg x_- x_-\cdot\pg x_-}{|x_-|^{\al+2}}d\e\Big|d\g,
	\end{equation}
	for any $\xi$, it is possible to find
	\begin{equation}\label{nliL}
	\|\lambda\|_{L^\infty}\leq c_\al \|F(x)\|_{L^\infty}^{1+\al}\|\pg^2 x\|_{L^p},
	\end{equation}
	so that the following estimate holds
	$$
	|B_5|\leq c_\al F(x)\|F(x)\|_{L^\infty}^{3+\al}\|\pg^2x\|_{L^p}\|\pg^2 x\|_{L^2}^2.
	$$
	Gathering all the $B_m$ estimates it is possible to find
	\begin{equation*}
	\partial_t F(x)\leq c_\al F(x)(1+\|\pg^2x\|_{L^p}\|F(x)\|_{L^\infty})\|F(x)\|_{L^\infty}^{2+\al}
	\|\pg^2 x\|_{L^2}^2
	\end{equation*}
	and finally
	\begin{equation}\label{AC0a1}
	\frac{d}{dt} \|F(x)\|_{L^\infty}\leq c_\al (1+\|\pg^2x\|_{L^p}\|F(x)\|_{L^\infty})\|F(x)\|_{L^\infty}^{3+\al}
	\|\pg^2 x\|_{L^2}^2
	\end{equation}\\
	
	\emph{Case $1\leq\al<2$:}\\
	
	In this more singular case we consider
	$$
	\partial_t F(x)=D_1+D_2+D_3,
	$$
	where
	$$
	D_1=-\frac{|\e|x_-}{|x_-|^3}\cdot\intT\frac{\pg x(\g)-\pg x(\g\!-\!\xi)-(\pg x(\g\!-\!\e)-x(\g\!-\e\!-\!\xi))}{|\xdx|^{\al}}d\xi,
	$$
	$$
	D_2=-\frac{|\e|x_-}{|x_-|^3}\cdot\intT(\pg x(\g\!-\!\e)-\pg x(\g\!-\e\!-\!\xi))
	(g(\g,\xi)-g(\g-\eta,\xi))d\xi ,
	$$
	$$
	D_3=-\frac{|\e|x_-}{|x_-|^3}\cdot(\lambda(\g)\pg x_-+\lambda_-\pg x(\g\!-\!\e))
	$$
	with $g$ given in \eqref{gformula}.
	Decomposing further $D_1$ by $D_1=D_{11}+D_{12}+D_{13}$ with
	$$
	D_{11}=-\frac{|\e|(x_--\pg x(\g)\e)}{|x_-|^3}\cdot\intT\frac{\pg x(\g)\!-\!\pg x(\g\!-\!\xi)}{|\xdx|^{\al}}d\xi,
	$$
	$$
	D_{12}=\frac{|\e|(x_--\pg x(\g\!-\!\e)\e)}{|x_-|^3}\cdot\intT\frac{\pg x(\g\!-\!\e)\!-\!\pg x(\g\!-\!\e\!-\!\xi)}{|\xdx|^{\al}}d\xi,
	$$
	and
	\begin{align*}
	D_{13}=-\frac{|\e|\e}{|x_-|^3}\intT\Big(&\frac{\pg x(\g)\cdot(\pg x(\g)-\pg x(\g\!-\!\xi))}{|\xdx|^{\al}}\\
	&-\frac{\pg x(\g\!-\!\e)\cdot(\pg x(\g\!-\!\e)-x(\g\!-\e\!-\!\xi))}{|\xdx|^{\al}}\Big)d\xi,
	\end{align*}
	we find next the desired bound for the following terms:
	$$
	|D_{11}|+|D_{12}|\leq c_\al F(x)\|\pg^2 x\|_{L^\infty}^2\|F(x)\|_{L^\infty}^{2+\al}\leq
	c_\al F(x)\|F(x)\|_{L^\infty}^{2+\al}\|\pg^2 x\|_{L^2}\|\pg^3x\|_{L^2}.
	$$
	Above we use Gagliardo-Nirenberg in the last step. For the last term in the splitting we take
	\begin{align*}
	D_{13}=-\frac{|\e|\e}{2|x_-|^3}\intT\frac{|\pg x(\g)-\pg x(\g\!-\!\xi)|^2-|\pg x(\g\!-\!\e)-x(\g\!-\e\!-\!\xi)|^2}{|\xdx|^{\al}}d\xi,
	\end{align*}
	to obtain
	\begin{align*}
	D_{13}=-\frac{|\e|\e}{2|x_-|^3}\intT\Big(&\frac{\pg x(\g)-\pg x(\g\!-\!\xi)+\pg x(\g\!-\!\e)-\pg x(\g\!-\e\!-\!\xi)}{|\xdx|^{\al}}\Big)\\
	&\cdot \Big(\pg x(\g)-\pg x(\g\!-\!\xi)-(\pg x(\g\!-\!\e)-\pg x(\g\!-\e\!-\!\xi))\Big)d\xi.
	\end{align*}
	Above identity yields
	$$
	|D_{13}|\leq c_\al F(x)\|\pg^2 x\|_{L^\infty}^2\|F(x)\|_{L^\infty}^{2+\al}\leq c_\al F(x)\|F(x)\|_{L^\infty}^{2+\al}\|\pg^2 x\|_{L^2}\|\pg^3x\|_{L^2},
	$$
	and therefore the same bound for $D_1$:
	$$
	|D_{1}|\leq c_\al F(x)\|F(x)\|_{L^\infty}^{2+\al}\|\pg^2 x\|_{L^2}\|\pg^3x\|_{L^2}.
	$$
	Using the mean value theorem in $g$ we find
	$$
	|g(\g,\xi)-g(\g-\e,\xi)|=c_{\beta}|\e|\frac{|(x(\g_s)\!-\!x(\g_s\!-\!\xi))\cdot (\partial_{\gamma}x(\g_s)\!-\!\partial_{\gamma}x(\g_s\!-\!\xi))|}{|x(\g_s)\!-\!x(\g_s\!-\!\xi)|^{2+\al}}
	$$
	with $\g_s=\g-s\e$, $s\in(0,1)$. Therefore, it is possible to bound as follows
	$$
	|D_{2}|\leq c_\beta F(x)\|\pg^2 x\|_{L^\infty}^2\|F(x)\|_{L^\infty}^{2+\al}\leq c_\al F(x)\|F(x)\|_{L^\infty}^{2+\al}\|\pg^2 x\|_{L^2}\|\pg^3x\|_{L^2} .
	$$
	Finally, for $D_3$ we find
	$$
	|D_{3}|\leq c_\al F(x)\|F(x)\|_{L^\infty}(\|\pg^2 x\|_{L^\infty}\|\lambda\|_{L^\infty}+|\pg x|\|\pg\lambda\|_{L^\infty}).$$
	Bound \eqref{cliag} allows to get
	$$
	|\lambda(\xi,t)|\leq 2\intT\intT\Big|\frac{\pg x_-}{|\pg x|^2}\cdot\frac{\pg^2 \xe}{|x_-|^\al}d\e\Big|d\g+2\intT\Big|\frac{\pg x}{|\pg x|^2}\cdot\intT\frac{\pg x_- x_-\cdot\pg x_-}{|x_-|^{\al+2}}d\e\Big|d\g,
	$$
	for any $\xi$, so that it yields
	$$
	\|\lambda\|_{L^\infty}\leq c_\al \|F(x)\|_{L^\infty}^{2+\al}|\pg x|_{C^\delta}\|\pg^2x\|_{L^2} \leq c_\al \|F(x)\|_{L^\infty}^{2+\al}\|\pg^2 x\|_{L^p}\|\pg^2x\|_{L^2},
	$$
	Repeating the procedure to get bound \eqref{bnlipl} it is possible to obtain
	$$
	|\pg x|\|\pg\lambda\|_{L^\infty}\leq c_\al \|F(x)\|_{L^\infty}^{1+\al}\|\pg^2x\|_{L^p}^2.
	$$
	Plugging the last two estimates in $D_3$ inequality above provides
	$$
	|D_{3}|\leq c_\al F(x)\|F(x)\|^{2+\al}_{L^\infty}(\|F(x)\|_{L^\infty}\|\pg^2 x\|_{L^p}\|\pg^2x\|^{\frac32}_{L^2}\|\pg^3 x\|_{L^2}^{\frac12}+\|\pg^2x\|_{L^p}^2).$$
	Gathering all the $D_m$ estimates we find
	$$
	\partial_t F(x)\leq c_\al F(x)\|F(x)\|_{L^\infty}^{2+\al}
	((1+\|F(x)\|_{L^\infty}\|\pg^2 x\|_{L^p})\|\pg^2x\|_{L^2}\|\pg^3 x\|_{L^2}+\|\pg^2x\|_{L^p}^2)
	$$
	and finally
	\begin{align}
	\begin{split}\label{AC1a2}
	\frac{d}{dt}\|F(x)\|_{L^\infty}\leq& c_\al\|F(x)\|_{L^\infty}^{3+\al}
	(1\!+\!\|F(x)\|_{L^\infty}\|\pg^2 x\|_{L^p})\|\pg^2x\|_{L^2}\|\pg^3 x\|_{L^2}\\
	&\qquad +c_\al\|F(x)\|_{L^\infty}^{3+\al}\|\pg^2x\|_{L^p}^2.
	\end{split}
	\end{align}

	\subsection{Proof of Theorem \ref{LE0a1H2}}
	We gather here the necessary estimates to prove theorem \ref{LE0a1H2}. We consider inequalities \eqref{nL2x}, \eqref{cnh2jf} and \eqref{AC0a1} taking $p<2$ to find
	$$
	\frac{d}{dt}(\|x\|_{H^2}+\|F(x)\|_{L^\infty})\leq \mathcal{P}_2(\|x\|_{H^2}+\|F(x)\|_{L^\infty}),
	$$
	with $\mathcal{P}_2$ a polynomial function. Then, it is possible to integrate the estimate to get an uniform bound for $\|x\|_{H^2}+\|F(x)\|_{L^\infty}$ for a time $T>0$ depending only on $\|x_0\|_{H^2}+\|F(x_0)\|_{L^\infty}$. Through usual approximation arguments (see \cite{G} for more details) those a priori energy estimates provide the existence result. Uniqueness follows using a small modification of argument in \cite{CCG}.
	
	\subsection{Proof ot Theorem \ref{LE1a2H3}}
	Similarly, using \eqref{nL2x}, \eqref{H3evol}, \eqref{AC1a2} together with Sobolev embedding we find
	$$
	\frac{d}{dt}(\|x\|_{H^3}+\|F(x)\|_{L^\infty})\leq \mathcal{P}_3(\|x\|_{H^3}+\|F(x)\|_{L^\infty}),
	$$
	with $\mathcal{P}_3$ a polynomial function. Previous arguments conclude the result.
	
	\subsection{Proof of Theorem \ref{RC}}
	
	In this section we consider $C(T)$ any constant depending on the quantity
	\begin{equation}\label{CT}
	\int_0^T(\|\pg^2 x\|_{L^p}(s)+
	\|F(x)\|_{L^\infty}(s))\|\pg^2 x\|_{L^p}(s)\|F(x)\|_{L^\infty}^{2+\al}(s)ds<\infty.
	\end{equation}
	\\
	
	\emph{Case $0<\al<1$:}\\
	
	Estimates \eqref{nL2x}, \eqref{cnh2jf} together with condition \eqref{CT} yield
	$$
	\sup_{t\in[0,T]}\|x\|_{H^2}(t)\leq \|x_0\|_{H^2}C(T), 
	$$
	using Gronwall's inequality. Integration by parts allows to bound as follows
	$$
	|\pg x\cdot\pg x_t|= \frac1{2\pi}\Big|\intT\intT\frac{\pg^2 x}{|\pg x|^2}\cdot\frac{\pg x_-}{|x_-|^{\al}}d\e d\g\Big|\leq \|\pg^2 x\|_{L^p}\|F(x)\|_{L^\infty}^{3+\al}|\pg x|^2
	$$
	so that
	$$
	\|\pg^2x\|_{L^p}^{-1}\leq c|\pg x|^{-1}\leq c|\pg x_0|^{-1}C(T).
	$$
	Above inequality together with the $H^2$ bound above allow to find in \eqref{AC0a1} that
	\begin{equation}\label{BCarcchord0a1}
	\frac{d}{dt} \|F(x)\|_{L^\infty}\leq C(T) \|F(x)\|_{L^\infty} (\|\pg^2 x\|_{L^p}+\|F(x)\|_{L^\infty})\|\pg^2x\|_{L^p}\|F(x)\|_{L^\infty}^{2+\al}
	\end{equation}
	Gronwall and \eqref{CT} provides
	$$
	\sup_{t\in[0,T]} \|F(x)\|_{L^\infty}(t)\leq  \|F(x_0)\|_{L^\infty}C(T). 
	$$
	It yields existence of solutions up to a time $T$.\\
	
	\emph{Case $1\leq\al<2$:}\\
	
	Proceeding as before, it is possible to find
	$$
	\sup_{t\in[0,T]}\|x\|_{H^2}(t)\leq \|x_0\|_{H^2}C(T).
	$$
	Estimate \eqref{H3evol} then yields
	$$
	\sup_{t\in[0,T]}\|x\|_{H^3}(t)\leq \|x_0\|_{H^3}C(T).
	$$
	Above two estimates provide in \eqref{AC1a2} the following
	\begin{equation*}
	\frac{d}{dt}\|F(x)\|_{L^\infty}\leq C(T)\|F(x)\|_{L^\infty}^{3+\al}
	(\|\pg^2x\|_{L^p}\!+\!\|F(x)\|_{L^\infty}\|\pg^2 x\|_{L^p}\!+\!\|\pg^2x\|_{L^p}^2),
	\end{equation*}
	as $p>2$. Using \eqref{cnh2jf} it is possible to get
	$$
	\|\pg^2 x\|_{L^p}^{-1}\leq c\|\pg^2 x\|_{L^2}^{-1}\leq c \|\pg^2 x_0\|_{L^2}^{-1}C(T),
	$$
	so that \eqref{BCarcchord0a1} follows. Then the arc-chord condition is bounded and the solution exists up to time $T$.

	\subsection*{{\bf Acknowledgments}}
	FG research was partially supported by the grant MTM2014-59488-P (Spain) and by the ERC through the Starting Grant project H2020-EU.1.1.-639227.
	NP was partially supported by the ERC through the Starting Grant project H2020-EU.1.1.-639227.

	\vspace{2cm}
	
	\begin{tabular}{l}
		\textbf{Francisco Gancedo} \\
		{\small Departamento de An\'alisis Matem\'atico \& IMUS}\\
		{\small Universidad de Sevilla} \\ 
		{\small C/ Tarfia, s/n}\\ 
		{\small Campus Reina Mercedes, 41012, Sevilla, Spain} \\ 
		{\small Email: fgancedo@us.es}
	\end{tabular}
	
	\vspace{0.5cm}
	\begin{tabular}{ll}
		\textbf{Neel Patel} \\
		{\textbf{Former Address}}\\
		{\small Departamento de An\'alisis Matem\'atico \& IMUS} & {\small }\\
		{\small Universidad de Sevilla} & {\small }\\
		{\small C/ Tarfia, s/n} & {\small }\\
		{\small Campus Reina Mercedes, 41012, Sevilla, Spain} & {\small}\\
		\textbf{Current Address}\\
		{\small Departament of Mathematics} & {\small }\\
		{\small University of Michigan, Ann Arbor} & {\small }\\
		{\small East Hall, 530 Church Street 48109} & {\small} \\
		{\small Ann Arbor, Michigan, USA 48109} & {\small} \\
		{\small Email: neeljp@umich.edu}
	\end{tabular}

\end{document}